\documentclass{amsart}
\usepackage[nohug]{diagrams}
\usepackage[mathscr]{eucal}
\usepackage{amsfonts}
\usepackage{amsmath}
\usepackage{amsthm}
\usepackage{amssymb}
\usepackage{latexsym}

\usepackage[noadjust]{cite}

\newenvironment{polynomial}
{\par\vspace{\abovedisplayskip}%
\setlength{\leftskip}{\parindent}%
\setlength{\rightskip}{\leftskip}%
\medmuskip=4mu plus 2mu minus 2mu
\binoppenalty=0
\noindent$\displaystyle}
{$\par\vspace{\belowdisplayskip}
}

\usepackage{graphicx}
\usepackage{color}
\usepackage{epsfig}

\newfont{\msam}{msam10}
\newtheorem{theorem}{Theorem}[section]

\newtheorem{theorem_intro}{Theorem}

\newtheorem{proposition}[theorem]{Proposition}
\newtheorem{corollary}[theorem]{Corollary}
\newtheorem{lemma}[theorem]{Lemma}

\theoremstyle{definition}
\newtheorem{definition}[theorem]{Definition}
\newtheorem{remark}[theorem]{Remark}

\newtheorem{example}[theorem]{Example}
\newtheorem{conjecture}[theorem]{Conjecture}
\newtheorem{question}[theorem]{Question} 

\newtheorem{introconjecture}{Conjecture}
\newtheorem{introquestion}{Question}

\let\nc\newcommand

\nc{\la}{\label}

\def\bthm{\begin{theorem}}
\def\ethm{\end{theorem}}
\def\blemma{\begin{lemma}}
\def\elemma{\end{lemma}}
\def\bproof{\begin{proof}}
\def\eproof{\end{proof}}
\def\bprop{\begin{proposition}}
\def\eprop{\end{proposition}}

\def\Z{\mathbb{Z}}
\def\T{\mathbb{T}}

\def\O{\mathcal{O}}

\def\U{\mathcal{U}}

\def\H{\mathscr{H}}

\def\s{\hat{s}}

\def\e{\boldsymbol{\mathrm{e}}}
\def\f{\boldsymbol{\mathrm{f}}}

\def\ult{{\underline{t}}}
\def\loc{\mathrm{loc}}

\def\yy{\hat{y}}
\def\g{\mathfrak{g}}

\def\eae{{\e A \e}}
\def\sl{\mathfrak{sl}}

\def\c{\mathbb{C}}
\def\C{\mathbb{C}}

\nc{\Hom}{{\rm{Hom}}}
\nc{\Ext}{{\rm{Ext}}}
\nc{\htau}{{\bar{ t}}}
\nc{\HOM}{\underline{\rm{Hom}}}
\nc{\EXT}{\underline{\rm{Ext}}}
\nc{\TOR}{\underline{\rm{Tor}}}
\nc{\End}{{\rm{End}}}
\nc{\Map}{{\rm{Map}}}
\nc{\Out}{{\rm{Out}}}
\nc{\GL}{{\rm{GL}}}
\nc{\SL}{{\rm{SL}}}
\nc{\PGL}{{\rm{PGL}}}
\nc{\G}{{\rm{G}}}
\nc{\Rep}{{\rm{Rep}}}
\nc{\ad}{{\rm{ad}}}
\nc{\dlim}{\varinjlim}

\def\HH{\mathrm{H}}
\def\SH{\mathrm{SH}}
\def\BH{H}
\def\S{\mathrm{S}}

\newcommand{\Tr}{{\rm{Tr}}}
\newcommand{\tr}{{\rm{tr}}}

\numberwithin{equation}{section}

\newcommand{\ochar}{{\mathcal O \mathrm{Char}}}
\newcommand{\chr}{{\mathrm{Char}}}
\newcommand{\SU}{ {\mathrm{SU} }}
\newcommand{\orep}{ {\mathcal O \mathrm{Rep}}}

\begin{document}
%
\title[]{Double affine Hecke algebras and generalized Jones polynomials}
\date{\today}
\author{Yuri Berest}
\address{Department of Mathematics, Cornell University, Ithaca, NY 14853-4201, USA}
\email{berest@math.cornell.edu}
%
%
\author{Peter Samuelson}
\address{Department of Mathematics, University of Toronto, Toronto, ON M4Y 1H5, Canada}
\email{psam@math.toronto.ca}

\begin{abstract}
 In this paper, we propose and discuss implications of a general conjecture that there is a canonical action of a rank 1 double affine Hecke algebra on the Kauffman bracket skein module of the complement of a knot $K \subset S^3$. We prove this in a number of nontrivial cases, including all $(2,2p+1)$ torus knots, the figure eight knot, and all 2-bridge knots (when $q=\pm 1$). As the main application of the conjecture, we construct 3-variable polynomial knot invariants that specialize to the classical colored Jones polynomials introduced by Reshetikhin and Turaev in \cite{RT90}. 
 
 We also deduce some new properties of the classical Jones polynomials and prove that these hold for all knots (independently of the conjecture). We furthermore conjecture that the skein module of the unknot is a submodule of the skein module of an arbitrary knot. We confirm this for the same example knots, and we show that this implies the colored Jones polynomials of $K$ satisfy an inhomogeneous recursion relation.
\end{abstract}

\maketitle
\setcounter{tocdepth}{2}
\tableofcontents

\section{Introduction}
In this paper we introduce new connections between the representation theory of double affine Hecke algebras and the colored Jones polynomials of a knot in $S^3$. These connections can be motivated by the following general considerations.

If $K$ is a knot in $S^3$, the most natural algebraic invariant of $K$ is the fundamental group $\pi_1(S^3\setminus K)$ of the complement. This group is not a complete invariant, since complements of different knots can have isomorphic fundamental groups, but it is known that the \emph{peripheral map}
\begin{equation}\label{eq_permap}
\alpha:\pi_1(\partial(S^3\setminus K)) \to \pi_1(S^3\setminus K)
\end{equation}
is a complete knot invariant. More precisely, a theorem of Waldhausen \cite{Wal68} implies that the peripheral map determines the knot complement, and a theorem of Gordon and Luecke \cite{GL89} shows knots in $S^3$ are determined by their complements.

The peripheral map $\alpha$ is quite complicated, so it is natural to simplify (\ref{eq_permap}) by replacing groups with their linear representations. To this end, fix a complex reductive algebraic group $G$, and let $\Rep(\pi,G)$ be the set of group representations from a (discrete) group $\pi$ into $G$. The algebraic structure of $G$ gives the set $\Rep(\pi,G)$ the structure of an affine scheme, and $G$ acts on this scheme by conjugation. We write $\mathrm{Char}(\pi,G):=\Rep(\pi,G)//G$ for the algebro-geometric quotient, and  $\ochar(\pi,G)$ for the corresponding coordinate ring (which is the ring of $G$-invariant regular functions on $\Rep(\pi,G)$). This construction is functorial, so a map of groups $f: \pi \to \pi'$ induces a map of commutative algebras $f_*:\ochar(\pi,G) \to \ochar(\pi',G)$.

The boundary of the complement of (a neighborhood of) a knot is a torus $S^1\times S^1$, and if $\T \subset G$ is a maximal torus, we have the following natural map:
\[
 \T\times \T = \Rep(\Z^2, \T) \to \Rep(\Z^2, G) \twoheadrightarrow  \mathrm{Char}(\Z^2,G)
\]

It is known (see, e.g. \cite{Tha01}) that this induces an isomorphism between $(\T\times \T)/W$ and the component of $\chr(\Z^2,G)$ containing the trivial character (here $W$ is the Weyl group of $G$ acting diagonally on $\T\times\T$). In \cite{Ric79}, Richardson showed that if $G$ is simply connected, then $\chr(\Z^2,G)$ is connected.  Therefore, for simply connected $G$ we have an isomorphism of commutative algebras
\begin{equation}
  \O(\T\times \T)^W \cong \ochar(\Z^2, G) 
\end{equation}

The commutative algebra   $\O(\T \times \T)^W$ admits interesting (noncommutative) deformations which have been studied extensively in recent years. These deformations can be described in the following way. Denote by $P$ and $P^\vee$ the weight and co-weight lattices of the Lie algebra $\g$ of $G$, and write $\c[P\oplus P^\vee]$ for the group algebra of their direct sum. The diagonal action of $W$ on $P\oplus P^\vee$ extends by linearity to $\c[P\oplus P^\vee]$, and we can thus define the semi-direct product $\c[P\oplus P^\vee]\rtimes W$. The algebra $\O(\T \times \T)^W$ is canonically isomorphic to the invariant subalgebra $\c[P\oplus P^\vee]^W$, which embeds (non-unitally) in $\c[P\oplus P^\vee]\rtimes W$ via the map $a \mapsto a\e = \e a \e$ (where $\e = \sum_W w/\lvert W \rvert$ is the symmetrizing idempotent of $W$). The image of this last map is called the \emph{spherical subalgebra}. In this way, we get an identification $\O(\T \times \T)^W \cong \e\left(\C[P\oplus P^\vee]\rtimes W \right) \e$.

Now, for each $G$ as above, Cherednik \cite{Che95} (see also \cite{Che05}) defined the \emph{double affine Hecke algebra} $\HH_{q,t}$ of type $G$ as a two-parameter family of deformations of $\C[P\oplus P^\vee]\rtimes W$, depending on $q,t \in \C^*$. 
The symmetrizing idempotent of $W$ deforms to a distinguished idempotent $\e = \e_{q,t} \in \HH_{q,t}$, and the spherical subalgebra of $\c[P\oplus P^\vee]\rtimes W$ thus deforms to the spherical subalgebra of $\HH_{q,t}$, which we denote $\SH_{q,t} := \e \HH_{q,t}\e$.
In particular, when $q=t=1$ there is a natural algebra isomorphism between the spherical subalgebra $ \SH_{1,1}$ and the commutative algebra $\O(\T\times \T)^W$.

Summarizing the above discussion, for each knot $K \subset S^3$, we have a map of commutative algebras
\begin{equation}\label{repvariatiesmap}
\alpha_*: \SH_{1,1} \to \ochar(\pi_1(S^3\setminus K), G)
\end{equation} 
This leads us to propose the following natural questions:
\begin{introquestion}
Let $K \subset S^3$ be a knot and $N = \ochar(\pi_1(S^3\setminus K),G)$, viewed as an $\SH_{1,1}$-module via the map $\alpha_*$.
\begin{enumerate}\label{mainquestion}
\item Is there a canonical $\SH_{q,t}$-module $N_{q,t}$ which is a deformation of $N$?
\item What knot invariants can be extracted from $N_{q,t}$?
\end{enumerate}
\end{introquestion}

We remark that one's initial inclination might be to deform $\alpha_*$ as an algebra homomorphism, but this is too restrictive. For generic parameters, it is known that $\SH_{q,t}$ is a simple algebra, and $\ochar(\pi_1(S^3\setminus K), G)$ is `small' (in particular, the image of $\chr(\pi_1(S^3\setminus K),G)$ inside $\chr(\Z^2,G)$ is Lagrangian). Therefore, a deformation of $\alpha_*$ as an algebra homomorphism would necessarily have a nontrivial kernel, which is impossible by the simplicity of $\SH_{q,t}$.

To the best of our knowledge, this question has only been raised for $q$-deformations (i.e. for $t=1$) and has only been answered when $\g = \sl_2$ and $t=1$. (There is partial progress for $t=1$ and $\g = \sl_n$ in \cite{Sik05}, see also \cite{CKM12}.) When $\g = \sl_2$, we can identify $P \cong P^\vee \cong \Z$ and $W \cong \Z_2$, so that $\C[ P \oplus P^\vee] \rtimes \Z_2 \cong \C[X^{\pm 1},Y^{\pm 1}]\rtimes \Z_2$, where $\Z_2$ acts by simultaneously inverting $X$ and $Y$. Then $\Z_2$ acts in the same way on the \emph{quantum torus} $A_q$, which is the deformation of the algebra $\C[X^{\pm 1},Y^{\pm 1}]$ with $X,Y$ satisfying the relation $XY=q^2YX$.
We then have isomorphisms $\HH_{q,1} \cong A_q\rtimes \Z_2$ and $\SH_{q,1} \cong A_q^{\Z_2}$ (where $\e = (1+s)/2$ is the symmetrizing idempotent of $\Z_2$).

The connection to representation varieties comes from the \emph{Kauffman bracket skein module}. This is a topologically defined vector space $K_q(M)$ associated to an oriented 3-manifold $M$ and the parameter $q \in \C^*$ that has three important properties: 
\begin{enumerate}
 \item For a surface $F$, the vector space $K_q(F \times [0,1])$ is an algebra (typically noncommutative).
 \item If $\partial M = F$, then $K_q(M)$ is a module over $K_q(F\times [0,1])$. 
 \item If $q=\pm 1$, then $K_{q=\pm 1}(M)$ is a commutative algebra (for any $M$).
\end{enumerate}
 In \cite{PS00}, Przytycki and Sikora showed that $K_{q=-1}(M)$ is naturally isomorphic to $\ochar(\pi_1(M), \SL_2(\C))$ (see also \cite{Bul97}). By a theorem of Frohman and Gelca in \cite{FG00}, $K_q((S^1)^2 \times [0,1])$ is isomorphic to the algebra $A_q^{\Z_2}$. Combining these two theorems shows that $N_{q,t=1} := K_q(S^3\setminus K)$ is an $A_q^{\Z_2}$-module, which gives a positive answer to the first part of Question \ref{mainquestion} (when $\g = \sl_2$ and $t=1$).

At this point, we pause to remark that the knot invariant $K_q(S^3\setminus K)$ is different from many other knot invariants in a fundamental way. Roughly, many knot invariants are defined combinatorially, in the sense that they assign certain data to each crossing in a diagram of $K$ and then combine this data to produce an invariant that does not depend on the choice of diagram. In contrast, the definition of the module $K_q(S^3\setminus K)$ depends on the global topology of the complement $S^3\setminus K$, and this makes it difficult to prove general statements about $K_q(S^3\setminus K)$. In particular, one may ask what facts are known about $K_q(S^3\setminus K)$ for all $q$ and for all knots $K$, and to the best of our knowledge, there are two such statements: $K_q(S^3\setminus K)$ is a module over $A_q^{\Z_2}$, and $K_q(S^3\setminus K)$ determines the colored Jones polynomials of $K$ (see below). However, we believe that the calculations in Section \ref{sec_qeqqskeinmodules} give some evidence that these modules are not as intractable as they might seem. In examples, these modules are $q$-analogues of smooth holonomic $D$-modules, i.e. vector bundles with flat connections over $\C^*$. (Precisely, they are $\Z\rtimes \Z_2$-equivariant vector bundles over $\C^*$.)

The skein module $K_q(S^3\setminus K)$ is known to be closely related to other `quantum' knot invariants. In particular, in \cite{RT90} Reshetikhin and Turaev defined a polynomial invariant $J_V(\g, K; q) \in \C[q^{\pm 1}]$ for each finite dimensional representation $V$ of the quantum group $\U_q(\g)$. 
If $\g = \sl_2$ and $V$ is the defining representation of $\U_q(\sl_2)$, then $J_V(\sl_2,K; q)$ is the famous Jones polynomial, and if $V_n$ is the $n$-dimensional irreducible representation of $\U_q(\sl_2)$, then the polynomials $J_n(K;q) := J_{V_n}(\sl_2,K;q)$ are called the \emph{colored Jones polynomials}. 
The second part of Question \ref{mainquestion} is answered by a theorem of Kirby and Melvin. The embedding $S^3\setminus K \hookrightarrow S^3$ induces a $\C[q^{\pm 1}]$-linear map $\epsilon: K_q(S^3\setminus K) \to K_q(S^3) = \C[q^{\pm 1}]$, and it was shown in \cite{KM91} that 
\begin{equation}\label{equation_introjpoly}
J_n(K;q) = \epsilon(S_{n-1}(L)\cdot \varnothing) 
\end{equation}
where $S_{n-1}(L)$ is the $(n-1)^{\mathrm{st}}$ Chebyshev polynomial evaluated at the longitude $L$ of $K$ and applied to the empty link $\varnothing$.
 
Our main goal in this paper is to introduce the Hecke parameter $t$ into this story. In fact, in the rank 1 case, more deformation parameters are available: there is a family of algebras $\H_{q,\ult}$ depending on a parameter $q \in \C^*$ and four additional parameters $\ult \in (\C^*)^4$. This family is a nontrivial deformation of $A_q\rtimes \Z_2$ (see e.g. \cite{Sah99}), which is actually the universal (i.e. ``maximum possible'') deformation (see \cite{Obl04}). The algebra $\H_{q,\ult}$ is the double affine Hecke algebra associated to the (nonreduced) root system of type $C^{\vee} C_1$, and it was introduced by Sahi \cite{Sah99} (see also \cite{NS04} and \cite{Sto03}) to study the Askey-Wilson polynomials, which are generalizations of the famous Macdonald polynomials. As an abstract algebra, $\H_{q,\ult}$ is generated by elements $T_0, T_1, T_0^\vee, T_1^\vee$ subject to the relations
 \begin{align*}
 (T_0-t_1)(T_0+t_1^{-1}) &= 0\\ (T_0^\vee-t_2)(T_0^\vee+t_2^{-1}) &= 0 \\
 (T_1-t_3)(T_1+t_3^{-1}) &= 0\\ (T_1^\vee-t_4)(T_1^\vee+t_4^{-1}) &= 0\\
 T_1^\vee T_1T_0 T_0^\vee &= q
\end{align*}
 For $\g = \sl_2$, Cherednik's double affine Hecke algebra $\HH_{q,t}$ is isomorphic to $\H_{q,1,1,t^{-1},1}$ (see Remark \ref{remark_a1cciso}).
 
 If $q \in \C^*$ is not a root of unity, then $A_q\rtimes \Z_2$ and $A_q^{\Z_2}$ are \emph{Morita equivalent} algebras; in other words, the categories of modules over $A_q\rtimes \Z_2$ and $A_q^{\Z_2}$ are equivalent via the projection functor $N \mapsto \e N$. This implies that there is a unique  $A_q\rtimes \Z_2$-module $\hat K_q(S^3\setminus K)$ such that $\e \hat K_q(S^3\setminus K)$ and $K_q(S^3\setminus K)$ are isomorphic $A_q^{\Z_2}$-modules. Explicitly, 
 \[\hat K_q(S^3\setminus K) := A_q \otimes_{A_q^{\Z_2}} K_q(S^3\setminus K)\] 
 We call $\hat K_q(S^3\setminus K)$ the \emph{nonsymmetric skein module} - by definition it is a module over $A_q\rtimes \Z_2 = \H_{q,(1,1,1,1)}$. We can now reformulate Question \ref{mainquestion} as follows:

\begin{introquestion}\label{mainquestion2}
 Is there a \emph{canonical} deformation of the $\H_{q,(1,1,1,1)}$-module $\hat K_q(S^3\setminus K)$ to a family of modules over $\H_{q,\ult}$?
\end{introquestion}

When $\ult = (t_1,t_2,1,1)$, we (conjecturally) give a positive answer to this question using an approach
inspired by a construction of shift functors for rational double affine Hecke algebras developed in \cite{BC11} (see also \cite{BS12}). Let $D_q$ be the localization of the algebra $A_q\rtimes \Z_2$ obtained by inverting all nonzero polynomials in $X$. For every $\ult \in (\C^*)^4$, there is a natural embedding of $\H_{q,\ult}$ into $D_q$:
\begin{equation}\label{equation_introdlembedding}
 \Theta: \H_{q,\ult} \hookrightarrow D_q,
\end{equation}
whose image is the subalgebra generated by $X$, $X^{-1}$, and the following operators (see \cite{Sah99}, \cite{NS04}):
\[
 T_0 = t_1 sY - \frac{q \bar t_1 X  + \bar t_2}{q^{-1}X^{-1}-qX}(1-sY),\quad \quad
 T_1 = t_3s + \frac{\bar t_3 X^{-1}+\bar t_4 }{X^{-1}-X}(1-s),
\]
where $\bar t_i = t_i - t_i^{-1}$. (The operator $T_1$ is usually called the Demazure-Lusztig operator since it arises in the representation theory of affine Hecke algebras, see \cite{Lus89}.) If $\hat K_q^{\loc}(S^3\setminus K)$ is the localization of $\hat K_q(S^3\setminus K)$ at nonzero polynomials in $X$, then $\Theta$ gives $\hat K^\loc_q(S^3\setminus K)$ the structure of a module over $\H_{q,\ult}$. We can now state our main conjecture: 
\begin{introconjecture}\label{mainconjecture}
 For all $K$, the action of $\H_{q,t_1,t_2,1,1}$ on $\hat K_q^{loc}(S^3\setminus K)$ preserves the subspace\footnote{Technically, part of this conjecture is that the localization map $\hat K_q(S^3\setminus K) \to \hat K_q^\loc(S^3\setminus K)$ is injective.} $\hat K_q(S^3\setminus K)$.
\end{introconjecture}
We remark that this conjecture implies that $\hat K_q(S^3\setminus K)$ is naturally a module over the 3-parameter algebra $\H_{q,t_1,t_2}$, even though only $1$ of these parameters appears in the definition of the skein module. It is natural to ask whether Conjecture \ref{mainconjecture} can be extended to the full double affine Hecke algebra $\H_{q,t_1,t_2,t_3,t_4}$ depending on all five parameters. The simplest example shows that this is not possible: if $t_3 \not= 1$ or $t_4 \not= 1$, the operator $T_1$ does not preserve the skein module of the unknot. However, we believe that this is the \emph{only} obstruction to a canonical extension the action of $\H_{q,t_1,t_2}$ on $\hat K_q(S^3\setminus K)$ to all five parameters\footnote{In examples, 5-parameter deformations of $\hat K_q(S^3\setminus K)$ can be produced `by hand' (see Section \ref{sec_5paramdeformations}), but these are not canonical in general, unlike the deformations arising from Conjecture \ref{mainconjecture}.}. More precisely, we have the following two conjectures:
\begin{introconjecture}\label{conj_unknotsubmodule}
 For all knots, there is an embedding of left $A_q^{\Z_2}$-modules 
 \[
  K_q(S^3 \setminus \mathrm{unknot}) \hookrightarrow K_q(S^3\setminus K)
 \]
\end{introconjecture}

\begin{introconjecture}\label{conj_5param}
 Assume Conjecture \ref{conj_unknotsubmodule} and let $\bar K_q(S^3\setminus K)$ be the quotient of $ \hat K_q(S^3\setminus K)$ by the image of the skein module of the unknot. Then the action of $\H_{q,t_1,t_2,t_3,t_4}$ on $\bar K_q^\loc(S^3\setminus K)$ preserves the subspace $\bar K_q(S^3\setminus K) \subset \bar K_q^\loc(S^3\setminus K)$.
\end{introconjecture}

We provide some evidence for these conjectures with the following theorem (see Theorem \ref{thm_mainconjecture}, Theorem \ref{thm_conjecture2}, and Corollary \ref{cor_qeq1}):
\begin{theorem_intro}\label{maintheorem}
 Conjectures \ref{mainconjecture}, \ref{conj_unknotsubmodule}, and \ref{conj_5param} hold in the following cases:
 \begin{enumerate}
  \item when $K \subset S^3$ is the unknot, a $(2,2p+1)$ torus knot, or the figure eight knot,
  \item when $q=-1$ and $K$ is any 2-bridge knot.
 \end{enumerate}

\end{theorem_intro}

Strictly speaking, for 2-bridge knots we prove symmetric versions of Conjectures \ref{mainconjecture}, \ref{conj_unknotsubmodule}, and \ref{conj_5param} (see Section \ref{sec_qeqq_2bridge}). We also show that in this case the restriction $q=-1$ is almost unnecessary: more precisely, we conjecturally identify the $q=-1$ limit of the nonsymmetric skein module for 2-bridge knots (see Conjecture \ref{conj_q1skeinmodule}), and we show that if this identification is correct, Conjecture \ref{mainconjecture} holds for all 2-bridge knots with no restriction on $q$ (see Corollary \ref{corollary_qeqq}).

We now provide some remarks about these conjectures. First, Conjecture 1 is equivalent to the statement that the operator $(1-q^2X^2)^{-1}(1-s\yy)$ preserves the nonsymmetric skein module. In Lemma \ref{lemma_twisteddunkl} we show that $\HH_{q,t}$ can be embedded in $D_q$ using this operator. In particular, Conjecture \ref{mainconjecture} implies that $\HH_{q,t}$ acts on the nonsymmetric skein module of a knot complement. We therefore expect that Question \ref{mainquestion} has a positive answer for $\sl_n$, and for reductive $\g$ when $q=1$, at least up to a similar twist by an automorphism.

Second, Conjecture \ref{mainconjecture} can also be stated directly in terms of skein modules without using Morita theory - see Remark \ref{remark_eUe}. When specialized to $q=-1$, this interpretation implies that the rational function
\[
 F(\rho) =  \frac{\Tr(\rho(ml)) - \Tr(\rho(ml^{-1}))}{\Tr(\rho(m))^2 - 4}
\]
on $\mathrm{Char}(T^2)$ restricts to a regular function on the image of the map $\mathrm{Char}(S^3\setminus K) \to \mathrm{Char}(T^2)$. (Here $m$ and $l$ are elements in $\pi_1(T^2)$ which represent the meridian and longitude of the knot.) 
This geometric statement should be interpreted with care because the schemes involved are singular at the poles of $F$. In particular, $F$ is ``set-theoretically a regular function'' on $\mathrm{Char}(T^2)$.

We next provide two applications of Conjecture \ref{mainconjecture}. First, the existence of a natural $\H_{q,t_1,t_2}$-module structure on the non-symmetric skein module $\hat K_q(S^3\setminus K)$ allows us to define 3-variable polynomial knot invariants $J_{n}(K; q,t_1,t_2) \in \C[q^{\pm 1},t_1^{\pm 1}, t_2^{\pm 1}]$ that specialize to the colored Jones polynomials when $t_1=t_2=1$. To this end, we modify the Kirby-Melvin formula (\ref{equation_introjpoly}):
\begin{equation}\label{equation_intro3varpoly}
 J_n(K; q, t_1,t_2) := \epsilon(S_{n-1}(L_{t_1,t_2})\cdot \varnothing)
\end{equation}
where we replaced the longitude $L$ (viewed as an operator on $K_q(S^3\setminus K)$) by its natural $(t_1,t_2)$-deformation $L_{t_1,t_2} := T_1T_0 + T_0^{-1}T_1^{-1}$, which is called the \emph{Askey-Wilson operator} (cf. \cite{AW85} and \cite[Prop. 5.8]{NS04}). By definition, the $L_{1,1} = L$, and this combined with (\ref{equation_introjpoly}) shows that $J_n(K;q,t_1,t_2)$ specializes to the classical colored Jones polynomial. The Askey-Wilson operator has a denominator involving the meridian - the key point of Conjecture \ref{mainconjecture} is that these denominators cancel with the structure constants of the skein module of the knot complement. In Proposition \ref{prop_mirror} we show that if $\bar K$ is the mirror of the knot $K$, then $J_n(\bar K; q, t_1,t_2) = J_n(K; q^{-1}, t_1^{-1},t_2^{-1})$, which generalizes the well-known symmetry for the classical colored Jones polynomials. This provides some evidence that definition (\ref{equation_intro3varpoly}) is natural.

We remark that a strong form of the AJ conjecture (see, e.g. Conjecture 3 of \cite{Le06}) states that the submodule of $K_q(S^3\setminus K)$ generated by the empty link is determined by the colored Jones polynomials $J_n(K; q)$. If this statement is true for a knot $K$ \emph{and} a closed formula for the polynomials $J_n(K; q)$ is available, then one can in principle compute the 3-variable polynomials $J_n(K; q,t_1,t_2)$ without skein theory. (However, this calculation will be complicated even in simple examples.) This leads to the following questions:
\begin{introquestion}
 Is there an algorithm for computing $J_m(K; q,t_1,t_2)$ for a fixed $m$ that does not require computing the skein module $K_q(S^3\setminus K)$ or the colored Jones polynomials $J_n(K; q)$ for all $n$? Is there an interpretation of $J_m(q,t_1,t_2)$ in terms of representation theory of the quantum group $\U_q(\sl_2)$?
\end{introquestion}

One may also ask whether there is a purely topological construction of these deformations of skein modules, or of the corresponding polynomial knot invariants. This question (and its relation to Cherednik's 2-variable polynomials for torus knots \cite{Che11}) will be discussed in \cite{Sam14}.

As a second application, we deduce from Conjecture \ref{mainconjecture} some algebraic properties of the classical (colored) Jones polynomials which, to the best of our knowledge, have not appeared in the earlier literature. Namely, we prove the following (see Theorem \ref{thm_divisibility}):
\begin{theorem_intro}\label{thm_introdivisibility}
If Conjecture \ref{mainconjecture} holds for a knot $K$, then the rational function
\[
P_{j}(K;n;q) := \frac{(q^2-1)\left[ J(n+j) + J(n-1-j)\right]}{q^{4n-2}-1}
\]
is actually a Laurent polynomial. Furthermore, there are rational functions $a_{i,l}(K,j;q) \in \C(q)$, not depending on $n$, such that
\[
 P_{j}(K;n;q) = \sum_{k,l} a_{k,l}(K,j;q) q^{2nk}J(n+l)
\]
\end{theorem_intro}

In this corollary, we have used the notation $J(n) := J_{n}(K; q)$ and the convention $J(-n) = -J(n)$ (for our choice of normalization of $J(n)$ see Remark \ref{remark_signconvention}). We also remark that the proof of Theorem \ref{thm_introdivisibility} suggests that the sequence $J(n)$ of colored Jones polynomials satisfies a recursion relation governed by the algebra $\H_{q,\ult}$ (see Question \ref{question_Hrecursion}). We hope to address this relation in later work.

As further evidence for Conjecture \ref{mainconjecture}, we follow a suggestion of Garoufalidis and use Habiro's cyclotomic expansion of the colored Jones polynomials to prove the following theorem (see Theorem \ref{theorem_divisibilityfromhabiro}):
\begin{theorem_intro}\label{theorem_introdivisibilityfromhabiro}
 The rational function $P_j(K;n;q) \in \C(q)$ is a Laurent polynomial for all knots $K\subset S^3$.
\end{theorem_intro}
The result of Theorem \ref{theorem_introdivisibilityfromhabiro} seems to be new; however, one of its implications (namely, the numerator of $P_j(n;q)$ is zero when $q = -e^{i\pi/(2n+1)}$) also follows from Proposition 2.1 of \cite{CM11}. We also note that this theorem can be viewed as a congruence relation, and it is remarkably similar to several  congruence relations for knot polynomials conjectured in \cite{CLPZ14}. In fact, the proof of Theorem \ref{theorem_introdivisibilityfromhabiro} extends almost verbatim to a proof of \cite[Conj. 1.6]{CLPZ14}. We provide the details in Section \ref{sec_congruences} (see Theorem \ref{thm_clpzconj}).

We confirm Conjecture \ref{conj_unknotsubmodule} for the figure eight and for all $(2,2p+1)$ torus knots in Theorem \ref{thm_conjecture2}. This statement has a conceptual explanation - it can be viewed as a quantization of the fact that $L-1$ always divides the $A$-polynomial of the knot $K$. (See Remarks \ref{remark_lm1dividesA} and \ref{remark_unknotsubmodule} for further explanation.) We also point out that even in the simplest examples, this embedding is not obvious: in particular, the empty link in the skein module of the unknot is sent to a nontrivial element in the skein module of $S^3 \setminus K$. We expect that there is a topological interpretation of this embedding, but we will not address this here. As an application of Conjecture \ref{conj_unknotsubmodule}, we prove the following theorem (see Theorem \ref{theorem_inhomogenous}):

\begin{theorem_intro}\label{thm_introinhomogeneous}
 Suppose $f: K_q(S^3\setminus \mathrm{unknot}) \to K_q(S^3\setminus K)$ is an $A_q^{\Z_2}$-module map such that $\mathrm{im}(f) \subset K_q(T^2)\cdot \varnothing$. Then there exist 2-variable Laurent polynomials $c_k(-,-)$ and a sequence
 \[
  P(n) := \sum_k c_k(q, q^{2n}) J(n+k)
 \]
such that $P(n) = P(0)$, for all $n$.
\end{theorem_intro}

We now summarize the contents of the paper. In Section \ref{preliminaries} we give an introduction to double affine Hecke algebras and Kauffman bracket skein modules. 
In Section \ref{sec_qequals1} we prove Conjectures \ref{mainconjecture}, \ref{conj_unknotsubmodule} and \ref{conj_5param} for 2-bridge knots (when $q=-1$). In Section \ref{sec_qeqqskeinmodules} we use computations by Gelca and Sain to give complete descriptions of the skein modules of the $(2,2p+1)$ torus knots and the figure eight knot, and we prove Conjecture \ref{conj_unknotsubmodule} for these knots. In Section \ref{sec_divisibility}, we prove Theorems \ref{thm_introdivisibility} and \ref{theorem_introdivisibilityfromhabiro}, which involve divisibility properties for colored Jones polynomials. We also prove Theorem \ref{thm_introinhomogeneous}, which involves inhomogeneous recursion relations for colored Jones polynomials. In Section \ref{sec_ccdeformations} we prove Conjectures \ref{mainconjecture} and \ref{conj_5param} for $(2,2p+1)$ torus knots and the figure eight knot.  In Section \ref{sec_5paramdeformations}, we construct non-canonical deformations of $\hat K_q(S^3 \setminus \mathrm{trefoil})$ to a module over $\H_{q,\ult}$ for arbitrary $\ult \in (\C^*)^4$. In an Appendix we include example computations of 3 variable polynomials specializing to colored Jones polynomials of the trefoil, the $(5,2)$ torus knot, and the figure eight knot.

\noindent \textbf{Acknowledgements:}
We would like to thank I. Cherednik for guidance with references and several helpful comments, S. Garoufalidis for suggesting the approach to the proof of Theorem \ref{theorem_divisibilityfromhabiro}, and R. Gelca for kindly allowing the use of his figures. We would also like to thank D. Bar-Natan, O. Chalykh, B. Cooper, P. Etingof, J. Kamnitzer, T. L\^e, J. Marche, A. Marshall, G. Muller, A. Oblomkov, M. Pabiniak, D. Thurston, and B. Webster for enlightening conversations. The second author is grateful to the users of the website MathOverflow who have provided several helpful answers (see, e.g. \cite{35687}).

The work of the first author was partially supported by NSF grant DMS 09-01570.

\section{Preliminaries}\label{preliminaries}
In this section we provide the background necessary for the rest of the paper by discussing definitions and basic properties of the Kauffman bracket skein module and double affine Hecke algebras.

\subsection{Knot groups and their character varieties}\label{subsec_knotgroups}

Recall that two maps $f,g:M \to N$ of manifolds are \emph{ambiently isotopic} if they are in the same orbit of the identity component of the diffeomorphism group of $N$. This is an equivalence relation, and a \emph{knot} in a 3-manifold $M$ is the equivalence class of a smooth embedding $K: S^1 \hookrightarrow M$. If $N_K \subset M$ is an (open) tubular neighborhood of $K$, then the complement $M \setminus N_K$ has a torus boundary.

For an oriented knot $K \subset S^3$ there is a canonical identification $T = S^1 \times S^1 \to \partial(S^3 \setminus K)$. More precisely, let  $N_K \subset S^3$ be a closed tubular neighborhood of $K$, and let $N_c$ be the closure of its complement. Then the following lemma provides a unique (up to isotopy) identification of $N_K \cap N_c$ with $S^1\times S^1$ (see \cite[Thm. 3.1]{BZ03}):
\begin{lemma}
 There is a unique (up to isotopy) pair of simple loops  (the meridian $m$ and longitude $l$)  in $T$ subject to the conditions
\begin{enumerate}\label{lemma_meridianlongitude}
\item  $m$ is nullhomotopic in $N_K$,
\item $l$ is nullhomotopic in $N_c$,
\item $m,l$ intersect once in $T$,
\item in $S^3$, the linking numbers $(m,K)$ and $(l,K)$ are 1 and 0, respectively.
\end{enumerate}
\end{lemma}

Therefore, to an oriented knot $K \subset S^3$ one can associate the data $(\pi_1(S^3 \setminus K), m, l)$, where $m, l \in \pi_1(S^3 \setminus K)$ are the elements corresponding to the meridian and longitude defined above. Since the meridian and longitude are well-defined up to (base-point free) isotopy, the elements $m,l$ are well-defined up to inner automorphism. The following theorem (see \cite[Thm. 3.15]{BZ03}) shows that this data is a complete invariant of the knot.

\begin{theorem}[{\cite{Wal68}}]
 Two knots $K, K' \subset S^3$ are ambiently isotopic if and only if there is an isomorphism $\phi: \pi_1(S^3 \setminus K) \to \pi_1(S^3 \setminus K')$ that satisfies $\phi(m) = m'$ and $\phi(l) = l'$.
\end{theorem}

\subsubsection{Character varieties and the $A$-polynomial of a knot}\label{section_charvarieties}
If $\pi$ is a finitely generated (discrete) group and $G$ is an algebraic group, the set $\mathrm{Rep}(\pi, G) := \Hom(\pi, G)$ has a natural affine scheme structure. Informally, one way to realize this structure is to pick generators $g_1, \ldots, g_n \in \pi$ for $\pi$, so that a representation $\rho:\pi \to G$ is completely determined by the images of the $g_i$. This realizes $\mathrm{Rep}(\pi, G)$ as a subscheme of $G^n$, where the ideal defining this subscheme is defined by the relations between the $g_i$. It is well known that $\mathrm{Rep}( -, G)$ is functorial with respect to group homomorphisms $f: \pi \to \pi'$. In particular, the scheme structure on $\mathrm{Rep}(\pi, G)$ is independent of the choices made (see, e.g. \cite{LM85}).

There is a natural action of $G$ on $\mathrm{Rep}(\pi,G)$ (by conjugation), and this induces an action of $G$ on the corresponding commutative algebra $\O(\mathrm{Rep}(\pi,G))$. We denote the subalgebra of invariant functions by
\[
 \ochar(\pi,G) := \O(\mathrm{Rep}(\pi,G))^G
\]
The \emph{character variety} is the spectrum of this algebra: $\chr(\pi,G) := \mathrm{Spec}(\ochar(\pi,G))$. This scheme parameterizes closed $G$ orbits on $\mathrm{Rep}(\pi,G)$. 

We now assume $\pi = \pi_1(M)$ for a manifold $M$ and specialize to $G = \SL_2(\C)$. To shorten notation, we will write $\ochar(M) := \ochar(\pi_1(M), \SL_2(\C))$, etc. If $K \subset S^3$ is a knot and $M := S^3 \setminus K$ is the knot complement, then the inclusion $T^2 = \partial M \subset M$ induces a map of algebras $\iota: \ochar(T^2) \to \ochar(M)$.

We may identify the algebra $\ochar(T^2)$ with $\O(\C^* \times \C^*)^{\Z_2}$, where the generator of $\Z_2$ acts via $(a,b) \mapsto (a^{-1}, b^{-1})$. The algebra $\O(\C^*\times \C^*)^{\Z_2}$ is generated by the functions $x(a,b) = a+a^{-1}$, $y(a,b) = b+b^{-1}$, and $z(a,b) = ab+(ab)^{-1}$. Under the identification with $\ochar(T^2)$, these functions correspond to $(A,B) \mapsto \tr(A)$, $(A,B) \mapsto \tr(B)$, and $(A,B) \mapsto \tr(AB)$, respectively (c.f. Theorem \ref{thm_qeq1repvar}).

We now recall the definition of the $A$-polynomial, which was originally introduced in \cite{CCG94}. If $M$ is a knot complement, we have the diagram
\[
 \C^* \times \C^* \twoheadrightarrow (\C^* \times \C^*) / \Z_2 \cong \chr(T^2) \leftarrow \chr(M)
\]
We let $X_M \subset \C^*\times \C^*$ be the union of the 1-dimensional components of the preimage of the Zariski closure of the image of $\chr(M)$. A theorem of Thurston says that $X_M$ is nonempty, which allows the following definition.
\begin{definition}
 The $A$-polynomial is the polynomial $A(m,l) \in \C[m^{\pm 1},l^{\pm 1}]$ that defines the curve $X_M$.
\end{definition}
(Here we have used Lemma \ref{lemma_meridianlongitude} to pick generators $m,l$ for $\O(\C^*\times \C^*)$.)

\begin{remark}\label{remark_lm1dividesA}
 The abelianization of the fundamental group of a knot complement is isomorphic to $\Z$,  
  and it is well known that the set of representations factoring through the abelianization map is a component of the character variety. This implies that $l-1$ always divides the $A$-polynomial. If $K$ is the unknot, then $\pi_1(S^3\setminus K) = \Z$, which implies that the $A$-polynomial of the unknot divides the $A$-polynomial of an arbitrary knot. If we write $A_K$ and $A_U$ for the $A$-polynomials of a knot $K$ and the unknot $U$, then $A_K = B_K A_U$, and we have a map of $\C[m^{\pm 1},l^{\pm 1}]$-modules (which is \emph{not} a map of algebras):
  \begin{equation}\label{equation_unknotembeddingqeq1}
  \phi: \C[m^{\pm 1},l^{\pm 1}] / A_U \to \C[m^{\pm 1},l^{\pm 1}] / A_K,\quad \quad f \mapsto B_K f
  \end{equation}
\end{remark}

\subsection{Kauffman bracket skein modules}\label{kbsmsection}
A \emph{framed link} in an oriented 3-manifold $M$ is an 
embedding of a disjoint union of annuli $S^1 \times [0,1]$ into $M$. (The framing refers to the $[0,1]$ factor and is a technical detail that will be suppressed when possible.) We will consider framed links to be equivalent if they are ambiently isotopic. In what follows, the letter $q$ will denote either an element of $\C$ or the generator of the ring $\C[q,q^{-1}]$ (we will specify which when it matters and when it is not clear from context).

Let  $\mathscr L(M)$ be the vector space 
spanned by the set of ambient isotopy classes of framed unoriented links in $M$ (including the empty link). Let $\mathscr L'(M)$ be the smallest subspace of $\mathscr L(M)$ containing the skein expressions 
$L_+ - qL_0 - q^{-1}L_\infty$ and $L \sqcup \bigcirc + (q^2+q^{-2})L$. The links $L_+$, $L_0$, and $L_\infty$ are identical outside of a small 3-ball (embedded as an oriented manifold), and inside the 3-ball 
they appear as in Figure \ref{kbsm}. (All pictures drawn in this paper will have blackboard framing. In other words, a line on the page represents a strip $[0,1]\times [0,1]$ in a tubular neighborhood of the page, 
and the strip is always perpendicular to the paper (the intersection with the paper is $[0,1]\times\{0\}$).)

\begin{figure}
\begin{center}
\input{kbsmrelation.pstex_t}
\caption{Kauffman bracket skein relations}\label{kbsm}
\end{center}
\end{figure}

\begin{definition}[\cite{Prz91}]
The \emph{Kauffman bracket skein module} is the vector space $K_q(M) := \mathscr L / \mathscr L'$. It contains a canonical element $\varnothing \in K_q(M)$ corresponding to the empty link.
\end{definition}

\begin{remark}
 To shorten the notation, if $M = F \times [0,1]$ for a surface $F$, we will often write $K_q(F)$ for the skein module $K_q(F\times [0,1])$.
\end{remark}

\begin{example}\label{skeins3}
One original motivation for defining $K_q(M)$ is the isomorphism
\[\C[q,q^{-1}] \stackrel \sim \to K_q(S^3), \quad 1 \mapsto \varnothing \]
Kauffman proved that this map is an isomorphism and that the inverse image of a link is the Jones polynomial of the link. The map is surjective because the skein relations allow one to remove all crossings and loops in a diagram of any link, but
showing it is injective is (essentially) equivalent to showing the Jones polynomial of a link is well-defined, which is a non-trivial theorem.
\end{example}

In general $K_q(M)$ is just a vector space - however, if $M$ has extra structure, then $K_q(M)$ also has extra structure. In particular,
\begin{enumerate}
\item If $M = F \times [0,1]$ for some surface $F$, then $K_q(M)$ is an algebra, where the multiplication is given by ``stacking links." 
\item If $M$ is a manifold with boundary, then $K_q(M)$ is a module over $K_q(\partial M)$. The multiplication is given by ``pushing links from the boundary into the manifold.'' 
\item\label{embedding} An oriented embedding $M \hookrightarrow N$ of 3-manifolds induces a linear map $K_q(M) \to K_q(N)$. Therefore, $K_q(-)$ can be considered as a functor on the category whose objects are oriented 3-dimensional manifolds and whose morphisms are oriented embeddings.\footnote{To be pedantic, $K_q(-)$ is functorial with respect to maps $M \to N$ that are oriented embeddings when restricted to the interior of $M$. In particular, if we identify a surface $F$ with a boundary component of $M$ and $N$, then the gluing map $M \sqcup N \to M \sqcup_F N$ induces a linear map $K_q(M) \otimes_\C K_q(N) \to K_q(M\sqcup_F N)$.}
\item If $q=\pm 1$, then $K_q(M)$ is a commutative algebra (for any oriented 3-manifold $M$). The multiplication is given by ``disjoint union of links,'' which makes sense because when $q=\pm 1$, the skein relations allow strands to `pass through' each other.
\end{enumerate}

\begin{remark}
 In fact, the first two properties are a special case of the third. For example, there is an obvious map $F\times [0,1] \sqcup F\times [0,1] \to F\times [0,1]$, and the product structure of $K_q(F \times [0,1])$ comes from the application of the functor $K_q(-)$ to this map.
\end{remark}

\begin{example}
 Let $M = (S^1 \times [0,1]) \times [0,1]$ be the solid torus. If $u$ is the nontrivial loop, then the map $\C[u] \to K_q((S^1\times[0,1])\times [0,1])$ sending $u^n$ to $n$ parallel copies of $u$ is surjective (because all crossings and trivial loops can be removed using the skein relations). This is clearly an algebra map, and it also injective (see, e.g. \cite{SW07}).
\end{example}

\subsubsection{Skein modules and representation varieties}
Here we recall a theorem of Przytycki and Sikora \cite{PS00} (and also of Bullock \cite{Bul97}) that identifies the commutative algebra $K_{q=-1}(M)$ with the algebra $\ochar(M)$ of functions on the $\SL_2(\C)$-character variety of $\pi_1(M)$.

An unbased loop $\gamma: S^1 \to M$ determines a conjugacy class in $\pi_1(M)$, and since the trace of a matrix is invariant on conjugacy classes, we can define $\mathrm{Tr}(\gamma) \in \ochar(M)$ via
\[
 \mathrm{Tr}(\gamma)(\rho) := \mathrm{Tr}[\rho(\gamma)]
\]

\begin{theorem}[\cite{PS00}, \cite{Bul97}]\label{thm_qeq1repvar}
 The assignment $\gamma \mapsto -\mathrm{Tr}(\gamma)$ extends to an algebra isomorphism 
 \[ K_{q=-1}(M) \stackrel \sim \to \ochar(M)\]
\end{theorem}

The key observation behind this theorem is that for $q=-1$, the skein relation becomes the following:
\[
 \mathrm{Tr}(A)\mathrm{Tr}(B) = \mathrm{Tr}(AB) + \mathrm{Tr}(AB^{-1})
\]
(This identity is a simple consequence of the Hamilton-Cayley identity, and is valid for any matrices $A,B \in \SL_2(\C)$.)
\subsubsection{The Kauffman bracket skein module of the torus}
We recall that the \emph{quantum torus} is the algebra
\[
A_q := \frac{\C\langle X^{\pm 1},Y^{\pm 1}\rangle}{XY-q^2YX}
\]
where $q \in \C^*$ is a parameter. Note that $\Z_2$ acts by algebra automorphisms on $A_q$ by inverting $X$ and $Y$.

We now we recall a beautiful theorem of Frohman and Gelca in \cite{FG00} that gives a connection between skein modules and the invariant subalgebra $A_q^{\Z_2}$. First we introduce some notation. Let $T_n \in \c[x]$ be the Chebyshev polynomials defined 
by $T_0 = 2$, $T_1 = x$, and the relation $T_{n+1} = xT_n-T_{n-1}$. If $m,l$ are relatively prime, write $(m,l)$ for the $m,l$ curve on the torus (the simple curve wrapping around 
the torus $l$ times in the longitudinal direction and $m$ times in the meridian's direction). It is clear that the links $(m,l)^n$ span $K_q(T^2)$, and it follows from \cite{SW07} that this set is actually a basis.  However, a more convenient basis is given by the elements $(m,l)_T := T_d((\frac m {d}, \frac l {d}))$ (where $d = \mathrm{gcd}(m,l)$). 
Define $e_{r,s} = q^{-rs}X^{r}Y^s \in A_q$, which form a linear basis for the quantum torus $A_q$ and satisfy the relations
\[e_{r,s}e_{u,v} = q^{rv-us}e_{r+u,s+v}\]

\begin{theorem}[\cite{FG00}]\label{fg00}
The map $K_q(T^2) \to A_q^{\Z_2}$ given by $(m,l)_T\mapsto e_{m,l}+e_{-m,-l}$ is an isomorphism of algebras.
\end{theorem}

\begin{remark}\label{remark_canonicalmodulestructure}
 As explained in Section \ref{subsec_knotgroups}, if $K$ is an oriented knot, then there is a canonical identification of $S^1\times S^1$ with the boundary of $S^3\setminus K$. If the orientation of $K$ is reversed, this identification is twisted by the `hyper-elliptic involution' of $S^1\times S^1$ (which negates both components). However, this induces the identity isomorphism on $K_q(T^2\times [0,1])$, so the $A_q^{\Z_2}$-module structure on $K_q(S^3\setminus K)$ is canonical and does not depend on the choice of orientation of $K$.
\end{remark}

We also recall another presentation of this algebra that will be useful for computations.
Let $x,y,z \in K_q(T^2)$ be the meridian, longitude, and $(1,1)$ curve, respectively.
\begin{theorem}[\cite{BP00}]\label{thm_bp00}
The algebra $K_q(T^2)$ is generated by $x,y,z$ with relations
\begin{equation}\label{relationsforB'}
[x,y]_q = (q^2-q^{-2})z,\quad
[z,x]_q = (q^2-q^{-2})y,\quad
[y,z]_q = (q^2-q^{-2})x
\end{equation}
and the additional cubic relation
\begin{equation}\label{casimir_rel}
q^2x^2 + q^{-2}y^2+ q^2z^2 -qxyz= 2(q^2+q^{-2})
\end{equation}
\end{theorem} 

Combining this presentation with the isomorphism $K_q(T^2) \cong A_q^{\Z_2}$, we have $x \mapsto X+X^{-1}$, $y \mapsto Y+Y^{-1}$, and $z \mapsto q^{-1}(XY+X^{-1}Y^{-1})$.

\subsubsection{A topological pairing}\label{topologicalpairing}
Let $K \subset S^3$ be a knot. There is a natural pairing 
\[ 
K_q(D^2\times S^1)\otimes_\C K_q(S^3\setminus K) \to \C 
\]
which is used to compute colored Jones polynomials. Informally, this pairing is induced by gluing a solid torus $D^2\times S^1$ to the complement of a tubular neighborhood of a knot to obtain $S^3$.

Let  $N_K \subset S^3$ be a closed tubular neighborhood of $K$, and let $N_c$ be the closure of its complement. Then $N_K \cap N_c$ is a torus $T$, 
and we let $N_T$ be a closed tubular neighborhood of $T$. By Remark \ref{remark_canonicalmodulestructure}, both $K_q(N_K)$ and $K_q(N_c)$ have canonical $A_q^{\Z_2}$-module structure. More precisely, if we identify $N_T$ with $T \times [0,1]$, then the 
embedding $T \times [0,1] \hookrightarrow S^3$ gives $K_q(N_c)$ and $K_q(N_K)$ a left and right $A_q^{\Z_2}$-module structure\footnote{The asymmetry between left and right comes from the definition of 
multiplication in $F \times [0,1]$: the product $ab$ means ``stack $a$ on top of $b$.'' Since the tori $T^2\times \{0\}$ and $T^2\times\{1\}$ are glued to $N_c$ and $N_K$, the spaces $K_q(N_c)$ and $K_q(N_K)$ are a left and right $K_q(T^2)$-modules (respectively).} (respectively).

It is easy to see that $K_q(N_K \sqcup N_c) \cong K_q(N_K)\otimes_\C K_q(N_c)$, and the embedding property (\ref{embedding}) above shows that this induces a map $\langle -,-\rangle:K_q(N_K)\otimes_\C K_q(N_c) \to K_q(S^3)$. 
If $\alpha \subset N_T$ is a link, it can be isotoped to a link inside $N_K$ or a link inside $N_c$, and inside $S^3$ both these links are isotopic. Since these isotopies define the module structure of $K_q(N_K)$ and $K_q(N_c)$, the pairing $\langle -,-\rangle$ descends to
\begin{equation}\label{knotpairing}
\langle -,-\rangle:K_q(D^2\times S^1)\otimes_{K_q(T^2)} K_q(S^3\setminus K) \to \C
\end{equation}
(To ease notation for later reference, in this formula we have identified $N_K$ with the solid torus $D^2 \times S^1$ and written $S^3\setminus K$ for $N_c$. We also used the isomorphism $K_q(S^3) \cong \C$ 
described in Example \ref{skeins3}.)

\subsubsection{The colored Jones polynomials}\label{subsec_coloredJpolys}
The colored Jones polynomials $J_{n, K}(q) \in \C[q^{\pm 1}]$ of a knot $K \subset S^3$ were originally defined by Reshetikhin and Turaev in \cite{RT90} using the representation theory of $\U_q(\sl_2)$. (In fact, their definition works for any semisimple Lie algebra $\mathfrak g$, but we only deal with $\mathfrak g = \sl_2$.) Here we recall a theory of Kirby and Melvin that shows that $J_{n,K}(q)$ can be computed in terms of the pairing from the previous section. 

If $N_K$ is a tubular neighborhood of the knot $K$, then we identify $K_q(N_K)\cong \C[u]$, where $u \in K_q(N_K)$ is the image of the (0-framed) longitude $l \in K_q(\partial N_K)$. 
Let $S_n \in \C[u]$ be the Chebyshev polynomials of the second kind, which satisfy the initial conditions $S_0 = 1$ and $S_1 = u$, and the recursion relation $S_{n+1} = uS_n - S_{n-1}$.

\begin{theorem}[\cite{KM91}]\label{thm_coloredjonespolys}
 If $\varnothing \in K_q(S^3\setminus K)$ is the empty link, we have
 \[
  J_{n,K}(q) = (-1)^{n-1}\langle \varnothing\cdot S_{n-1}(u), \varnothing \rangle
 \]
\end{theorem}
\begin{remark}\label{remark_signconvention}
The sign correction is chosen so that $J_{n, \mathrm{unknot}}(q) = (q^{2n}-q^{-2n})/(q^2-q^{-2})$. Also, with this normalization, $J_{0,K}(q) = 0$ and $J_{1,K}(q) = 1$ for every knot $K$. These conventions agree with the convention of labelling irreducible representations of $\U_q(\sl_2)$ by their dimension.
\end{remark}


\subsection{The $C^{\vee} C_1$ double affine Hecke algebra}
In this section we define the 5-parameter family of algebras $\H_{q,\ult}$ which was introduced by Sahi in \cite{Sah99} (see also \cite{NS04}).
This is the universal deformation of the algebra $\C[X^{\pm 1}, Y^{\pm 1}] \rtimes \Z_2$ (see \cite{Obl04}), and it depends on the parameters $q \in \C^*$ and $\ult \in (\C^*)^4$. The algebra $\H_{q,\ult}$ can be abstractly presented as follows: it is generated by the elements $T_0$, $T_1$, $T_0^\vee$, and $T_1^\vee$ subject to the relations
\begin{align}\label{ccdaharelations}
 (T_0-t_1)(T_0+t_1^{-1}) &= 0\notag\\
 (T_0^\vee-t_2)(T_0^\vee+t_2^{-1}) &= 0\notag\\
 (T_1-t_3)(T_1+t_3^{-1}) &= 0\\
 (T_1^\vee-t_4)(T_1^\vee+t_4^{-1}) &= 0\notag\\
 T_1^\vee T_1T_0 T_0^\vee &= q\notag
\end{align}
\begin{remark}
Comparing notation to \cite{NS04}, our $q^{-2}$ is their $q$, and our parameters $(t_1,t_2,t_3,t_4)$ are their $(k_0,u_0,k_1,u_1)$. These parameters relate to the original parameters $a,b,c,d$ of Askey and Wilson \cite{AW85} via $a=t_3t_4$, $b = -t_3t_4^{-1}$, $c=q^{-1}t_1t_2$, and $d = -q^{-1}t_1t_2^{-1}$. 
\end{remark}

\begin{remark}
The algebra $\H_{q,\ult}$ is a flat deformation of the fundamental group algebra of an orbifold Riemann surface. Recall that if $X$ is a simply connected Riemann surface and $\Gamma$ is a cocompact lattice (i.e. a Fuchsian subgroup) in $\mathrm{Aut}(X)$, the quotient $\Sigma = X / \Gamma$ is defined as an orbifold and $\Gamma$ is isomorphic to the (orbifold) fundamental group of $\Sigma$ (see, e.g. \cite[Sec. 2]{Sco83}):
\[
 \pi_1^{\mathrm{orb}}(\Sigma, *) = \langle a_1,b_1,\ldots,a_g,b_g,c_1,\ldots,c_n \mid c_i^{n_i} = 1,\, \prod_{i=1}^g [a_i,b_i] c_1\cdots c_n = 1\rangle
\]
where $c_i$ are generators corresponding to loops around special points of $X$ with stabilizers $\Z / n_i\Z$ and $n_i > 1$. In the case when $X = \C$ and $\Gamma = (\Z \oplus i\Z)\rtimes \Z_2$ acting on $X$ by translation-reflections, we recover from the isomorphism $G \cong \pi_1^{\mathrm{orb}}(\Sigma,*)$ the presentation
\[
 \Gamma = \langle c_1,c_2,c_3,c_4 \mid c_i^2 = 1,\, c_1c_2c_3c_4=1\rangle
\]
where the $c_i$ are loops around 4 special points $\{0,1/2,1/2+i/2,i/2\} \in \C$. Thus, $\H_{1,1,1,1,1} \cong \C[\Gamma]$. For other interesting examples of Hecke algebras associated to Fuchsian groups, see \cite{EOR07}.
\end{remark}

For $\ult = \underline{1}$, this algebra is isomorphic to $A_q\rtimes\Z_2$ (see Remark \ref{remark_ccotherpresentation}). This algebra can also be realized concretely as a subalgebra of a certain localization of $A_q\rtimes \Z_2$. More precisely, we recall that $A_q\rtimes \Z_2$ is generated by $X$, $\yy$, and $s$ (all invertible), which satisfy the relations
\[
 sX=X^{-1}s,\quad s\yy = \yy^{-1}s, \quad s^2 = 1,\quad X\yy = q^2\yy X
\]
Let $D_q$ be the localization of $A_q\rtimes \Z_2$ obtained by inverting all nonzero polynomials in $X$, and define the following operators in $D_q$:
\begin{eqnarray*}
 T_0 &=& t_1 s\yy - \frac{q^2 \bar t_1 X^2  + q\bar t_2 X}{1-q^2X^2}(1-s\yy)\\
 T_1 &=& t_3s + \frac{\bar t_3 +\bar t_4X}{1-X^2}(1-s)
\end{eqnarray*}
(We have slightly abused notation by giving these operators the same names as the abstract generators of $\H_{q,t}$.) The following Dunkl-type embedding is defined using these operators (see \cite[Thm. 2.22]{NS04}):
\begin{proposition}[\cite{Sah99}]\label{prop_dunklembedding}
 The assignments 
 \begin{equation}\label{ccdunklembedding}
 T_i \mapsto T_i,\quad T_0^\vee \mapsto qT_0^{-1}X,\quad T_1^\vee \mapsto X^{-1}T_1^{-1}
\end{equation}
extend to an injective algebra homomorphism $\H_{q,\ult} \to D_q$.
\end{proposition}

We recall that 
the standard polynomial representation $V = \C(X)$ of $D_q$ is isomorphic as a $\C[X^{\pm 1}]$-module to rational functions in $X$, with the action of $s$ and $\yy$ given by
\[
 s\cdot f(X) = f(X^{-1}),\quad\quad \yy\cdot f(X) = f(q^{-2}X)
\]
Under this action, it is easy to check that both operators $T_0$ and $T_1$ preserve the subspace $\C[X^{\pm 1}] \subset \C(X)$, which shows that the action of $\H_{q,\ult}$ on $\C(X)$ preserves $\C[X^{\pm 1}] \subset \C(X)$. In other words, $\C[X^{\pm 1}]$ is an $\H_{q,\ult}$-module, which is called the \emph{polynomial representation}.

\begin{remark}\label{remark_ccotherpresentation}
The algebra $\H_{q,\ult}$ is also generated by the elements
 $X^{\pm 1}$, $Y := T_1T_0$, and $T := T_1$. With this set of generators, the relations become the following (see \cite[2.21]{NS04}):
\begin{eqnarray}\label{equation_ccxyt}
XT &=& T^{-1}X^{-1} - \bar t_4\notag\\
T^{-1}Y &=& Y^{-1}T + \bar t_1\notag\\
T^2 &=& 1 + \bar t_3T\notag\\
TXY &=& q^2T^{-1}YX - q^2 \bar t_1 X - q \bar t_2  - \bar t_4Y
\end{eqnarray}
(where $\bar t_i = t_i - t_i^{-1}$). With this presentation, it is clear that $\H_{q,1,1,1,1} = A_q\rtimes \Z_2$ (as subalgebras of $D_q$), since under this specialization of the parameters, we have $T = T_1 = s$ and $Y = \hat y$.
\end{remark}

The element $\e := (T_1+t_3^{-1})/(t_3+t_3^{-1}) \in \H_{q,\ult}$ is an idempotent, and the algebra $\S\H_{q,t} := \e \H_{q,\ult}\e$ is called the \emph{spherical subalgebra}. It is easy to check that $\e$ commutes with $X+X^{-1}$, and this implies the subspace $\e \cdot \C[X^{\pm 1}]$ is equal to the subspace $\C[X+X^{-1}]$ of symmetric polynomials. The spherical algebra therefore acts on $\C[X+X^{-1}]$, and this module is called the \emph{symmetric polynomial representation}.

A presentation for the spherical subalgebra $\S\H_{q,t}$ has been given in \cite{Koo08}. We now recall this presentation in our notation. First, we define 
\begin{eqnarray*}
x &=& (X+X^{-1})\e\\ 
y &=& (Y+Y^{-1})\e\\
z &=& \frac{[x,y]_q}{(q^2-q^{-2})}
\end{eqnarray*}
\begin{theorem}[\cite{Koo08}]
The spherical subalgebra is generated by $x,y,z$ with relations
\begin{eqnarray*}
 [x,y]_q &=& (q^2-q^{-2})z\\
{ } [y,z]_q &=& (q^2-q^{-2})x - (q-q^{-1})By - (q-q^{-1})D_1\\
{ } [z,x]_q &=& (q^2-q^{-2})y - (q-q^{-1})Bx + (q-q^{-1})D_0\\
 Q_0 &=& -qxyz + q^2x^2 + q^{-2}y^2 + q^2z^2 - qD_1x - q^{-1}D_0y - qB(xy-(q-q^{-1})z)
\end{eqnarray*}
\end{theorem}
The right hand side of the final relation is a central element in the algebra generated by the elements $x$, $y$ , and $z$ subject to the first three relations, and the constants $Q_0$, $B$,  $D_1$, and $D_0$ are given by
\begin{eqnarray*}
 (q+q^{-1})^2Q_0 &:=& -(q^2+q^{-2})\bar t_1 \bar t_2 (\overline{qt_3}) \bar t_4 + (q^2+q^{-2})t_2^{-2}(t_1^2+1+t_2^2) \\
 &{ }& + (1+q^{-2}t_1^{-2})[q^4+q^2(q^2+t_1^2)(t_4^2+t_4^{-2})] + (q^2+t_2^2)(q^2+t_2^{-2})(q^{-2}t_3^{-2}+q^2t_3^2)\\
 &{ }& - 4(q+q^{-1})^2 + 6 + q^{-4} + q^{-2}\\
 B &:=& \frac{1}{q+q^{-1}}(\bar t_2 (\overline{qt_3}) + \bar t_1 \bar t_4)\\
 D_0 &:=& \bar t_1 (\overline{qt_3}) + \bar t_2 \bar t_4\\
 D_1 &:=& \bar t_1 \bar t_2 + (\overline{qt_3})\bar t_4
\end{eqnarray*}
(Here we have used the notation $\bar t_i := t_i - t_i^{-1}$ and $\overline{qt_3} := qt_3-q^{-1}t_3^{-1}$.)

\begin{remark}
 If $q=1$ then the spherical subalgebra is commutative - the spectrum of this algebra was studied in detail in \cite{Obl04}. Also, if $t_i=1$ this presentation agrees with the presentation of \cite{BP00} for the skein algebra of the torus (see Theorem \ref{thm_bp00}).
\end{remark}

The operator $L_{t_1,t_2} := Y+Y^{-1} = T_1T_0 + T_0^{-1}T_1^{-1}$ is called the \emph{Askey-Wilson operator}: it also commutes with $\e$, so it preserves the subspace of symmetric functions $\C[X+X^{-1}]$. If we write $L_{t_1,t_2}^{\mathrm{sym}}$ for the restriction of this operator to $\C[X+X^{-1}]$, then $L_{t_1,t_2}^{\mathrm{sym}}$ is diagonalizable with distinct eigenvalues (for generic parameters), and its eigenvectors are the \emph{Askey-Wilson polynomials} (see, e.g. \cite{Mac03}). In the following lemma, we use the notation $(x)^+ = x+x^{-1}$ and $(x)^- = x-x^{-1}$.
\begin{lemma}\label{lemma_askeywilson}
 If $t_3=t_4=1$ the Askey-Wilson operator can be written as follows:
\begin{align*}
  L_{t_1,t_2}^{\mathrm{sym}} =  \frac{1}{(X^2)^+-(q^2)^+}\Bigg[&t_1(X^2\yy)^+ + t_1^{-1}(X^{-2}\yy)^+ 
   - t_1^+(X^2)^+  - (t_1^{-1}q^2)^+(\yy^+-2) \\
   &+ t_2^-\Big(q(X\yy)^+-q^{-1}(X^{-1}\yy)^+ -q^-X^+\Big)\Bigg] + t_1^+
\end{align*}
\end{lemma}
\begin{proof}
 This follows from \cite[Prop. 5.8]{NS04} and a short calculation.
\end{proof}
(We remark that under the isomorphism of Theorem \ref{fg00}, $(X\yy)^+ = q(1,1)$, where the right hand side is the $(1,1)$ curve on the torus. This accounts for the apparent differences in the powers of $q$ in the last terms of the above formula for $L_{t_1,t_2}^{\mathrm{sym}}$ and Remark \ref{remark_eUe}.)

\begin{remark}\label{remark_a1cciso}
 Cherednik's $\sl_2$ double affine Hecke algebra $\HH_{q,t}$ is isomorphic to $\H_{q,1,1,t^{-1},1}$. Under this specialization, the presentation (\ref{equation_ccxyt}) of $\H_{q,1,1,t^{-1},1}$ becomes
 \begin{equation}\label{equation_a1daha}
  TXT = X^{-1},\quad TY^{-1}T = Y,\quad (T-t^{-1})(T+t) = 0,\quad XY=q^2T^{-2}YX
 \end{equation}
The standard presentation of $\HH_{q,t}$ (see \cite{Che05}) replaces the last relation (\ref{equation_a1daha}) with $XY=q^2YXT^2$. Under the map $\HH_{q,t} \to \H_{q,1,1,t^{-1},1}$ given by $X \mapsto X^{-1}$, $Y \mapsto Y^{-1}$, and $T \mapsto T^{-1}$, this becomes the last relation in (\ref{equation_a1daha}).
\end{remark}

Let $M$ be an $A_q\rtimes \Z_2$ module, let $M^\loc$ be its localization at nonzero polynomials in $X$, and suppose $M \to M^\loc$ is injective. Let $U_0 = (1-q^2X^2)^{-1}(1-\s \yy) \in D_q$.
\begin{lemma}\label{lemma_twisteddunkl}
If $U_0M \subset M$, then $\HH_{q,t}$ acts on $M$.
\end{lemma}
\begin{proof}
 Cherednik's embedding of $\HH_{q,t}$ into $D_q$ is given by the following formulas (see \cite{Che05}):
 \[
  X \mapsto X,\quad T \mapsto T_1,\quad Y \mapsto \yy s T_1
 \]
with parameters $(t_1,t_2,t_3,t_4) = (1,1,t,1)$. We now twist this embedding by the automorphism of $D_q$ given by $X \mapsto qX$, $\yy \mapsto \yy$, and $s \mapsto s \yy$. After this twist, the map $\HH_{q,t} \to D_q$ is given by
\[
  X \mapsto qX,\quad T \mapsto ts \yy + {\bar t}U_0,\quad Y \mapsto t \yy  + {\bar t}sU_0
 \]
 By assumption, $U_0M \subset M$, which implies $\HH_{q,t}M \subset M$.
\end{proof}

\section{Deformed skein modules of 2-bridge knots}\label{sec_qequals1}
In this section we will prove Conjectures \ref{mainconjecture}, \ref{conj_unknotsubmodule}, and \ref{conj_5param} (when $q=-1$) for an arbitrary 2-bridge knot. To this end we will use an algebraic construction of $\SL_2$ character varieties of finitely generated groups due to Brumfiel and Hilden \cite{BH95}. We begin by recalling the results of \cite{BH95} in the form that we need.

\subsection{The Brumfiel-Hilden construction}\label{sec_BHconstruction}
Let $\mathbf{Grp}$ be the category of (discrete) groups and let $\mathbf{Alg}^*$ be the category of associative $\C$-algebras equipped with an anti-involution $a \mapsto a^*$. 
Assigning to a group $\pi$ its complex group algebra $\C\pi$ defines a functor $\C [-]:\mathbf{Grp} \to \mathbf{Alg}^*$, where the anti-involution on $\C\pi$ is defined on the group elements by $g^* := g^{-1}$ and is extended to $\C\pi$ by linearity. The group algebra functor has an obvious right adjoint $\SU:\mathbf{Alg}^* \to \mathbf{Grp}$ which is defined by $\SU(A) := \{a \in A \mid aa^* = 1\}$. 

Now, for a commutative $\C$-algebra $B$, the algebra $\mathbb M_2(B)$ of $2\times 2$ matrices over $B$ has an anti-involution given by the classical adjoint:
\begin{equation}\label{equation_adjoint}
\left(
\begin{array}{cc} a&b\\c&d\end{array}\right)^* := \left(\begin{array}{cc}d&-b\\-c&a\end{array}\right)
\end{equation}
In this case, $\SU[\mathbb M_2(B)] = \mathrm{\SL}_2(B)$. Thus for any commutative $\C$-algebra $B$ we have a natural bijection
\begin{equation}\label{equation_adjointpair}
\mathrm{Hom}_{\mathbf{Alg}^*}(\C[\pi],\mathbb{M}_2(B)) = \mathrm{Hom}_{\mathbf{Grp}}(\pi,\mathrm{\SL}_2(B)) 
\end{equation}

The representation scheme $\Rep(\pi) := \Rep(\pi,\SL_2)$ is defined by its functor of points, and with the identification (\ref{equation_adjointpair}), this can be written as
\[
\mathrm{Rep}(\pi): \mathbf{CommAlg} \to \mathbf{Sets},\quad B \mapsto \mathrm{Hom}_{\mathbf{Alg}^*} (\C[\pi],\mathbb M_2(B))
\]
As explained in Section \ref{section_charvarieties}, this functor is representable by the commutative algebra $\orep(\pi) := \O(\Rep(\pi))$. 
Let $\rho_u:\C[\pi] \to \mathbb M_2(\orep(\pi))$ be the algebra map corresponding to the universal representation $\rho_u:\pi \to \mathrm{\SL}_2(\orep(\pi))$.  Geometrically, the scheme $\Rep(\pi)$ parametrizes representations of $\pi$ into $\mathrm{\SL}_2(\C)$, and $\rho_u(g)$ is the section of the trivial bundle $\Rep(\pi) \times \mathbb M_2(\C)$ given by $\rho \mapsto \rho(g)$.

It is easy to see that for any $\pi$, the homomorphism $\rho_u$ factors through the algebra
\[
\BH[\pi] := \frac{\C[\pi]}{\langle [g,h+h^{-1}] \mid g,h \in \pi \rangle}
\]
which we call the \emph{Brumfiel-Hilden algebra} of $\pi$ (since it was introduced and studied in \cite{BH95}). The algebra $\BH[\pi]$ has a canonical (commutative) subalgebra $\BH^+[\pi] := \{a \in \BH[\pi] \mid a^* = a\}$ (which is actually central in $\BH[\pi]$). The meaning of these algebras is made clear by the following result proved in \cite[Prop. 9.1]{BH95}:

\begin{theorem}\label{theorem_BHbasic}
Assume that $\pi$ is a finitely presented group.
\begin{enumerate}
\item The universal representation $\rho_u:\C[\pi] \to \mathbb M_2(\orep(\pi))$ factors through $\BH[\pi]$, and the induced map $\bar \rho_u:\BH[\pi] \to \mathbb M_2(\orep(\pi))$ is injective.
The image of $\bar \rho_u$ coincides with the subring of $\GL_2(\C)$-invariants in $\mathbb M_2(\orep(\pi))$, and we therefore have a canonical isomorphism of algebras
\[
\bar \rho_u:\BH[\pi] \stackrel \sim \to \mathbb M_2(\orep(\pi))^{\GL_2(\C)}
\]
\item The invariant subring $\BH^+[\pi]$ is mapped bijectively by $\bar \rho_u$ onto $\orep(\pi)^{\GL_2(\C)}$.
\end{enumerate}
\end{theorem}

Summarizing, we have the following commutative diagram:
\begin{diagram}
 \BH^+[\pi] & \rInto & \BH[\pi] & \lOnto & \C[\pi]\\
 \dTo^{\rotatebox[origin=c]{90}{$\sim$}} & & \dTo^{\rotatebox[origin=c]{90}{$\sim$}} & & \dTo^{\rho_u}\\
 \orep(\pi)^{\GL_2(\C)} & \rInto & \mathbb M_2(\orep(\pi))^{\GL_2(\C)} & \rInto & \mathbb M_2(\orep(\pi))
\end{diagram}

Geometrically, this theorem says that $\BH[\pi]$ is the algebra of $\GL_2(\C)$-equivariant matrix-valued functions on the representation scheme, and $\BH^+[\pi]$ is the (commutative) algebra of $\GL_2(\C)$-invariant scalar functions. In other words, $\BH^+[\pi]$ is the ring of functions on the $\SL_2$ character variety of $\pi$.

Now, suppose $M$ is a manifold with boundary $\partial M$. The inclusion $\partial M \to M$ induces a map of groups $\alpha:\pi_1(\partial M) \to \pi_1(M)$. 
By functoriality of the Brumfiel-Hilden construction, we have the following commutative diagram of algebra homomorphisms:

\begin{diagram}
 \BH^+[\pi_1(\partial M)] & \rInto & \BH[\pi_1(\partial M)] & \lOnto & \C[\pi_1(\partial M)]\\
 \dTo^{\alpha_*} & & \dTo^{\bar \alpha} & & \dInto^{\alpha}\\
 \BH^+[\pi_1(M)] & \rInto & \BH[\pi_1(M)] & \lOnto & \C[\pi_1(M)]
\end{diagram}

If $M = S^3\setminus K$ is the complement of a nontrivial knot in $S^3$, then $\partial M = T^2$ and the map $\alpha$ is an embedding (see e.g. \cite[Prop. 3.17]{BZ03}) (however, the induced map $\bar \alpha$ is not injective). Also, in this case the map $\C[\pi_1(T^2)] \to \BH[\pi_1(T^2)]$ is an isomorphism because the fundamental group of $T^2$ is abelian. With the identifications of Theorem \ref{theorem_BHbasic}, the leftmost arrow is precisely the peripheral map discussed in Section \ref{section_charvarieties}.

\subsection{The Brumfiel-Hilden algebra of a 2-bridge knot}
Theorem \ref{theorem_BHbasic} reduces the problem of describing the character variety $\chr(\pi,\SL_2)$ for a finitely presented group to that of describing the algebra $\BH[\pi]$. In many interesting cases, e.g. for 2-generator groups, $\BH[\pi]$ can be computed explicitly (see \cite[Chap. 3]{BH95}). If $K \subset S^3$ is a 2-bridge knot, the fundamental group $\pi_1(S^3\setminus K)$ is generated by two elements (meridians) subject to one relation; in this case, $\BH^+[\pi]$ is isomorphic to the ring of regular functions on a plane curve and $\BH[\pi]$ has the structure of a generalized quaternion algebra over $\BH^+[\pi]$. We will briefly describe this structure below, and refer the reader to \cite[Chap. 4*]{BH95} for proofs and more details. We begin with the standard presentation of $\pi_1(S^3\setminus K)$ in the case of 2-bridge knots.

\subsubsection{The fundamental group}\label{sec_fundamentalgroup}
Recall that the 2-bridge knots can be indexed (non-uniquely) by pairs $(p,q)$ of relatively prime odd integers with $p > 0$ and $\lvert q \rvert < p$. For a knot $K = K(p,q)$, the group $\pi := \pi_1(S^3\setminus K(p,q))$ has a presentation
\[
 \pi = \langle a,b \mid aw = wb\rangle
\]
where $a$ and $b$ are meridians around two trivial strands contained in one of the balls in a 2-bridge decomposition of $L$. The element $w$ can be expressed in terms of $a^{\pm 1}$ and $b^{\pm 1}$ as a product of two words
\[
 w = v \bar v
\]
which are images of each other under the anti-involution of the free group $F\langle a,b\rangle$ that switches $a$ and $b$. The word $v$ is given by
\[
 v := b^{e_1}a^{e_2}\cdots (a \textrm { or } b)^{e_{d}}
\]
where $d = (p-1)/2$ and the exponents $e_n \in \{\pm 1\}$ are computed by the rule $\mathrm{sign}(e_n) = \mathrm{sign}(k_n)$, with $k_n$ defined by the conditions
\[
 k_n \equiv n q\,(\mathrm{mod}\, 2 p),\quad -  p < k_n < p,\quad k_n \not= 0
\]
The peripheral map $\alpha: \pi_1(T^2) \to \pi$ is defined by the assignments
\[
 \alpha(m) = a,\quad \quad \alpha(l) = w \tilde w a^{-s}
\]
where $\tilde w$ is the word $w$ written backwards (without switching $a$ and $b$), and $s = 4\sum_{n=1}^{d} e_n$.

\begin{example}
Let $p=5$ and $q=3$. The corresponding knot $K = K(5,3)$ is the figure eight knot. In this case, $d = (5-1)/2 = 2$, $k_1 = 3$, $k_2 = -4$, so that $e_1 = 1$, $e_2=-1$, and $s = e_1 + e_2 = 0$. We therefore have $w = ba^{-1}b^{-1}a$, and the fundamental group has the following presentation:
\[
 \pi_1(K) = \langle a,b\mid aba^{-1}b^{-1}a = ba^{-1}b^{-1}ab\rangle = \langle a,b \mid aba^{-1}ba = bab^{-1}ab\rangle
\]
(The second presentation is standard, and the relation in the first is conjugate to the relation in the second.)
In this notation, the peripheral map $\alpha: \pi_1(T^2) \to \pi_1(K)$ is given by
\[
 m \mapsto a,\quad l \mapsto ba^{-1}b^{-1}a^2b^{-1}a^{-1}b
\]

\end{example}

\subsubsection{The Brumfiel-Hilden algebra}
For a 2-bridge knot, the algebra $\BH[\pi]$ has the following structure (see \cite[Prop. A.4*.9]{BH95}).
\begin{theorem}\label{theorem_BHA49}
 Let $\pi$ be the fundamental group of a 2-bridge knot. There is a polynomial $Q = Q(I,J) \in \C[I,J]$ of degree $d = (p-1)/2$ such that $\BH[\pi]$ admits an $\BH^+[\pi]$-module decomposition
 \begin{equation*}
  \BH[\pi] = \BH^+[\pi] \oplus \bar \BH^+[\pi]i \oplus \BH^+[\pi]j \oplus \bar \BH^+[\pi]k
 \end{equation*}
where
\[
 \BH^+[\pi] := \frac{\C[x, I, J]}{\langle IQ, I+J-4(x^2-1) \rangle},\quad \bar \BH^+[\pi] := \frac{\C[x, I, J]}{\langle Q, I+J-4(x^2-1) \rangle}
\]
The multiplication in $\BH[\pi]$ is determined by the (generalized) quaternion relations
\[
 i^2 = I,\quad j^2 = J,\quad ij=-ji=k
\]
\end{theorem}
The canonical projection $\C[\pi] \to \BH[\pi]$ is given by the equations
\begin{eqnarray*}
 a^{\pm 1} &\mapsto& x \pm (i+j)/2\\
 b^{\pm 1} &\mapsto& x \mp (i-j)/2
\end{eqnarray*}
Finally, the canonical anti-involution $*: \BH[\pi] \to \BH[\pi]$ is given by
\[
 x^* = x,\quad i^* = -i,\quad j^* = -j, \quad k^* = -k
\]

\begin{remark}
 Under the identification of Theorem \ref{theorem_BHbasic}, $H^+[\pi] \cong \ochar(\pi)$, the generators $x$, $I$, and $J$ correspond to the following functions:
\begin{eqnarray}
 x &\mapsto &\frac 12 \Tr(a)\notag\\
 I &\mapsto &\frac 12\Tr({a})^2 + \frac 12 \Tr(a)\Tr(b) - \Tr({ab}) - 2\label{eq_trmap}\\
 J &\mapsto &\frac 12 \Tr(a)^2 - \frac 12 \Tr(a)\Tr(b) + \Tr({ab}) - 2\notag
\end{eqnarray}
 where $\Tr(g)$ is the character function $\rho \mapsto \Tr(\rho(g))$. The image of the polynomial $Q$ in $\BH^+[\pi]$ determines a curve of characters of mostly irreducible $\SL_2$ representations of $\pi$: more precisely, the characters (equivalently, the conjugacy classes) of irreducible representations $\rho:\pi \to \SL_2(\C)$ correspond to algebra homomorphisms $\varphi: H^+[\pi] \to \C$ such that $\varphi(Q) = 0$ and $\varphi(I) \not= 0$. It is shown in \cite{BH95} that $Q$ actually has integral coefficients, i.e. $Q(I,J) \in \Z[I,J]$. We give a formula for $Q$ in Remark \ref{remark_notationQ}.
\end{remark}

\subsection{The nonsymmetric skein module (at $q=-1$)}
In this section we prove Conjectures \ref{mainconjecture}, \ref{conj_unknotsubmodule}, and \ref{conj_5param} for all 2-bridge knots when $q=-1$. We fix a 2-bridge knot $K = K(p,q)$ and write $\BH := \BH[\pi]$ and $\BH^+ := \BH^+[\pi]$ for the corresponding knot group $\pi = \pi_1(S^3\setminus K)$. Let $X \in \BH$ and $Y \in \BH$ denote the images of the meridian $m$ and longitude $l$ under the (induced) peripheral map
\begin{equation}\label{equation_alphabar}
 \bar \alpha: \BH[T^2] = \C[m^{\pm 1},l^{\pm 1}] \to \BH
\end{equation}
With the identification of Theorem \ref{theorem_BHA49}, we have
\begin{equation}\label{equation_yuri9}
 X^{\pm 1} = x \pm (i+j)/2,\quad Y = w \tilde w X^{-s}
\end{equation}

Next, we let $\BH^+[X^{\pm 1}]$ denote the subalgebra of $\BH$ generated by $\BH^+$ and $X^{\pm 1}$. Note that $\BH^+[X^{\pm 1}]$ is commutative and the canonical anti-involution on $\BH$ restricts to $\BH^+[X^{\pm 1}]$ and maps $X \mapsto X^{-1}$. The following lemma is implicit in \cite{BH95} (see loc. cit., Prop. A.4*.10):
\begin{lemma}\label{lemma_yuri1}
 For any 2-bridge knot, the image of the peripheral map (\ref{equation_alphabar}) is contained in $\BH^+[X^{\pm 1}]$.
\end{lemma}
\begin{proof}
 We need to prove that $Y \in \BH^+[X^{\pm 1}]$. Following \cite{BH95}, we introduce the anti-automorphism $\gamma:\BH \to \BH$ that fixes $a$ and $b$ and sends $ab \mapsto ba$. We also introduce the involution $\sigma:\BH \to \BH$ that ``switches $a$ and $b$'' (i.e. $\sigma(a) = b$ and $\sigma(b) = a$). It is easy to see that $\gamma$ and $\sigma$ act trivially on $\BH^+$, while 
 \[
  \gamma: (i,j,k) \mapsto (i,j,-k)\quad \sigma: (i,j,k) \mapsto (-i,j,-k)
 \]
If we write $v = D + Ei + Fj + Gk$ for some $D,E,F,G \in \BH^+$, then
\[
 \bar v = \sigma \gamma (v) = D - Ei + Fj + Gk
\]
If follows that 
\[
 w = v \bar v = L + Mj + N k,\quad \tilde w = \gamma(w) = L + Mj - Nk
\]
where
\begin{eqnarray}\label{equation_yuri10}
 L &=& D^2 - E^2I+F^2J - G^2IJ\notag\\
 M &=& 2DF+2GEI\\
 N &=& 2DG+2EF\notag
\end{eqnarray}
We therefore have
\begin{equation}\label{equation_yuri11}
 w \tilde w = (L^2+M^2J+N^2IJ) + 2MNJi + 2MLj
\end{equation}
On the other hand, $w w^* = 1$ and $aw - wb = 0$ give (respectively)
\begin{eqnarray}
 L^2 - M^2J + N^2IJ &=& 1\label{equation_yuri12}\\
 (L-JN)i &=& 0\label{equation_yuri13}
\end{eqnarray}
Writing $\delta := i+j$, we see from (\ref{equation_yuri9}) that $\delta = X-X^{-1}$. Combining this with the equations (\ref{equation_yuri11}), (\ref{equation_yuri12}), and (\ref{equation_yuri13}) we get
\begin{equation}\label{equation_yuri14}
Y = w \tilde w X^{-s} = [(1+2M^2J) + 2ML(i+j)]X^{-s} = [(1+2M^2J) + 2ML\delta]X^{-s} \in \BH^+[X^{\pm 1}] 
\end{equation}
\end{proof}

\begin{remark}\label{remark_notationQ}
 In the notation of the previous lemma, $Q = L - JN$.
\end{remark}

Next, we note that $S = \{1,\delta,\delta^2,\ldots\}$ is an Ore subset in $\BH$. We write $\BH[\delta^{-1}]$ for the localization of $\BH$ at $S$ and define
\begin{equation}\label{equation_yuri15}
 M := \BH^+[X^{\pm 1}] + \BH^+[X^{\pm 1}]Q\delta^{-1} \subset \BH[\delta^{-1}]
\end{equation}
where $Q = Q(I,J)$ is the polynomial featured in Theorem \ref{theorem_BHA49}. By construction, $M$ is a module (actually, a fractional ideal) over $\BH^+[X^{\pm 1}]$, and by Lemma \ref{lemma_yuri1}, it is a module over $\C[m^{\pm 1},l^{\pm 1}]$. We extend this last module structure to $\C[m^{\pm 1},l^{\pm 1}]\rtimes \Z_2$, letting $s \in \Z_2$ act on $M$ by the canonical involution $X \mapsto X^{-1}$.

Now, recall the double affine Hecke algebra from (\ref{ccdaharelations}). We let $q=-1$ and $t_3=t_4=1$ and define
\begin{eqnarray}
 T_0 &\mapsto & -t_1 s Y + (\bar t_1 X + \bar t_2 )\delta^{-1}(1+sY)\label{equation_yuri16}\\
 T_0^\vee &\mapsto& T_0^{-1}X^{-1}\notag\\
 T_1 &\mapsto& s\notag\\
 T_1^\vee &\mapsto& Xs\notag
\end{eqnarray}
where $\bar t_i = t_i - t_i^{-1}$ for $i=1,2$. The main result of this section is the following:
\begin{theorem}\label{theorem_qeq1}
 For any 2-bridge knot, the assignment (\ref{equation_yuri16}) extends to an action of $\H_{-1,t_1,t_2}$ on $M$.
\end{theorem}

\begin{remark}
 The term $(1+sY)$ in formula (\ref{equation_yuri16}) agrees with the term $(1-s\yy)$ in formula (\ref{ccdunklembedding}) because of the sign in Theorem \ref{thm_qeq1repvar} (a loop $\gamma$ in the skein algebra at $q=-1$ gets sent to the function $\rho \mapsto -\mathrm{Tr}(\rho(\gamma))$ on the character variety). (The change $X \mapsto X^{-1}$ is inessential.)
\end{remark}
\begin{proof}
 Let $U = \delta^{-1}(1+sY) = \delta^{-1}(1+Y^{-1}s) \in \BH[\delta^{-1}]$. We need to check that $UM \subset M$.
For this, it suffices to check that there exist $f,g,h \in \BH^+[X^{\pm 1}]$ such that
\begin{eqnarray}
 Y &=& fQ + g\delta - 1\label{equation_yuri18}\\
 YQ &=& (1 + h\delta)Q\label{equation_yuri19}
\end{eqnarray}
or, equivalently, there exist $f',g',h' \in \BH^+[X^{\pm 1}]$ such that 
\begin{eqnarray}
 Y^{-1} &=& f'Q + g'\delta - 1\label{equation_yuri13'}\\
 Y^{-1}Q &=& (1 + h'\delta)Q\label{equation_yuri14'}
\end{eqnarray}
Indeed, observe that 
\begin{equation}\label{equation_yuri20}
 M = \BH^+[X^{\pm 1}]+\BH^+[X^{\pm 1}]Q\delta^{-1} =\BH^+ + \BH^+\delta + \BH^+Q\delta^{-1}
\end{equation}
so it suffices to check that $U$ maps $\BH^+$, $\BH^+\delta$, and $\BH^+Q\delta^{-1}$ into $M$. Assuming (\ref{equation_yuri13'}) and (\ref{equation_yuri14'}), we have
\begin{eqnarray*}
 U[\BH^+] &=& \delta^{-1}(1+Y^{-1}s)\BH^+ = \delta^{-1}(1+Y^{-1})\BH^+\\
 &=& \delta^{-1}(f'Q+g'\delta)\BH^+ = g'\BH^+ + f'\BH^+Q\delta^{-1}
\end{eqnarray*}
Similarly, we have
\begin{eqnarray*}
 U[\BH^+Q\delta^{-1}] &=& \delta^{-1}(1+Y^{-1}s)\BH^+Q\delta^{-1} = \delta^{-1}(1-Y^{-1})Q\BH^+\delta^{-1}\\
 &=& \delta^{-1}(-h'\delta Q)\BH^+\delta^{-1} = -h'\BH^+Q\delta^{-1} \subset M\\
 U[\BH^+\delta] &=& \delta^{-1}(1+Y^{-1}s)\BH^+\delta = \delta^{-1}(1-Y^{-1})\BH^+\delta\\
 &=& (1-Y^{-1})\BH^+ \subset \BH^+[X^{\pm 1}] \subset M
\end{eqnarray*}
where in the last line we used Lemma \ref{lemma_yuri1} (and the fact that $X$ and $Y$ commute).

To prove (\ref{equation_yuri18}), we note that by Remark \ref{remark_notationQ} we have $L-JN = Q$, and we also have $I = 4(x^2-1) - J = \delta^2 - J$. Hence, by (\ref{equation_yuri12}) and (\ref{equation_yuri14}) we have
\begin{eqnarray*}
 Y &=& X^{-s}(1+2M^2J + 2ML\delta)\\
 &=& X^{-s}(2L^2 + 2N^2J(\delta^2-J) - 1 + 2ML\delta)\\
 &=& X^{-s}[2(L+NJ)(L-NJ)+2(N^2J\delta + LM)\delta - 1]\\
 &=& 2X^{-s}(L+NJ)Q + 2X^{-s}(N^2J\delta + LM)\delta - X^{-s}\\
 &=& 2X^{-s}(L+NJ)Q + X^{-s}[2NJ\delta + 2LM + A(X)]\delta - 1
\end{eqnarray*}
where $A(X) := X + X^3 + \cdots + X^{s-1}$ with $s = 4\sum_{n=1}^d e_n$. Formula (\ref{equation_yuri19}) follows from Lemma \ref{lemma_yqeqq} below, which completes the proof of the theorem.
\end{proof}

As a consequence of Theorem \ref{theorem_qeq1}, we can deduce Conjecture \ref{mainconjecture} (for $q=-1$) ``at the symmetric level.''
\begin{corollary}\label{cor_qeq1}
 The $\SL_2$-character ring of any 2-bridge knot carries a natural action of the spherical subalgebra $\mathrm S\H_{-1,t_1,t_2}$ of the double affine Hecke algebra $\H_{-1,t_1,t_2}$.
\end{corollary}
\begin{proof}
 Since $\e \delta = (1+s)\delta / 2 = 0$, equation (\ref{equation_yuri20}) shows that $\e M = \BH^+$, while $\BH^+ = \ochar(\pi,\SL_2)$ by Theorem \ref{theorem_BHbasic}. Then the action of $\H_{-1,t_1,t_2,1,1}$ on $M$ induces an action of $\mathrm S\H_{-1,t_1,t_2,1,1}$ on $\e M$.
\end{proof}

To prove Conjecture \ref{conj_unknotsubmodule} we need the following lemma which refines the formula (\ref{equation_yuri19}) used in the proof of Theorem \ref{theorem_qeq1}. We will keep the notation introduced earlier in this section.
\begin{lemma}\label{lemma_yqeqq}
 For any 2-bridge knot, the following identity holds in $\BH[\pi]$:
 \begin{equation}\label{eq_y1}
  YQ = Q
 \end{equation}
\end{lemma}
\begin{proof}
 We will actually prove a stronger identity: namely,
 \begin{equation}\label{eq_y2}
  \tilde w a^{-s/2} Q = Q
 \end{equation}
where $s = 4\sum_{n=1}^d e_n$ and $\tilde w$ is the word $w$ written backwards. To see that (\ref{eq_y2}) implies (\ref{eq_y1}), first note that there is an \emph{automorphism}\footnote{The existence of this automorphism is a consequence of the fact that 2-bridge knots are invertible \cite[Ex. 10.4]{BH95}.} of the fundamental group $\pi$ mapping $a\mapsto a^{-1}$, $b \mapsto b^{-1}$. This induces an automorphism $\beta$ of the Brumfiel-Hilden algebra $\BH[\pi]$ which fixes elements of $\BH^+$ and maps $(i,j,k) \mapsto (-i,-j,k)$. Applying $\beta$ to (\ref{eq_y2}), we get $w^*a^{s/2}Q = Q$. Therefore, $w^*a^{s/2}Q = \tilde w a^{-s/2}Q$, which implies $Q = a^{-s/2}w \tilde w a^{-s/2}Q = w \tilde w a^{-s} Q = YQ$, since $w \tilde w$ commutes with $a$.

To prove (\ref{eq_y2}), we assume (without loss of generality) that $s \geq 0$ and 
\[
 v = b^{e_1}a^{e_2}\cdots b^{e_{d-1}} a^{e_d},\quad w = v \bar v,\quad \tilde w = a^{e_1}b^{e_2}\cdots b^{e_d} a^{e_d}b^{e_{d-1}}\cdots b^{e_1}
\]
Next, we introduce the following notation: for $e = \pm 1$, we write $x_e := x + ej/2$, so that 
\[
 a^{e_n} = x + e_n(i+j)/2 = x_{e_n} + e_ni/2,\quad b^{e_n} = x + e_n(j-i)/2 = x_{e_n}-e_ni/2
\]
for $n=1,2,\ldots,d$. Observe that we obviously have
\begin{equation}\label{eq_y45}
 x_{e_n}x_{e_m} = x_{e_m}x_{e_n},\quad ix_e = x_{-e}i
\end{equation}
Using the commutator relations (\ref{eq_y45}) and the fact that $iQ = 0$ in $\BH[\pi]$, we compute
\begin{eqnarray*}
 \tilde w a^{-s/2} Q &=& (x_{e_1} + e_1 i/2)(x_{e_2}-e_2i/2)\cdots(x_{e_d}-e_di/2)(x_{e_d}+e_di/2)(x_{e_{d-1}}-e_{d-1}i/2)\cdots \\
 &\,&\cdots (x_{e_1}-e_1i/2)(x_- - i/2)^{s/2}Q\\
 &=& x_{e_1}x_{e_2}\cdots x_{e_d}x_{e_d}x_{e_{d-1}}\cdots x_{e_1} x_-^{s/2}Q\\
 &=& x_+^{2N_+}x_-^{2N_-}x_-^{s/2}Q\\
 &=& x_+^{2N_+}x_-^{2N_+}Q\\
 &=& (x_+x_-)^{2N_+}Q
\end{eqnarray*}
where $N_+$ is the number of $(+1)$'s among the $\{e_1,\ldots,e_d\}$ and $N_-$ is the number of $(-1)$'s, so that $N_+-N_- = \sum_{n=1}^d e_n = s/4$. Finally, note that
\[
 x_+x_- = (x+j/2)(x-j/2) = x^2-J/4 = 1+I/4
\]
This shows $(x_+x_-)Q = Q$, and therefore $\tilde w a^{-s/2}Q = (x_+x_-)^{2N_+}Q = Q$.
\end{proof}
\begin{remark}
 The proof of Lemma \ref{lemma_yqeqq} shows that we also have $v a^{-s/4}Q = Q$ in $\BH[\pi]$.
\end{remark}

To prove Conjecture \ref{conj_unknotsubmodule}, we note that by Theorem \ref{theorem_BHA49} (c.f. also \cite[Prop A.4*.10]{BH95}) the element $Q \in \BH[\pi]$ has the following form:
\[
 Q = a_0 + a_1 I + \cdots + a_{d-1}I^{d-1} + (-1)^d I^d
\]
where $a_i \in \C[x^2]\subset \BH^+$. Hence every element of $\BH^+[X^{\pm 1}] \cong \C[X^{\pm 1},I]/(IQ)$ can be written (uniquely) in the form
\[
 u = u_0 + u_1I + \cdots + u_dI^d \,\,(\mathrm{mod}\,IQ)
\]
where $u_i \in \C[X^{\pm 1}]$. It follows that $\BH^+[X^{\pm 1}]Q\delta^{-1} = \C[X^{\pm 1}]Q\delta^{-1}$ in $\BH[\delta^{-1}]$. The module $M = \BH^+[X^{\pm 1}] + \BH^+[X^{\pm 1}]Q\delta^{-1}$ from (\ref{equation_yuri15}) is therefore free over $\C[X^{\pm 1}]$ of rank $d+1$; for a basis in $M$ we can take $\{1,I,\ldots,I^{d-1},Q\delta^{-1}\}$. Now by Lemma \ref{lemma_yqeqq}, $M_0 := \C[X^{\pm 1}]Q\delta^{-1} \subset M$ is a $\C[X^{\pm 1},Y^{\pm 1}]\rtimes \Z_2$-submodule of $M$ which is isomorphic to the sign representation. 
The corresponding $A_1^{\Z_2}$-module $\e M_0$ is thus isomorphic to the skein module of the unknot, which implies Conjecture \ref{conj_unknotsubmodule} for two-bridge knots at the symmetric level (when $q=-1$). 

Now, for $M_0$ as above, consider the quotient $\bar M := M / M_0$ and identify
\begin{equation}\label{eq_yuristar}
\bar M = \C[X^{\pm 1}] 1 \oplus \cdots \oplus \C[X^{\pm 1}] I^{d-1} 
\end{equation}
so that $\bar M^\loc = \C(X^{\pm 1}) 1 \oplus \cdots \oplus \C(X^{\pm 1}) I^{d-1} $. To prove Conjecture \ref{conj_5param} we need to show that the operators
\begin{eqnarray*}
 T_0 &=& -t_1sY + (\bar t_1X+\bar t_2)\delta^{-1}(1+sY)\\
 T_1 &=& t_3s + (\bar t_3 X + \bar t_4)\delta^{-1}(1-s)
\end{eqnarray*}
preserve the subspace $\bar M \subset \bar M^\loc$ for all values $t_1,\ldots,t_4$. For this, it suffices to check that $U = \delta^{-1}(1+sY)$ and $U_1 = \delta^{-1}(1-s)$ preserve (\ref{eq_yuristar}). The inclusion $U\bar M \subset \bar M$ follows from (\ref{equation_yuri18}) by the same argument as in the proof of Theorem \ref{theorem_qeq1}, and $U_1\bar M \subset \bar M$ follows from the fact that $s$ acts trivially on the basis vectors $\{1,I,\ldots,I^{d-1}\}$. Thus Conjecture \ref{conj_5param} follows.

We conclude this section by exhibiting an interesting relation between the polynomial $Q$ and the classical Alexander polynomial $\Delta_K(t)$ of a 2-bridge knot. First, we observe that for any knot group $\pi = \pi_1(S^3\setminus K)$ and for any complex reductive group $G$, there is a natural map
\begin{equation}\label{eq_yy1}
 f: \T \stackrel \sim \to \Rep(\pi, \T) \hookrightarrow \Rep(\pi,G) \to \chr(\pi,G)
\end{equation}
where $\T \subset G$ is a maximal torus of $G$ (the first arrow in the definition of $f$ is an isomorphism induced by the abelianization map $\pi \to \pi/[\pi,\pi] = \Z$ of $\pi$). It is easy to see that $f$ factors through the quotient by $W$ so that $f:\T / W \to \chr(\pi,G)$. Hence, by dualizing (\ref{eq_yy1}) we get a map of commutative algebras
\begin{equation}\label{eq_yy2}
 f_*: \ochar(\pi,G) \to \O(\T)^W \subset \O(\T)
\end{equation}

If $G = \SL_2(\C)$ and $\pi$ is the fundamental group of a 2-bridge knot, we can identify $\ochar(\pi,G) = \C[x,I]/(IQ)$ and $\O(\T) = \C[X^{\pm 1}]$ using (\ref{eq_trmap}). With this identification, the map (\ref{eq_yy2}) is given by
\[
 f_*(x) = (X+X^{-1})/2,\quad f_*(I) = 0
\]
A direct calculation (similar to the one in Lemma \ref{lemma_yqeqq}) shows that
\[
 f_*(Q) = X^{-s/2}\Big[1-X^{2e_1} + X^{2(e_1+e_2)}-X^{2(e_1+e_2+e_3)} + \cdots + X^{2(e_1+\cdots + e_{p-1})}\Big]
\]
where $s = 4 \sum_{i=1}^d e_i$. The expression in the right-hand side coincides with a known formula for the Alexander polynomial $\Delta_K(t)$ of a 2-bridge knot evaluated at $t=X^2$ (see, e.g. \cite[Thm. 1.2(1)]{Fuk05} or \cite{Min82}). Thus, we conclude
\begin{proposition}
 For any 2-bridge knot, the image of $Q$ under map (\ref{eq_yy2}) is equal to $\Delta_K(X^2)$.
\end{proposition}
\begin{remark}
 The relation between $Q$ and the Alexander polynomial $\Delta_K(t)$ was observed in \cite{BH95} (see loc. cit., Example $A.8^*.13$). However, this relation is stated in \cite{BH95} in purely algebraic terms, without referring to the map (\ref{eq_yy2}), and the proof in \cite{BH95} is quite different from ours.
\end{remark}

\subsection{Conjecture \ref{mainconjecture} for 2-bridge knots for an arbitrary $q$}\label{sec_qeqq_2bridge}
Technically, we proved Conjectures \ref{mainconjecture}, \ref{conj_unknotsubmodule}, and \ref{conj_5param} for the module $M \subset \BH[\delta^{-1}]$ defined in (\ref{equation_yuri15}), which imply symmetric versions of these conjectures for skein modules of 2-bridge knot complements (cf. Corollary \ref{cor_qeq1}). However, we believe the following is true:

\begin{conjecture}\label{conj_q1skeinmodule}
 The module $M$ defined in (\ref{equation_yuri15}) is the $q=- 1$ specialization of the nonsymmetric skein module of the 2-bridge knot $K$. 
\end{conjecture}

Using the results of Section \ref{sec_qeqqskeinmodules}, it is easy to check that this conjecture is true for the trefoil and the figure eight knot. We now show that Conjecture \ref{conj_q1skeinmodule} implies Conjecture \ref{mainconjecture} for arbitrary $q$. Let $U = (1-q^2X^2)^{-1}(1-s\yy) \in D_q$.
\begin{theorem}\label{theorem_q1impliesqq}
 Let $N$ be a module over $A_q\rtimes \Z_2$ which is free and finitely generated over $\C[X^{\pm 1}]$, and suppose $UN \subset N$ for $q=\pm 1$. Then $UN \subset N$ for arbitrary $q$.
\end{theorem}
\begin{proof}
 Pick an identification of $\C[X^{\pm 1}]$-modules $N \cong \C[X^{\pm 1}]\otimes_\C V$ for some finite dimensional vector space $V$. The action of $\yy$ and $s$ on $N$ are completely determined by the matrices $A(X),B(X) \in \End_{\C[X^{\pm 1}]}(N)$ defined by 
 \[
  A(X)\cdot v := \yy\cdot (1\otimes v),\quad B(X)\cdot v := s\cdot (1\otimes v)
 \]
Define operators $S,P: N \to N$ via the formulas $S\cdot (f(X)\otimes v) = f(X^{-1})\otimes v$ and $P\cdot (f(X)\otimes v) = f(q^{-2}X)\otimes v$. Then the action of $\yy$ and $s$ on $N$ can be written in terms of the operators $P,S$ as follows:
\[
 \yy = A(X)P,\quad s = B(X)S
\]
where the equalities are inside $\End_\C(N)$. Furthermore, the operators $X$, $\yy$, and $s$ satisfy the relations of $A_q\rtimes \Z_2$, and the relation $\yy s \yy s = 1$ implies the identity
\begin{equation}\label{eq_abidentity}
 B(X)A(X^{-1})B(q^{-2}X^{-1})A(q^2X) = \mathrm{Id}
\end{equation}
Let $C(q) \in \End_\C(V)$ be the matrix $C(q) = B(q^{-1})A(q)$. By Remark \ref{remark_signproblems}, the condition $UN \subset N$ is equivalent to the conditions $C(q) = \mathrm{Id}$ and $C(-q) = \mathrm{Id}$.

Now by assumption we have $C(1) = C(-1) = \mathrm{Id}$, and if we substitute $X = q^{-1}$ into equation (\ref{eq_abidentity}), we get $C(q)^2 = \mathrm{Id}$. Now let $q = e^z$ and write $C(q) = C(e^z) = \sum_i C_i z^i$. Expanding the equation $C(q)^2 = \mathrm{Id}$ in powers of $z$, we obtain $C(e^z)^2 = C_0^2 + C_0C_1z + \cdots = \mathrm{Id}$, and induction on powers of $z$ shows that $C(e^z) = \mathrm{Id}$. This shows $C(q) = \mathrm{Id}$, and a similar argument shows $C(-q) = \mathrm{Id}$, which completes the proof.
\end{proof}

\begin{corollary}\label{corollary_qeqq}
 Conjecture \ref{conj_q1skeinmodule} implies Conjecture \ref{mainconjecture} for 2-bridge knots (for an arbitrary $q$).
\end{corollary}
\begin{proof}
 Indeed, if $M$ is the correct specialization of $  N = \hat{K}_q $
at $ q = - 1 $, then the assumption $UN \subset N$ of Theorem \ref{theorem_q1impliesqq}  holds by
(the proof of) Theorem \ref{theorem_qeq1} for $q=-1$. A theorem of Barret \cite{Bar99} shows that the $q=1$ and $q=-1$ skein modules are isomorphic, which shows $UN \subset N$ for $q=1$. The fact that $N$ is free and finitely generated over $\C[X^{\pm 1}]$ was proved in \cite{Le06}, and Theorem \ref{theorem_q1impliesqq} therefore implies $UN \subset N$ for arbitrary $q$, which implies Conjecture \ref{mainconjecture}.
\end{proof}

\section{Examples of skein modules of knot complements}\label{sec_qeqqskeinmodules}
In this section we give an explicit description of the nonsymmetric skein modules for the unknot, $(2,2p+1)$ torus knots, and the figure eight knot. In the process we prove Conjecture 2 for these knots:

\begin{theorem}\label{thm_conjecture2}
 If $K$ is the figure eight knot or any $(2,2p+1)$ torus knot, then $K_q(\mathrm{unknot})$ is a submodule of $K_q(S^3\setminus K)$.
\end{theorem}

\begin{remark}\label{remark_unknotsubmodule}
   This theorem can be viewed as a quantization of the map $\phi$ from (\ref{equation_unknotembeddingqeq1}). In particular, in examples the image of the empty link is a nontrivial element of the skein module of the knot complement, and this quantizes the fact that $\phi(1) = B_K$ (see Remark \ref{remark_lm1dividesA}).
  
  We also remark that the natural algebra map $\C[m^{\pm 1},l^{\pm 1}] / A_K \twoheadrightarrow \C[m^{\pm 1},l^{\pm 1}] / A_U$ from Remark \ref{remark_lm1dividesA} does \emph{not} quantize for $(2,2p+1)$ torus knots or the figure eight knot. In particular, for the trefoil, Lemma \ref{lemma_trefoilseq} shows that if $q$ is not a root of unity, then the unknot submodule of $M := K_q(\mathrm{trefoil})$ is the unique submodule, which shows that $M$ has a unique quotient. This quotient is clearly not isomorphic to the skein module of the unknot.
\end{remark}

To simplify notation we will divide the proof of Theorem \ref{thm_conjecture2} into separate subsections after first proving some useful technical lemmas. For the trefoil (i.e. the $(2,3)$ torus knot) and the figure eight knot, Gelca and Sain have given complete calculations of the symmetrized module $K_q(S^3\setminus K)$ in \cite{Gel02} and \cite{GS04}, respectively. Using these calculations we describe the corresponding $A_q\rtimes \Z_2$-modules explicitly. For the $(2,2p+1)$ torus knots, Gelca and Sain gave only partial computations of the module structure of $K_q(S^3\setminus K)$ in \cite{GS03}. We complete their computations to fully determine the module structure of the submodule of $K_q(S^3\setminus K)$ generated by the empty link and describe the corresponding $A_q\rtimes \Z_2$-module.

\begin{remark}
 The calculations in this section are lengthy, and the reader might worry about errors with signs or powers of $q$. However, a strong ``consistency check'' is available - one can use Lemma \ref{lemma_liftedcoloredJones} together with the module structures described in this section to give explicit computations of the colored Jones polynomials, and then compare these to known results. For the $(2,2p+1)$ knots this has been done in \cite[Lemma 6.4.7]{Sam12} for all $n$, and for the figure eight this has been done for many small $n$. (See also the explicit computations in the appendix.)
\end{remark}

We now establish a few technical lemmas. We recall that if $m, l \in K_q(T^2)$ are the meridian and longitude, respectively, and $z$ is the $(1,1)$ curve, then under the embedding $K_q(T^2) \hookrightarrow A_q\rtimes \Z_2$ we have
\[
 m \mapsto x := X+X^{-1},\quad l \mapsto y := Y+Y^{-1},\quad z \mapsto q^{-1}(XY+X^{-1}Y^{-1})
\]
As an algebra, the image of $K_q(T^2)$ is generated by $x$, $y$, and $z$. We now prove a lemma that is useful for constructing isomorphisms of $A_q^{\Z_2}$-modules - it essentially says that the $A_q^{\Z_2}$-module structure of $M$ is determined by the action of $y$ and $z$ on a $\C[x]$-basis for $M$.

\begin{lemma}\label{lemma_invsubalgebraiso}
 Suppose that $M$ and $N$ are modules over $A_q^{\Z_2}$, and that as a $\C[x]$-module,  $M$ is generated by elements $\{m_i\}\subset M$. Furthermore, suppose that $f:M \to N$ is an isomorphism of $\C[x]$-modules that satisfies $f(ym_i) = yf(m_i)$ and $f(zm_i) = zf(m_i)$. Then $f$ is an isomorphism of $A_q^{\Z_2}$-modules.
\end{lemma}
\begin{proof}
 From (\ref{relationsforB'}), the elements $x,y,z \in A_q^{\Z_2}$ satisfy the commutation relations
 \[
[x,y]_q = (q^2-q^{-2})z,\quad [z,x]_q = (q^2-q^{-2})y,\quad [y,z]_q = (q^2-q^{-2})x
 \]
(where we have used the notation $[a,b]_q := qab - q^{-1}ba$). An arbitrary element of $M$ can be written as $m = \sum_{i=1}^n p_i(x) m_i$, and using the $\C[x]$-linearity of $f$ and the commutation relations, powers of $x$ in the expressions $ym$ and $zm$ can inductively be moved to the left. This shows that $f(ym) = yf(m)$ and $f(zm) = zf(m)$ for arbitrary $m \in M$, which completes the proof.
\end{proof}

We also give a lemma which is useful for explicitly constructing modules over $A_q\rtimes \Z_2$.  Let $M = \C[X^{\pm 1}]\otimes_\C V$ and define operators $S, P: M \to M$ via the formulas
\begin{equation}\label{equation_diagonalaction}
S\cdot \left(f(X) \otimes v\right) := f(X^{-1})\otimes v,\quad P\cdot \left(f(X)\otimes v\right) := f(q^{-2}X)\otimes v 
\end{equation}
Then the operators $X$, $S$, and $P$ satisfy the relations of $A_q\rtimes \Z_2$ (and with these operators, $M$ is a direct sum of copies of the standard polynomial representation of $A_q\rtimes \Z_2$). Furthermore, if $A(X) \in \End_{\C[X^{\pm 1}]}(M)$, then $SA(X) = A(X^{-1})S$ and $PA(X) = A(q^{-2}X)P$ (where the equalities are inside $\End_\C(M)$).
\begin{lemma}\label{lemma_aqmodule}
  Suppose that $A(X),B(X)\in \End_{\C[X^{\pm 1}]}(M)$ satisfy
\[
 B(X^{-1})B(X) = \mathrm{Id}_M,\quad A(X)B(q^{-2}X)A(q^2X^{-1}) = B(X)
\]
Then the operators $X$, $s := B(X)S$, and $Y := A(X)P$ endow $M$ with the structure of an $A_q\rtimes \Z_2$-module.
\end{lemma}
\begin{proof}
The relation $XY=q^2YX$ follows from the fact that $X$ commutes with $A(X)$ and the fact that $XP=q^2PX$. Since $X$ commutes with $B(X)$, we see that $XsX = s$. The relation $s^2=1$ follows from the fact that $B(X^{-1})B(X)=1$. For the final relation, we compute
\[
  YsY = A(X)PB(X)SA(X)P = A(X)B(q^{-2}X)A(q^2X^{-1})PSP = A(X)B(q^{-2}X)A(q^2X^{-1})S = s
  \]
\end{proof}
In the following sections we will frequently use the element 
\[\delta := X-X^{-1} \in A_q\rtimes \Z_2\] 
We will also use the Chebyshev polynomials $S_n,T_n\in \C[x]$, which are defined by
\begin{equation*}
\begin{array}{lll}
 S_0 = 1, & S_1 = x, &S_{n+1} = xS_n - S_{n-1}\\
 T_0 = 2, & T_1 = x, &T_{n+1} = xT_n - T_{n-1}
\end{array}
\end{equation*}
\begin{lemma}\label{chebyshevidentities}
 The Chebyshev polynomials satisfy the identities
 \[
  (X-X^{-1}) S_n(X+X^{-1}) = X^{n+1}-X^{-n-1}\quad \textrm{ and }\quad T_n(X+X^{-1}) = X^n+X^{-n}.
 \]
\end{lemma}

\subsection{The unknot}
Let $K \subset S^3$ be the unknot, so that $S^3 \setminus K$ is a solid torus. Then $K_q(S^3\setminus K) \cong \C[u]1_K$, where $1_K \in K_q(S^3\setminus K)$ is the empty link. The action of $K_q(T^2)$ on $K_q(S^3\setminus K)$ is given by
\begin{eqnarray}\label{theunknot}
 x\cdot f(u)1_K &=& uf(u)1_K\notag\\
 y\cdot 1_K &=& (-q^2-q^{-2})1_K\\
 z\cdot 1_K &=& -q^{-3}u1_K\notag
\end{eqnarray}
(The image of the longitude inside the solid torus is contractible, and the $-q^{-3}$ factor in the third formula comes from the framing of the image of the $(1,1)$ curve inside the solid torus.)

We give the $\C[X^{\pm 1}]$-module $\hat M := \C[X^{\pm 1}]$ the structure of an $A_q\rtimes \Z_2$-module via the formulas
\begin{equation}\label{formula_liftedunknot}
 Y\cdot f(X) := -f(q^{-2}X),\quad s\cdot f(X) := -f(X^{-1})
\end{equation}
The module $\hat M$ is called the \emph{sign representation} of $A_q\rtimes \Z_2$. As a $\Z_2$-module, we have the decomposition $\hat M \cong \C[x] \oplus \C[x]\delta$. Since $s\cdot 1 = -1$, we see that $\e \hat M = \C[x]\delta$ as a $\C[x]$-module.

\begin{lemma}\label{unknotlift}
 The $\C[x]$-isomorphism $f:\e \hat M \to K_q(S^3\setminus K)$ defined by $f(\delta) = 1_K$ is an isomorphism of $A_q^{\Z_2}$-modules. In particular, the skein module of the unknot is the (symmetric) sign representation.
\end{lemma}
\begin{proof}
 By Lemma \ref{lemma_invsubalgebraiso}, the following computations show the claim:
 \begin{eqnarray*}
  y\cdot \delta &=& (Y+Y^{-1})(X-X^{-1}) = -(q^2-q^{-2})\delta\\
  z\cdot \delta &=& q^{-1}(XY+X^{-1}Y^{-1})(X-X^{-1}) = -q^{-3}(X+X^{-1})\delta
 \end{eqnarray*}
\end{proof}

\subsection{The trefoil}
Let $M = K_q(S^3 \setminus K)$ be the skein module of the complement of the trefoil knot. In \cite{Gel02}, Gelca showed that $M$ is free and finitely generated as a module over the meridian subalgebra $\C[x]$. (This was also shown in \cite{Le06} and \cite{BL05}.) Gelca's generators are $w', 1_K \in M$, where $1_K$ is the empty link and $w'$ is the loop labelled $d$ in Figure \ref{fig_22pp1}. However, the vector $w:=w'+q^{-2}1_K$ generates a proper submodule of $M$, so it is easiest to describe the $A_q^{\Z_2}$-module structure in terms of the basis $w, 1_K$. Translating his formulas for the action of $y$ and $z$ into this basis, we get
\begin{eqnarray*}
y\cdot w &=& -(q^2+q^{-2})w\\
z\cdot w &=& -q^{-3}S_1(x)w\\
y\cdot 1_K &=& (q^6S_4(x)-q^2)w + q^6T_6(x)1_K\\
z\cdot 1_K &=&  q^5S_3(x)w + q^5T_5(x)1_K
\end{eqnarray*}
(This follows from \cite{Gel02}, Lemma 3 and Lemma 7 for $q=0$ and $q=1$. The parameter $q$ in \cite{Gel02} is an integer,  unrelated to our $q$, and the parameter $t$ is our $q$.)

\begin{remark}
 From these formulas it is clear that $w$ generates an $A_q^{\Z_2}$-submodule of $M$. Comparing to formula (\ref{theunknot}), we see that this submodule is isomorphic to the skein module of the unknot. (The isomorphism is determined by sending the empty link to $w \in M$.)
\end{remark}

To describe the nonsymmetric module, we first define a $\C[X^{\pm 1}]$-module $\hat M$ via
\begin{equation}\label{equation_trefoilbasis}
 \hat M \cong \C[X^{\pm 1}]u \oplus \C[X^{\pm 1}]v
\end{equation}
(With this notation, $v$ will be identified with the empty link.) Let $P,S:\hat M \to \hat M$ be the operators given by (\ref{equation_diagonalaction}). We then define an $A_q\rtimes \Z_2$-module structure on $\hat M$ via the following matrices (written with respect to the ordered basis $u,v$):
\begin{equation}\label{equation_trefoilmatrices}
 s = \left[\begin{array}{cc}-1&0\\0&1\end{array}\right]S, \quad Y = \left[\begin{array}{cc}-1& q^2X^{-1}-q^6X^{-5}\\0&q^6X^{-6}\end{array}\right]P
\end{equation}
It is easy to check that these matrices satisfy the conditions of Lemma \ref{lemma_aqmodule}, which implies they define a representation of $A_q\rtimes \Z_2$.  Explicitly, the action of $Y$ on $\hat M$ is given by the formulas
\begin{equation}\label{eq_trefoilexplicit}
 Y\cdot u = -u,\quad Y\cdot v = \left(q^2X^{-1} - q^6X^{-5}\right)u + q^6X^{-6}v
\end{equation}

Since $s$ acts diagonally in the basis $(u,v)$, there is a $\C[x]$-module isomorphism $\e \hat M = \C[x]\delta u \oplus \C[x] v$.
\begin{lemma}\label{lemma_symmetrictrefoil}
 The $\C[x]$ isomorphism $f:\e \hat M \to M$ determined by $f(\delta u) = w$ and $f(v) = 1_K$ is an isomorphism of $A_q^{\Z_2}$-modules.
\end{lemma}
\begin{proof}
 Lemma \ref{lemma_invsubalgebraiso} reduces this to several straightforward computations. For example, 
 \begin{align*}
  y\cdot v &= (Y+Y^{-1})v = (Y + sYs)v = (1+s)Yv\\
  &= (1+s)\left[(q^2X^{-1}-q^6X^{-5})u + q^6X^{-6}v\right]\\
  &= \left[q^2(X^{-1}-X)-q^6(X^{-5}-X^5)\right]u + q^6(X^{-6}+X^{6})v\\
  &= \left[-q^2 + q^6S_4(x)\right]\delta u + q^6T_6(x)v
 \end{align*}
 (In the last step we used Lemma \ref{chebyshevidentities}.)
\end{proof}

The next lemma describes the structure of $M$ in terms of standard (induced) modules of $A_q\rtimes \Z_2$. Let $\tau: A_q\rtimes \Z_2 \to A_q\rtimes \Z_2$ denote the automorphism
\[
 \tau(X) = X,\quad \tau(s) = s,\quad \tau(Y) = q^{-1}XY
\]
Let $V^-$ be the sign representation of $A_q\rtimes \Z_2$ (i.e. the nonsymmetric skein module of the unknot), and let $V^+$ be the standard polynomial representation. 
\begin{lemma}\label{lemma_trefoilseq}
 $M$ admits a decomposition into a nonsplit exact sequence
 \[
  0 \to V^- \to M \to \tau^{-6}(V^+) \to 0
 \]
where $\tau^{N}(V^+)$ is the twist of $V^+$ by $\tau^N$. If $q$ is not a root of unity, then $V^-$ is the unique nontrivial submodule of $M$.
\end{lemma}
\begin{proof}
 The existence of this short exact sequence is clear because in our chosen basis (\ref{equation_trefoilbasis}) the operators $S$, $P$, $s$, $Y$, and $X$ all act by upper-triangular matrices. We have already identified the submodule with $V^-$, and the identification of the quotient is clear by (\ref{equation_trefoilmatrices}). If this sequence were split, there would exist an $m \in M$ with $m = v + f(x)u$ and $(q^6X^{-6} - Y)\cdot m = 0$ (since this equation holds in the quotient). However, this equation implies $q^6X^{-6}f(X) + f(q^{-2}X) = q^6X^{-5} - q^{2}X^{-1}$, and this is impossible because the total degree of the left hand side is at least $6$, while the total degree of the right hand side is $4$.
 
 If $q$ is not a root of unity, then in the standard polynomial representation $V^+$ the element $X^k$ is a $\C$-basis for the kernel of the operator $Y - q^{-2k}$. This implies that $V^+$ is simple, which implies $V^-$ and $\tau(V^+)$ are simple. Then the final claim follows from the general fact that a nonsplit extension of two simple modules has a unique nontrivial submodule.
\end{proof}

\subsection{$(2,2p+1)$ torus knots}
In this subsection we recall the calculations of Gelca and Sain \cite{GS03} for the $(2,2p+1)$ torus knot $K_p$. We also extend their calculations to completely determine the module structure of the submodule of $K_q(S^3\setminus K_p)$ generated by the empty link $1_K \in K_q(S^3\setminus K_p)$, and we give an explicit presentation of the nonsymmetric version of this module. 

In \cite{GS03}, the authors proved that there is an isomorphism of $\C[x]$-modules
\[
K_q(S^3\setminus K_p) \cong \bigoplus_{i=0}^p \C[x] v^i
\]
Here $v$ is the loop labelled $d$ in Figure \ref{fig_22pp1}, and $v^i$ is $i$ parallel copies of $v$ (and $v^0 = 1_K$ is the empty link, by convention). We define the element
\[
w := S_{p-1}(v) + q^{-2}S_p(v) \in K_q(S^3\setminus K_p)
\]
Then Gelca and Sain prove the following:
\begin{lemma}[\cite{GS03} Prop 4.4]\label{gelcalemma}
\begin{eqnarray}\label{22pplus1formulas1}
y\cdot 1_K &=& q^{4p+2}T_{4p+2}1_K + (-1)^{p+1}q^{2p+2}(q^2S_{2p+2}-q^{-2}S_{2p-2})w\notag\\
z\cdot 1_K &=& q^{4p+1}T_{4p+1}1_K + (-1)^{p+1}q^{2p+2}(qS_{2p+1}-q^{-3}S_{2p-3})w
\end{eqnarray}
\end{lemma}

\begin{figure}
\begin{center}
\includegraphics[scale=.6]{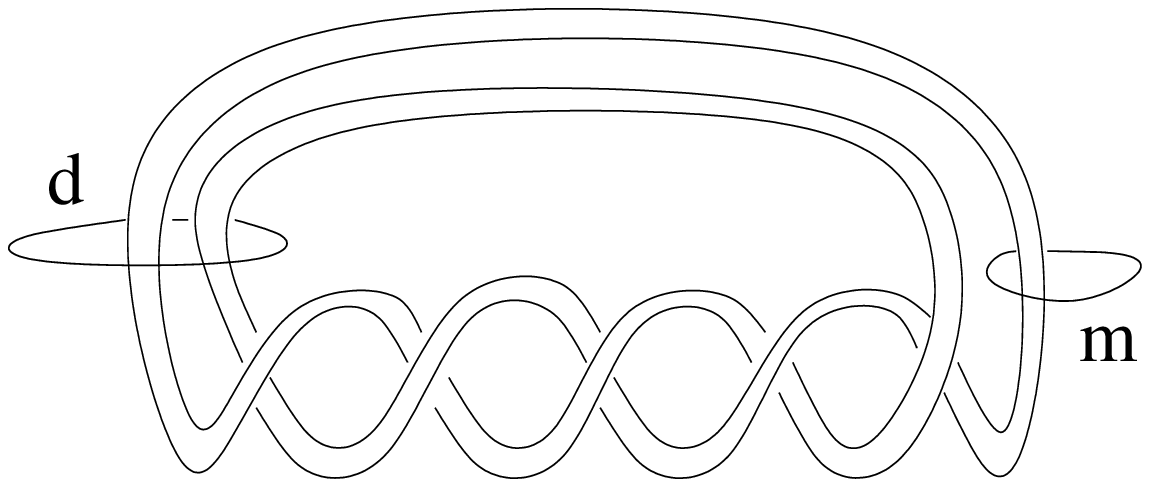}
\caption{Generators for the $(2,2p+1)$ knot}\label{fig_22pp1}
\end{center}
\end{figure}


We extend their computations with the following lemma:
\begin{lemma}\label{lemma_22pplus1sym}
 The action of $K_q(T^2)$ on $w \in K_q(S^3\setminus K_p)$ is determined by 
 \begin{eqnarray}\label{22pplus1formulas2}
  y\cdot w &=& -(q^2+q^{-2})w \label{equation_22pp1_1}\\ 
  z\cdot w &=& -q^{-3}xw\notag 
 \end{eqnarray}
\end{lemma}
\begin{proof}
(In \cite{GS03}, the convention for the isomorphism of Theorem \ref{fg00} is different from ours. In particular, in their convention the image of $(m,n)_T \in K_q(T^2)$ in $A_q^{\Z_2}$ is $q^{mn}(X^mY^{-n} + X^{-m}Y^n)$. In this proof only, we will follow their convention.) We first note that proving equations (\ref{equation_22pp1_1}) is equivalent to proving the following:
 \begin{eqnarray}
  (1,-4p-2)_T\cdot w &=& -q^{4p-2}\left( q^{-2}S_{4p+2} - q^2 S_{4p}\right)w \label{equation_22pp1_2}\\
  (1,-4p-1)_T\cdot w &=& -q^{4p-1}\left(q^{-2}S_{4p+1} - q^2 S_{4p-1}\right)w\notag
 \end{eqnarray}
 To see this, we first note that the proof of Lemma \ref{lemma_22pp1lift} (along with a short calculation) shows that both sets of equations (\ref{equation_22pp1_1}) and (\ref{equation_22pp1_2}) hold inside the $A_q^{\Z_2}$-module $\e \hat M$ (which is defined via the operators in (\ref{equation_22pp1liftops})). Then an appropriate Dehn twist of $T^2$ provides an automorphism of $A_q^{\Z_2}$ that fixes $x$ and sends $y,z$ to the elements $(1,-4p-1)_T$ and $(1,-4p-2)_T$, respectively. Therefore, the proof of Lemma \ref{lemma_invsubalgebraiso} shows that either set of equations (together with Lemma \ref{gelcalemma}) completely determine the $A_q^{\Z_2}$-module structure of $\e \hat M$. (We note that $\e \hat M$ is actually an $A_q^{\Z_2}$-module by Lemma \ref{lemma_22pp1_liftismod}.)

To prove the first equation of (\ref{equation_22pp1_2}) we follow the strategy of \cite[Prop. 4.1]{GS03}. Namely, if we remove from $S_3\setminus K_p$ a regular neighborhood of the M\"obius band that is bounded by the knot, then the resulting 3-manifold is a solid torus which contains both the $(1,-4p-2)_T$ curve and the element $w$ in its interior. Therefore, the left hand side of the first equation in (\ref{equation_22pp1_2}) can be simplifed inside the skein module of the solid torus as follows:
\begin{equation}\label{equation_22pp1_whenwillitend}
 (2p+1,-2)_T\cdot \left[ S_{p-1}(v) + q^{-2}S_p(v)\right] = q^{-4p-2}\left[q^{-4p}S_{3p}(v) + q^{-4p-6}S_{3p+1}(v) - q^{4p}S_p(v) - q^{4p+2}S_{p-1}(v)\right]
\end{equation}
(The image in the solid torus of the $(1,-4p-2)_T$ curve on the original torus is the same as the image of the $(2p+1,-2)_T$ curve on the torus which bounds the solid torus, and the $v$ curve in the original knot complement is the image of the longitude of the boundary of the solid torus.) Then the right hand side of the first equation in (\ref{equation_22pp1_2}) can be simplified using \cite[Thm. 3.1]{GS03}, and this agrees with the right hand side of equation (\ref{equation_22pp1_whenwillitend}). (To make the powers of $q$ match exactly, note that the rightmost parenthesized expression of \cite[Thm. 3.1]{GS03} is $qw$.) This completes the proof of the first equation of (\ref{equation_22pp1_2}).

The proof of the second equation in (\ref{equation_22pp1_2}) is more lengthy, so we include a sketch and leave the details for the interested reader. The strategy is to follow the proof of \cite[Prop. 4.3]{GS03}. To prove this, the authors define two sequences of skeins $a_k, b_k \in K_q(S^3\setminus K_p)$ so that $q^{-6p}a_{2p+1} = (1,-4p-1)_T\cdot 1_K$. These sequences of skeins can be modified in a straightforward way to obtain sequences $a'_k,b'_k$ with $q^{-6p}a'_{2p+1} = (1,-4p-1)_T\cdot w$. The authors then show that the sequence $b_k$ satisfies a second order recurrence that can be solved explicitly, and the sequence $a_k$ satisfies a first order recurrence with an inhomogenous term depending on $b_k$ which can also be solved explicitly. Then \cite[Thm. 3.1]{GS03} allows the simplification of this explicit expression to obtain the second formula of Lemma \ref{gelcalemma}. In a similar way, the sequences $a'_k$ and $b'_k$ can be written explicitly to obtain the second formula of (\ref{equation_22pp1_2}).
\end{proof}

We define the submodule $ M \subset K_q(S^3\setminus K_p)$ by
\[
M := K_q(T^2)\cdot 1_K
\]

\begin{corollary}
 We have equality of subspaces 
 $M = K_q(T^2)\cdot 1_K = \C[x] 1_K + \C[x]w$.
\end{corollary}
\begin{proof}
 This is straightforward from the formulas (\ref{22pplus1formulas1}) and (\ref{22pplus1formulas2}).
\end{proof}

We now describe the $A_q\rtimes \Z_2$-module $\hat M$ that satisfies $\e \hat M \cong M$. As in the case of the trefoil, we define the $\C[X^{\pm 1}]$-module structure first: 
\[ \hat M := \C[X^{\pm 1}]u \oplus \C[X^{\pm 1}] v\]
(In this notation, the empty link is identified with $v$.) We define operators $P, S:\hat M \to \hat M$ using formula (\ref{equation_diagonalaction}). Then we define the actions of $Y$ and $s$ via the following operators (which are written with respect to the ordered basis $(u,v)$):
\begin{equation}\label{equation_22pp1liftops}
 s = \left[\begin{array}{cc}-1&0\\0&1\end{array}\right]S, \quad Y = \left[\begin{array}{cc}-1& (-1)^{p}q^{2p+4}(X^{-2p-3}-q^{-4} X^{-2p+1})\\0 &q^{2(2p+1)}X^{-2(2p+1)}\end{array}\right]P
\end{equation}
\begin{lemma}\label{lemma_22pp1_liftismod}
The formulas (\ref{equation_22pp1liftops}) give $\hat M$ an $A_q\rtimes \Z_2$-module structure.
\end{lemma}
\begin{proof}
This follows from a straightforward calculation and Lemma \ref{lemma_aqmodule}. 
\end{proof}
Explicitly, the action of $Y$ on $\hat M$ is given by 
\[
 Y\cdot u = -u,\quad Y\cdot v = \left[ (-1)^pq^{2p+4}(X^{-2p-3}-q^{-4}X^{-2p+1})\right]u + q^{2(2p+1)}X^{-2(2p+1)}v
\]

As before, $s$ acts diagonally, which gives a decomposition $\e \hat M = \C[x]\delta u \oplus \C[x]v$.
\begin{lemma}\label{lemma_22pp1lift}
 The $\C[x]$-module isomorphism $f:\e \hat M \to M$ given by $f(\delta u) = w$ and $f(v) = 1_K$ is an isomorphism of $A_q^{\Z_2}$-modules.
\end{lemma}
\begin{proof}
 The element $u \in \hat M$ generates a proper submodule of $\hat M$, and it is clear from Lemma \ref{unknotlift} that the restriction of $f$ to this submodule is an isomorphism. We then compute
\begin{eqnarray*}
 y\cdot v &=& (Y+sYs)v = (1+s)Yv\\
 &=& (1+s)\left[(-1)^{p}q^{2p+4}(X^{-2p-3}-q^{-4}X^{-2p+1})u + q^{4p+2}X^{-4p-2}v\right]\\
 &=& (-1)^{p}q^{2p+4}(X^{-2p-3}-X^{2p+3}+q^{-4}X^{2p-1}-q^{-4}X^{-2p+1})u\\
 &\,& + q^{4p+2}(X^{-4p-2}+X^{4p+2})v\\
 &=& (-1)^{p+1}q^{2p+4}(-q^{-4}S_{2p-2}+S_{2p+2})\delta u + q^{4p+2}T_{4p+2}v
\end{eqnarray*}
(In the last step we have used Lemma \ref{chebyshevidentities}.) This shows that $f(y\cdot v) = y\cdot f(v)$. A similar computation shows $f(z\cdot v) = z\cdot f(v)$, and an application of Lemma \ref{lemma_invsubalgebraiso} completes the proof.
\end{proof}

\subsection{The figure eight}
Let $M = K_q(S^3 \setminus K)$ be the skein module of the complement of the figure eight knot. 
First we recall some facts from \cite{GS04} (translated into our notation). As $\C[x]$-modules, we have an isomorphism
\[
 M \cong \C[x]u \oplus \C[x]v \oplus \C[x]w
\]
Under this identification, the empty link is $u \in M$, and 
if $v', w' \in M$ are the loops labelled $y$ and $z$ (respectively) in Figure \ref{fig_fig8gens}, then $v = q^2 v' + u$ and $w = q^{-2}w' + u$. Gelca and Sain then give the following formulas to describe the module structure.

\begin{lemma}[\cite{GS04}]\label{actiononM}
The action of $y$ and $z$ on $M$ is determined by the formulas
\begin{eqnarray*}
 y\cdot u &=& (q^2+q^{-2})S_2u + (q^2S_2 + q^{-2})v + (q^2+q^{-2}S_2)w\\
 y\cdot v &=& (-q^6S_4 + q^2)u + (-q^6S_4+q^2S_2)v + (-q^6S_2-q^2)w\\
 y\cdot w &=& (q^{-2}-q^{-6}S_4)u + (-q^{-2}-q^{-6}S_2)v + (q^{-2}S_2-q^{-6}S_4)w\\
 z\cdot u &=& (qS_1+q^{-3}S_3)u + (q+q^{-3})S_1v + q^{-3}S_3w\\
 z\cdot v &=& -q^5S_3u+(qS_1-q^5S_3)v+(-q^5-q)S_1w\\
 z\cdot w &=& (q^{-3}S_1-q^{-7}S_5)u-q^{-7}S_3v + (q^{-3}S_3-q^{-7}S_5)w
 \end{eqnarray*}
\end{lemma}

\begin{figure}
\begin{center}
\includegraphics[scale=.6]{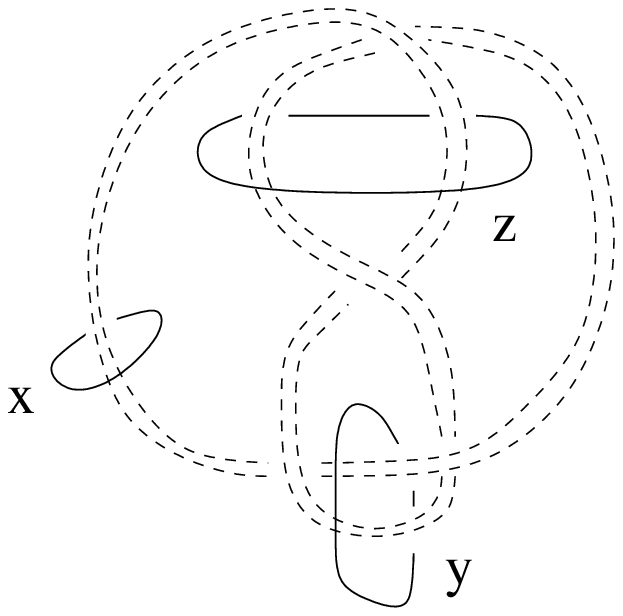}
\caption{Generators for the figure eight}\label{fig_fig8gens}
\end{center}
\end{figure}

As in the case of the trefoil, there is a proper submodule of $M$ which is isomorphic to the skein module of the unknot. In this case it is generated by the element 
 \[
  p := (x^2-3)u + v + w
 \]

 \begin{lemma}\label{peigenvec}
  We have equalities
  \[
   y\cdot p = (-q^2-q^{-2})p,\quad z\cdot p = -q^{-3}xp
  \] 
 \end{lemma}
\begin{proof}
This is an entertaining but lengthy computation which we omit.
\end{proof}

As before, it is convenient to describe the module structure of $M$ using the $\C[x]$-basis $p, u, v$.

\begin{lemma}\label{actiononquotient}
As $\C[x]$-modules, we have an isomorphism $M \cong \C[x]p \oplus \C[x]u \oplus \C[x]v$, and the actions of $y$ and $z$ in this basis are given by
\begin{eqnarray*}
 y\cdot p &=& -(q^2+q^{-2})p\\
 z\cdot p &=& -q^{-3}xp\\
 y\cdot u &=& (q^2+q^{-2}S_2)p + [-q^{-2}T_4+q^{-2}T_2+q^2T_0]u + (q^2-q^{-2})T_2 v\\
 z\cdot u &=& q^{-3}S_3p + (-q^{-3}T_5 + q^{-3}T_3+qT_1)u + (-q^{-3}T_3+qT_1)v\\
 y\cdot v &=& (-q^6S_2-q^2)p + (q^2-q^6)T_2u + (-q^6T_4+q^2T_2+q^2T_0)v\\
 z\cdot v &=& (-q^5-q)S_1p + (qT_3-q^5T_1)u + (-q^5T_3+2qT_1)v
\end{eqnarray*}
\end{lemma}
\begin{proof}
 This is a straightforward calculation.
\end{proof}

We now define the $A_q\rtimes \Z_2$-module $\hat M$ that satisfies $\e \hat M \cong M$. As a $\C[X^{\pm 1}]$-module we define
\[
 \hat M := \C[X^{\pm 1}] \otimes_\C \C\{ p', u,v\}
\]
(We have slightly abused notation by reusing the letters $u,v$, but this is justified by the fact that the inclusion $K_q(S^3\setminus K) \to \hat M$ identifies $u,v \in K_q(S^3\setminus K)$ with $u,v \in \hat M$. In particular, with this notation, the empty link is identified with $u$.) We use formula (\ref{equation_diagonalaction}) to define operators $S,P:\hat M \to \hat M$. As before, we will define the $A_q\rtimes \Z_2$-module structure on $\hat M$ using matrices (with respect to the ordered basis $p', u, v$). To write these matrices in a more compact form, we define the following polynomials in $\C[X^{\pm 1}]$:
\[
 a := -q^{-2}X^4+q^{-2}X^2+q^2, \quad b := -q^{-2}X^2+q^2X^{-2},\quad c := q^{-2}X^3 - q^2X^{-1}
\]

To shorten notation further in what follows, if $f \in \C[X^{\pm 1}]$, we write $f'(X) := f(q^2X^{-1})$ (note that $f'' = f$). We then define $s, Y \in \End_\C(M)$ as follows:
\begin{equation}\label{figeightops}
 s := \left[\begin{array}{ccc}-1&0&0\\0&1&0\\0&0&1\end{array}\right]S, \quad Y := \left[\begin{array}{ccc}-1& c&q^2c'\\0 &a&q^4b'\\0&b&a'\end{array}\right]P
\end{equation}

\begin{lemma}\label{lemma_figeightmod}
 The operators in (\ref{figeightops}) define a representation of $A_q\rtimes \Z_2$.
\end{lemma}
\begin{proof}
Let $B$ and $A(X)$ be the matrices in (\ref{figeightops}). The first condition of Lemma \ref{lemma_aqmodule} clearly holds, so we are reduced to checking that $A(X)BA(q^2X^{-1})=B$. 
We compute
\[
 \left[\begin{array}{ccc}-1& c&q^2c'\\0 &a&q^4b'\\0&b&a'\end{array}\right]B\left[\begin{array}{ccc}-1& c'&q^2c\\0 &a'&q^4b\\0&b'&a\end{array}\right] = \left[\begin{array}{ccc}-1&c' + a'c + q^2b'c'&q^2(c+q^2bc+ac')\\0&aa'+q^4b'b'&q^4(ab+ab')\\0&a'b+a'b'&q^4bb+aa' \end{array}\right]
\]
The fact that the matrix on the right is the matrix $B$ follows from the identities
 \[
 b' = -b,\quad c' = -q^2X^{-2}c,\quad aa'-q^4bb' = 1,\quad a'c + q^2b'c' + c' =0.
 \]
\end{proof}
Explicitly, the action of $Y$ on $\hat M$ is given by
\begin{eqnarray*}
Y\cdot p' &=& -p'\\
Y\cdot u &=& \left[q^{-2}X^3-q^2X^{-1} \right]p' + \left[-q^{-2}X^4+q^{-2}X^2+q^2\right]u + \left[-q^{-2}X^2+q^2X^{-2}\right]v\\
Y\cdot v &=& \left[q^6X^{-3}-q^2X \right]p' + \left[q^2X^2-q^6X^{-2} \right]u + \left[-q^6X^{-4}+q^2X^{-2}+q^2\right]v
\end{eqnarray*}
Again, since $s$ acts diagonally, we see that $\e \hat M = \C[x]\{\delta p', w,v\}$.
\begin{lemma}
 The $\C[x]$-isomorphism $f: \e \hat M \to M$ defined by $f(\delta p') = p$, and $f(u) = u$, and $f(v) = v$ is an isomorphism of $A_q^{\Z_2}$-modules.
\end{lemma}
\begin{proof}
 This is a straightforward computation which is quite similar to the proof of Lemma \ref{lemma_symmetrictrefoil}.
\end{proof}

\begin{remark}
 The lower-right $2\times 2$ block of the matrix defining the action of $Y$ in formula (\ref{figeightops}) appeared in \cite[Prop 4.5]{CM11}, where it was used to describe an inhomogenous recursion relation satisfied by the sequence of colored Jones polynomials for the figure eight knot. 
\end{remark}

\section{Divisibility and recursion relations of colored Jones polynomials}\label{sec_divisibility}
If $q^4-1$ is invertible, then $A_q\rtimes \Z_2$ and $A_q^{\Z_2}$ are Morita equivalent, so there is a unique $A_q\rtimes\Z_2$-module $\hat K_q(S^3\setminus K)$ such that $\e \hat K_q(S^3\setminus K)$ is isomorphic to $K_q(S^3\setminus K)$ as an $A_q^{\Z_2}$-module. We call $\hat K_q(S^3\setminus K)$ the `nonsymmetric skein module,' and we give a formula (\ref{equation_liftedcoloredjones}) for the colored Jones polynomials in terms of $\hat K_q(S^3\setminus K)$. 

Garoufalidis and L\^e \cite{GL05} defined an action of $A_q\rtimes \Z_2$ on the space $\mathbf H := \Hom(\Z, \C[q^{\pm 1}])$ of sequences of Laurent polynomials (see (\ref{actiononsequences})). If we define $J_{-n}(q) = -J_n(q)$, then we can consider $J_n(q)$ as an element of $\mathbf H$, and the main theorem of \cite{GL05} is that the annihilator of $J_n(q)$ in $A_q$ is non-zero. In other words, the sequence $J_n(q)$ satisfies a (generalized) recurrence relation.
Lemmas \ref{lemma_liftedcoloredJones} and \ref{lemma_actiononseq} give a close relationship between the $A_q\rtimes \Z_2$ modules $\hat K_q(D^2\times S^1)$ and $\mathbf H$ - the latter is the linear dual of the former (up to a twist by an automorphism, see Remark \ref{remark_completion}).

We use this observation to give two applications of our conjectures. First, Conjecture \ref{conj_unknotsubmodule} states that $K_q(S^3\setminus K)$ contains $K_q(S^3\setminus \mathrm{unknot})$ as a submodule. (In particular, the action of the longitude on the skein module of the knot complement has an eigenvector with eigenvalue $-q^2-q^{-2}$.) In Theorem \ref{theorem_inhomogenous} we show that this conjecture implies that the colored Jones polynomials of a knot satisfy a (generalized) inhomogeneous recursion relation. These generalized inhomogeneous recursion relations have been found in examples, but when this paper was written it had not been proved that they exist for all knots.

Second, the algebra $\H_{q,t_1,t_2,1,1}$ acts on the space $\mathbf H(q) := \Hom(\Z, \C(q))$ of sequences of \emph{rational} functions. We show that our conjectured action of $\H_{q,t_1,t_2,1,1}$ on the nonsymmetric skein module of the complement of a knot $K$ implies that $a\cdot (J_n(q)) \in \C(q) \mathbf H \subset \mathbf H(q)$, for any $a \in \H_{q,t_1,t_2,1,1}$. In other words, if $P_n(q)$ is a sequence of rational functions obtained by multiplying $J_n(q)$ by an element of $\H_{q,t_1,t_2,1,1}$, then there is a uniform common denominator for the rational functions in the sequence $P_n(q)$. (This is a non-trivial statement because the action of $\H_{q,t_1,t_2,1,1}$ does \emph{not} preserve the subspace $\C(q) \mathbf H \subset \mathbf H(q)$.) 

As a corollary, we show a divisibility property for the colored Jones polynomials: the rational function 
\[
\frac{(q-q^{-1})(J_{n+j}(q) + J_{n-1-j}(q))}{1-q^{4n-2}} 
\]
is actually a Laurent polynomial. Following a suggestion of Garoufalidis, we use Habiro's cyclotomic expansion of the colored Jones polynomials to give a proof of this statement that does not rely on Conjecture \ref{mainconjecture}.

\subsection{Nonsymmetric pairings}
In this section we show that the topological pairing (\ref{knotpairing}) lifts via the Morita equivalence between $A_q\rtimes \Z_2$ and $A_q^{\Z_2}$, and we give a formula for the colored Jones polynomials in terms of this nonsymmetric pairing.

\subsubsection{Pairings and Morita equivalence}
We first give a lemma that gives a sufficient condition for $A_q\rtimes \Z_2$ and $A_q^{\Z_2}$ to be Morita equivalent.
\begin{lemma}
 Suppose $q^4-1$ is invertible. Then $A_q\rtimes \Z_2$ and $A_q^{\Z_2}$ are Morita equivalent.
\end{lemma}
\begin{proof}
 Recall $\e = (1+s)/2 \in A_q\rtimes \Z_2$, and that $A_q^{\Z_2}$ is isomorphic to $\e (A_q\rtimes \Z_2)\e$ via the map $a \mapsto a\e$. Standard Morita theory shows that $A_q\rtimes \Z_2$ and $\e (A_q\rtimes\Z_2)\e$ are Morita equivalent if the two-sided ideal generated by $\e$ contains $1$. Then the following computation completes the proof:
 \[
  [Y, s[X, 1+s]]_qY^{-1}X = q^{-1}-q^3
 \]
\end{proof}

We now give a lemma showing that pairings lift via Morita equivalences. Let $A$ be an algebra with an idempotent $\e^2 = \e \in A$. 
\begin{lemma}\label{lemma_liftingpairings}
 Suppose $A\e A = A$, and let $M\e$ and $\e N$ be left and right $\e A \e$-modules, respectively. Then the natural map of vector spaces $M\e \otimes_{\e A \e} \e N \to M \otimes_A N$ is an isomorphism.
\end{lemma}
\begin{proof}
 Since $A\e A = A$, the functors $A \e \otimes_{\e A \e} -$ and $\e A \otimes_A - $ are inverse equivalences, so the natural map $A \e \otimes_\eae \e A \otimes_A A \to A$ given by $a\e \otimes \e b \otimes c \mapsto a\e b c$ is an isomorphism. We then have the following isomorphisms of vector spaces:
\begin{eqnarray*}
M\e \otimes_\eae \e N &\cong& (M \otimes_A A\e) \otimes_\eae (\e A \otimes_A N) \\
&\cong&  M \otimes_A (A\e \otimes_\eae \e A) \otimes_A N\\
&\cong& M\otimes_A A \otimes_A N\\
&\cong& M\otimes_A N
\end{eqnarray*}
Under this chain of isomorphisms we have $m\e \otimes \e n \mapsto m\e \otimes \e n$. 
\end{proof}

\begin{corollary}\label{cor_pairinglifts}
 If $q^4-1$ is invertible, then there is a unique (nonsymmetric) pairing
 \[
  \langle -,-\rangle:\hat K_q(D^2\times S^1)\otimes_{A_q\rtimes \Z_2} \hat K_q(S^3\setminus K) \to \C
 \]
which lifts the topological pairing of Section \ref{topologicalpairing}.
\end{corollary}

\subsubsection{The nonsymmetric skein module of the solid torus}
Recall that under the isomorphism $K_q(T^2) \to A_q^{\Z_2}$, the meridian and longitude are sent to $X+X^{-1}$ and $Y+Y^{-1}$, respectively, and the $(1,1)$ curve is sent to $q^{-1}(XY+X^{-1}Y^{-1})$. If $N_K \subset S^3$ is a closed tubular neighborhood of a knot $K \subset S^3$, then $N_K$ is diffeomorphic to $D^2 \times S^1$ and $K_q(N_K)$ is a right $K_q(T^2)$-module. Let $u = 1_K \cdot (Y+Y^{-1}) \in K_q(N_K)$, where $1_K$ is the empty link in $N_K$.
\begin{lemma}
 As vector spaces, $K_q(N_K) \cong \C[u]$. The right action of $K_q(T^2)$ is determined by the formulas
 \begin{align*}
  1_K \cdot f(Y+Y^{-1}) &= f(u)\\
  1_K \cdot (X+X^{-1}) &= -(q^2+q^{-2})1_K\\
  1_K \cdot q^{-1}(XY+X^{-1}Y^{-1}) &= -q^{-3}u
 \end{align*}
\end{lemma}
\begin{proof}
 The identification $K_q(N_K)\cong \C[u]$ and the first claimed formula follow from Section \ref{subsec_coloredJpolys}. The second formula follows from the fact that the meridian is contractible inside $D^2\times S^1$ and $0$-framed inside $S^3$ under the inclusion $D^2\times S^1 \hookrightarrow S^3$. The third follows from the fact that inside of the solid torus, the $(1,1)$ curve is isotopic to the longitude with a framing twist (which accounts for the factor of $q$). The statement that these formulas completely determine the module structure of $K_q(N_K)$ follows from Lemma \ref{lemma_invsubalgebraiso}.
\end{proof}

Let $V = \C[U^{\pm 1}]$, and give $V$ a right $A_q\rtimes \Z_2$-module structure via
\[
 f(U)\cdot g(Y) = f(U)g(U^{-1}),\quad f(U)\cdot X = -f(q^2U),\quad f(U)\cdot s = -f(U^{-1})
\]
Because of the sign in the action of $s$, we have $V\e = (U-U^{-1})\C[U+U^{-1}] \subset V$.
\begin{lemma}\label{lemma_liftedsolidtorus}
 The $\C$-linear isomorphism $K_q(N_K) \to V\e$ given by $f(u) \mapsto (U-U^{-1})f(U+U^{-1})$ is an isomorphism of right $A_q^{\Z_2}$-modules.
\end{lemma}
\begin{proof}
 The isomorphism is linear over $\C[Y+Y^{-1}]$ because $Y+Y^{-1}$ commutes with $s$. We then compute
 \begin{align*}
  (U-U^{-1})\cdot (X+X^{-1}) &= -(q^2+q^{-2})(U-U^{-1})\\
  (U-U^{-1})\cdot q^{-1}(XY+X^{-1}Y^{-1}) &= -q^{-3}(U-U^{-1})\cdot (Y+Y^{-1})
 \end{align*}
Then the claim follows from Lemma \ref{lemma_invsubalgebraiso}.
\end{proof}

\subsubsection{The colored Jones polynomials from the nonsymmetric pairing}
Let $\hat K_q(S^3\setminus K)$ be the nonsymmetric skein module of the complement of a knot $K$. From Corollary \ref{cor_pairinglifts} and Lemma \ref{lemma_liftedsolidtorus}, we see that if $q^4-1$ is invertible, then the topological pairing lifts to a nonsymmetric pairing
\begin{equation}\label{eq_liftedpairing}
 \langle -,-\rangle: V \otimes_{A_q\rtimes \Z_2} \hat K_q(S^3\setminus K) \to \C
\end{equation}
(The fact that we have used the same notation for the topological pairing and the nonsymmetric pairing is justified by the last sentence of the proof of Lemma \ref{lemma_liftingpairings}.) From Theorem \ref{thm_coloredjonespolys}, we have the equality
\[
 J_n(K; q) = (-1)^{n-1}\langle S_{n-1}(u), \varnothing\rangle
\]
 In the nonsymmetric setting, this formula simplifies substantially.
\begin{lemma}\label{lemma_liftedcoloredJones}
 We have the equality
 \begin{equation}\label{equation_liftedcoloredjones}
  J_n(K; q) = (-1)^{n}2\langle 1_V\cdot Y^n, \varnothing\rangle
 \end{equation}
\end{lemma}
\begin{proof}
 Under the isomorphism $K_q(D^2\times S^1) \cong V\e$, the empty link is identified with the element $U-U^{-1}$. Combining this with Lemma \ref{chebyshevidentities} gives 
 \[
  S_{n-1}(u) = (U-U^{-1})S_{n-1}(U+U^{-1}) = U^n-U^{-n} \in V
 \]
  Then the following computation completes the proof:
  \[
   \langle U^n-U^{-n}, \varnothing\rangle = -\langle 1_V Y^n(1+s),\varnothing\rangle = -\langle 1_VY^n, 2\e \varnothing\rangle = -2\langle 1_VY^n, \varnothing\rangle
  \]

\end{proof}

\subsubsection{The action of $A_q\rtimes \Z_2$ on the sequence $J_n(K; q)$}
If we define $\mathbf H := \Hom(\Z, \C[q^{\pm 1}])$ and fix a knot $K$, then $J(n) := J_n(K; q)$ is an element of $\mathbf H$ (after extending the colored Jones polynomials to negative integers via $J_{-n}(K; q) := -J_{n}(K; q)$). In \cite{GL05}, Garoufalidis and L\^e defined a left action of $A_q\rtimes \Z_2$ on $\mathbf H$ as follows:
\begin{equation}\label{actiononsequences}
 (Xf)(n) := -q^{-2n}f(n),\quad (Yf)(n) := -f(n+1),\quad (sf)(n) := -f(-n)
\end{equation}
(Actually, their action is a twist of this action by an automorphism of $A_q\rtimes\Z_2$ - we have chosen this twist so that Lemma \ref{lemma_actiononseq} holds.) We now relate this action to the formula for the colored Jones polynomial from Lemma \ref{lemma_liftedcoloredJones}. 
\begin{lemma}\label{lemma_actiononseq}
 For $a \in A_q\rtimes \Z_2$, we have the equality
 \begin{equation}
  (a\cdot J)(n) = (-1)^{n}\langle Y^n, a\cdot 1_M\rangle
 \end{equation}
\end{lemma}
\begin{proof}
 It suffices to show the claim for generators of $A_q\rtimes \Z_2$, and these are straightforward computations. For example,
\[
 (-1)^{n}\langle Y^n, Y \cdot 1_M\rangle = (-1)^{n}\langle Y^{n+1},1_M\rangle = -J(n+1) = (Y\cdot J)(n)
\]
\end{proof}

\begin{remark}
 This lemma has appeared in the literature for $a \in A_q^{\Z_2}$. However, the extension of the lemma to all of $A_q\rtimes \Z_2$ gives an interpretation of the appearence of the action (\ref{actiononsequences}). Also, the fact that $V$ is the sign representation (i.e. $1_Vs = -1_V$) gives a skein-theoretic interpretation of the sign in the definition $J(n) := -J(-n)$.
\end{remark}

\begin{remark}\label{remark_completion}
If $V$ is the nonsymmetric skein module of the solid torus, then the choice of basis $\{U^n\}$ for $V$ gives a linear isomorphism $V^* \to \mathbf H$, where $V^*$ is the linear dual of $V$. Since $V$ is a right $A_q\rtimes \Z_2$-module, its dual $V^*$ is a left $A_q\rtimes \Z_2$-module, and in this language, Lemma \ref{lemma_actiononseq} can be interpreted as the statement that the map $V^* \to \mathbf H$ is an isomorphism of left $A_q\rtimes\Z_2$-modules. 
\end{remark}

\subsection{Inhomogeneous recursion relations}
We will say that a sequence $f(n) \in \mathbf H$ satisfies an \emph{inhomogeneous recursion relation} if there is a nonzero $a \in A_q\rtimes \Z_2$ such that the sequence\footnote{In the definition of inhomogeneous recursion relation, we do not require $P(n)$ to be nonzero. If $P(n) = 0$, then $f(n)$ satisfies an inhomogeneous recursion relation that happens to be homogeneous.} $P(n) := af(n)$ satisfies $P(n) = P(0)$. Let $K_\varnothing \subset K_q(S^3\setminus K)$ be the $A_q^{\Z_2}$-submodule generated by the empty link $\varnothing$.

\begin{theorem}\label{theorem_inhomogenous}
 Suppose there is a nonzero map $K_q(S^3\setminus \mathrm{unknot}) \to K_\varnothing$ of $A_q^{\Z_2}$-modules. Then the sequence $J(n) := J_n(K; q)$ satisfies an inhomogeneous recursion relation.
\end{theorem}
\begin{proof}
 Let $V = \hat K_q(S^3\setminus \mathrm{unknot})$ be the nonsymmetric skein module of the unknot described in (\ref{formula_liftedunknot}). Since Morita equivalence is functorial, we have a nonzero map $V \to \hat K_\varnothing \subset \hat K_q(S^3\setminus K)$. Let $p \in \hat K_\varnothing$ be the image of $1_V$ under this map (which is \emph{not} the image of the empty link). The surjective map $A_q^{\Z_2} \twoheadrightarrow K_\varnothing$ lifts via the Morita equivalence to a surjective map $(A_q\rtimes \Z_2)\e \twoheadrightarrow \hat K_\varnothing$, which implies $\hat K_\varnothing$ is generated (as an $A_q\rtimes \Z_2$-module) by the empty link $\varnothing$. Therefore, there exists $a \in A_q\rtimes \Z_2$ such that $p = a\cdot  \varnothing$. Formula (\ref{formula_liftedunknot}) shows that $1_V$ is an eigenvector of $Y$ with eigenvalue $-1$, so we have $Ya\varnothing = -a\varnothing$. We define
 \[
 P(n) := a\cdot J(n)
 \]
 We then compute
 \begin{equation*}
  P(n+1) = -Ya\cdot J(n) = (-1)^{n}\langle Y^n,Ya\varnothing\rangle
  = (-1)^{n-1}\langle Y^n, a\varnothing\rangle
  = P(n)
 \end{equation*}
\end{proof}

\subsection{Divisibility properties of $J_n(K; q)$}
Recall that $D_q$ is the localization of $A_q\rtimes \Z_2$ at the multiplicative set consisting of all nonzero polynomials in $X$. If we define $\mathbf H(q) := \Hom(\Z, \C(q))$ to be the space of sequences of rational functions, then the action of $A_q\rtimes\Z_2$ on $\mathbf H$ extends to an action of $D_q$ on $\mathbf H(q)$ via the formulas
\begin{equation}\label{actionofdqonsequences}
 \left(\frac{F(X)}{G(X)}\cdot f\right)(n) := \frac{F(-q^{-2n})}{G(-q^{-2n})}f(n),\quad (Y\cdot f)(n) := -f(n+1),\quad (s\cdot f)(n) := -f(-n)
\end{equation} 

The double affine Hecke algebra $\H_{q,\ult}$ can be viewed as a subalgebra of $D_q$, and this gives $\mathbf H(q)$ the structure of an $\H_{q,\ult}$ module. Garoufalidis and L\^e \cite{GL05} showed that the $A_q\rtimes \Z_2$-module map defined by $a \mapsto a\cdot J(n)$ has a nontrivial kernel. This leads to the following question which we hope to address in future work:
\begin{question}\label{question_Hrecursion}
 Does the map $\H_{q,\ult} \to \mathbf H(q)$ defined by $a \mapsto a\cdot J(n)$ have a nontrivial kernel?
\end{question}

In this section we relate the conjectured action of $\H_{q,t_1,t_2,1,1}$ on the nonsymmetric skein module to the action of $\H_{q,t_1,t_2,1,1}$ on $\mathbf H(q)$. We recall from Proposition \ref{prop_dunklembedding} that under the standard embedding, $\H_{q,t_1,t_2,1,1}$ is the subalgebra of $D_q$ generated by the elements $X$, $s$, and 
\[
 Y_{k,u} := t_1 Y - (q^2 \bar t_1 X^{-2}  + q\bar t_2X^{-1})sU
\]
where $U$ is the operator 
\begin{equation}\label{equation_operatorU}
 U := \frac{1}{1-q^2X^2}(1-sY)
\end{equation}

Let $M := (A_q\rtimes \Z_2) \cdot 1_K \subset \hat K_q(S^3\setminus K)$ be the submodule of the nonsymmetric skein module of a knot $K$ generated by the empty link. If $M^{\loc}$ is the localization of $M$ at all nonzero polynomials in $X$, then both $D_q$ and its subalgebra $\H_{q,t_1,t_2,1,1}$ naturally act on $M^{\loc}$. We recall that Conjecture \ref{mainconjecture} states that the natural map $M \to M^\loc$ is injective and that the action of  $\H_{q,t_1,t_2,1,1}$ preserves the subspace $M \subset M^\loc$. It is clear that this statement implies $UY^j1_M \in M$ for all $j \in \Z$. We then define
 \begin{equation}\label{equation_P}
  P_{j}(n) := (-1)^{n+j}\frac{J(n+j)+J(n-j-1)}{q^{4n-2}-1}
 \end{equation}

\begin{theorem}\label{thm_divisibility}
 Assume Conjecture \ref{mainconjecture} holds for $K \subset S^3$. Then there exist $a_{j,k,l}(q) \in \C(q)$ depending on $K$ but not on $n \in \Z$ such that
 \[
  P_{j}(n) = \sum_{k,l}a_{j,k,l}(q)q^{-2nk}J(n+l)
 \]
Furthermore, $(q^2-1)P_{j}(n)$ is a Laurent polynomial in $\C[q^{\pm 1}]$.
\end{theorem}
\begin{proof}
 To establish the claims we use Lemma \ref{lemma_actiononseq} to compute the quantity $\langle Y^n, UY^j 1_M\rangle$ in terms of the colored Jones polynomials:
 \begin{align*}
  \langle Y^n, UY^j 1_M\rangle &= (-1)^{n}UY^jJ(n)\\
  &= (-1)^{n}\frac{1}{1-q^2X^2}(1-sY)Y^jJ(n)\\
  &= (-1)^{n+j}\frac{1}{1-q^2X^2}(1-sY)J(n+j)\\
  &= \frac{(-1)^j}{1-q^2X^2}\left[ (-1)^{n}J(n+j) + (-1)^{-n+1}J(-n+1+j)\right]\\
  &= \frac{ (-1)^{n+j}\left[J(n+j) + J(n-1-j)\right]}{q^{4n-2}-1}\\
  &= P_{j}(n)
 \end{align*}
 Since the nonsymmetric pairing exists whenever $q^4-1$ is invertible, the rational function $P_{j}(n)$ can only have poles when $q^4-1=0$. However, the colored Jones polynomials are Laurent polynomials, and the denominator of $P_{j}(n)$ has simple zeros when $q^2-1 = 0$ and does not have zeros when $q^2+1 = 0$. Therefore, $(q^2-1)P_{j}(n)$ is a Laurent polynomial, which shows the second claim.
 
 To show the first claim, note that Conjecture \ref{mainconjecture} implies that there exist $a_{j}(k,l) \in \C(q)$ such that 
 \[
 UY^j 1_M = \sum_{k,l}a_{j}(k,l)X^kY^l 1_M
 \]
 Combining this with the previous computation, we have the equalities
 \begin{eqnarray*}
  P_j(n) &=& (-1)^{n}\langle Y^n, UY^j1_M\rangle \\
  &=& (-1)^{n}\langle Y^n, \sum_{k,l}a_jX^kY^l 1_M\rangle\\
  &=& \sum_{k,l} a_{j,k,l}(q) (-1)^k q^{-2nk} J(n+l)
 \end{eqnarray*}
To obtain the claimed statement, the factor $(-1)^k$ can be absorbed into the coefficient $a_j(k,l)$. (Note that the power of $q$ cannot be absorbed into this coefficient because $a_j(k,l)$ does not depend on $n \in \Z$.)
\end{proof}


\subsection{Habiro's cyclotomic expansion}
In this section we use Habiro's cyclotomic expansion \cite{Hab08} of the colored Jones polynomials to prove that for any knot, the rational function $(q^2-1)P_j(n)$ from (\ref{equation_P}) is actually a Laurent polynomial. We first recall this expansion in our normalization conventions (see Remark \ref{remark_signconvention}). Define the polynomials
\[
 c_{n,k} := \prod_{j=1}^k (q^{4n}+q^{-4n}-q^{4j}-q^{-4j})
\]
By definition, $c_{n,0} = 1$, $c_{n,n} = 0$, and $c_{n,k+1} = (q^{4n}+q^{-4n}-q^{4k+4}-q^{-4k-4})c_{n,k}$.
\begin{theorem}[\cite{Hab08}]\label{theorem_habiro}
 There exist $H_k \in \Z[q^{\pm 1}]$, independent of $n$, such that
 \[
  J(n) = \sum_{k=0}^{n-1} \frac{q^{2n}-q^{-2n}}{q^2-q^{-2}}c_{n,k} H_k
 \]
\end{theorem}
Since $c_{n,n} = 0$, we may take the upper limit of this sum to be infinity. Also, as an example of the theorem, for the figure eight knot we have $H_k = 1$ for all $k$, and for the unknot we have $H_0 = 1$ and $H_k = 0$ for $k > 1$. Finally, since we use the convention $J(n) = -J(-n)$, the theorem is also true for negative $n$ if we define $H_{-k} = H_k$.
\begin{theorem}\label{theorem_divisibilityfromhabiro}
 The following rational function is actually a Laurent polynomial:
 \[
  (q^2-1)P_j(n) = (q^2-1)\frac{J(n+j)+J(n-j-1)}{q^{4n-2}-1}
 \]
\end{theorem}
\begin{proof}
 (In the statement of the theorem we have ignored the sign $(-1)^{n+j}$.) Since Theorem \ref{theorem_habiro} is true for both positive and negative $n$, we are free to assume that $n+j \geq 0$ and $n-j-1 \geq 0$. If we shorten notation by writing  $a = (q^{2n+2j}-q^{-2n-2j})$ and $b = (q^{2n-2j-2}-q^{2+2j-2n})$, we then have
 \begin{align*}
  (q^2-1)P_j(n)
  & = \frac{1}{(q+q^{-1})(q^{4n-2}-1)} \sum_{k=0}^\infty H_k \left[ac_{n+j,k} + bc_{n-j-1,k}\right]\\
  & =: \sum_{k=0}^\infty H_k \frac{s_k}{(q+q^{-1})(q^{4n-2}-1)}
  \end{align*}
 
  We prove that $s_k$ is divisible by ${(q+q^{-1})(q^{4n-2}-1)}$ by induction on $k$. Since $c_{n,0} = 1$, we have
  \begin{align*}
   s_{k=0} & = (q^{2n+2j}-q^{-2n-2j}) + (q^{2(n-j-1)} - q^{2(1+j-n)}) \\
   & = q^{2n+2j}- q^{2(1+j-n)} - q^{-2n-2j} + q^{2(n-j-1)} \\
   & = q^{2j}(q^{2n}-q^{-2n+2}) + q^{-2j}(q^{2n-2}-q^{-2n})\\
   & \equiv q^{2j-2n}(q^{-2} + q^{-4j}) \quad (\textrm{mod } q^{4n-2}-1)
  \end{align*}
  Since the expression on the final line is divisible by $q^{4j-2}+1$, it is divisible by $q^2+1$, and this proves the claim for $k=0$. For the inductive step, we will show $s_k \equiv s_{k-1} (\textrm{mod } q^{4n-2}-1)$.  We first compute
  \begin{align*}
   s_{k} &= ac_{n+j,k} + bc_{n-j-1,k}\\
   &= a\left(q^{4(n+j)}+q^{-4(n+j)}-q^{4k}-q^{-4k}\right)c_{n+j,k-1} \\
   &\,  + b\left(q^{4(n-j-1)}+q^{-4(n-j-1)}-q^{4k}-q^{-4k}\right)c_{n-j-1,k-1}
  \end{align*}
  We now split this into four terms, each of which can be dealt with similarly. For example,
  \begin{align*}
   & \,\,aq^{4(n+j)}c_{n+j,k-1} + bq^{-4(n-j-1)}c_{n-j-1,k-1}\\
   = & \,\,q^{4j} \left( a q^{4n} c_{n+j,k-1} + b q^{4-4n}c_{n-j-1,k-1}\right)\\
   \equiv &\,\, q^{4j}\left( a q^{4n} c_{n+j,k-1} + b q^{4n}c_{n-j-1,k-1}\right)\\
   \equiv &\,\, q^{4j+4n}s_{k-1}\quad (\textrm{mod } q^{4n-2}-1)
  \end{align*}
\end{proof}
\begin{remark}\label{remark_strongerdivisibility}
 This proof actually shows the following rational function is a Laurent polynomial:
 \[
  \frac{(q^2-1) \left(J(n+j)+J(n-j-1)\right)}{(q^{4n-2}-1)(q^{4j-2}+1)}
 \]
\end{remark}

\subsection{More divisibility properties}\label{sec_congruences}
In this section we prove Conjecture 1.6 of \cite{CLPZ14} using the techniques of the previous section (and, in particular, Habiro's theorem). 
Their conjecture is stated using a different normalization of the colored Jones polynomials, so in this section we use their normalization. In particular, their $q$ is our $q^2$, and they used the normalized colored Jones polynomials $\bar J_n(K)$, which are related to ours via $\bar J_n(K) = J_{n+1}(K) / J_{n+1}(\mathrm{unknot})$. Let $[n] := q^n - q^{-n}$. We prove the following theorem which implies \cite[Conj. 1.6]{CLPZ14}:

\begin{theorem}\label{thm_clpzconj}
 For any knot the following congruence holds:
\begin{equation}\label{eq_clpz}
 \bar J_{n-1}(q) - \bar J_{k-1}(q) \equiv 0 \quad (\mathrm{mod }\,\, [n-k][n+k])
\end{equation}
\end{theorem}
\begin{proof}
First we define
\[
 d_{n,j} := \prod_{k=1}^j(q^{2n}+q^{-2n}-q^{2k} - q^{-2k})
\]
(where $d_{n,0} = 1$ by convention). Habiro's theorem in the normalization conventions of \cite{CLPZ14} says that there exist polynomials $H_k \in \Z[q^{+\pm 1}]$, independent of $n$, such that
\[
 \bar J_{n-1}(q) = \sum_{j=0}^\infty d_{n,j} H_j
\]
(This sum is finite because $d_{n,n} = 0$.) We can therefore write the left hand side of (\ref{eq_clpz}) as follows:
\[
 \bar J_{n-1}(q) - \bar J_{k-1}(q) = \sum_{j=0}^\infty (d_{n,j} - d_{k,j})H_j
\]
We prove by induction on $j$ that each coefficient $a_j := (d_{n,j}-d_{k,j})$ is congruent to $0$ modulo $[n-k][n+k]$. The base case $j=0$ is trivial since $d_{n,0} = 1$. Now assume $a_{j-1} \equiv 0$. We first compute
\begin{eqnarray*}
 d_{n,j} &=& (q^{2n}+q^{-2n}-q^{2j}-q^{-2j})d_{n,j-1}\\
 &\equiv& (-[n-k][n+k] + q^{2n}+q^{-2n}-q^{2j}-q^{-2j})d_{n,j-1}\\
 &=& (q^{2k}+q^{-2k}-q^{2j}-q^{-2j})d_{n,j-1}
\end{eqnarray*}
We then have the following congruences:
\begin{eqnarray*}
 a_{j} &=& (q^{2n}+q^{-2n} - q^{2j}-q^{-2j})d_{n,j-1} - d_{k,j}\\
 &\equiv& (q^{2k+2}+q^{-2k-2} - q^{2j}-q^{-2j})d_{n,j-1} - d_{k,j}\\
 &=& (q^{2k+2}+q^{-2k-2} - q^{2j}-q^{-2j})(d_{n,j-1} - d_{k,j-1})\\
 &\equiv& 0
\end{eqnarray*}
This completes the proof of the theorem.
\end{proof}
\begin{remark}
 The difference in the signs in (\ref{eq_clpz}) and the numerator of Remark \ref{remark_strongerdivisibility} comes from the differences in normalization between $J_n$ and $\bar J_n$.
\end{remark}

\section{Canonical 3-parameter deformations}\label{sec_ccdeformations}
In this section we discuss deformations of (nonsymmetric) skein modules to the DAHA $\H_{q,t_1,t_2,t_3,t_4} = \H_{q,\underline t}$ of type $C^{\vee} C_1$ introduced by Sahi in \cite{Sah99} (see also \cite{NS04}).
%
To reduce confusion, in this section we write $X$, $\yy$, and $s$ for the generators of $A_q\rtimes \Z_2$. As before, we let $D_q$ be the algebra obtained from $A_q\rtimes \Z_2$ by inverting all nonzero polynomials in $X$. Then $\H_{q,\ult}$ is the subalgebra of $D_q$ generated by $X$, $X^{-1}$ and the following operators in $D_q$:
\begin{eqnarray*}
 T_0 &=& t_1 s\yy - \frac{q^2 \bar t_1 X^2  + q\bar t_2X}{1-q^2X^2}(1-s\yy)\\
 T_1 &=& t_3s + \frac{\bar t_3+\bar t_4 X}{1-X^2}(1-s)
\end{eqnarray*}

If $M$ is an $A_q\rtimes \Z_2$-module, we write $M^\loc := D_q \otimes_{A_q\rtimes \Z_2} M$ for the localization of $M$ at all nonzero polynomials $X$. If $M$ is free over the subalgebra $\C[X^{\pm 1}]$ (which is the case in all our examples), then the natural map $M \to M^\loc$ is injective. Since $\H_{q,\ult}$ is a subalgebra of $D_q$, it acts naturally on $M^\loc$.

We now prove Conjectures \ref{mainconjecture} and \ref{conj_5param} for our example knots.
\begin{theorem}\label{thm_mainconjecture}
 Let $M$ be the $A_q\rtimes \Z_2$ module which is the nonsymmetric skein module of the unknot, a $(2,2p+1)$ torus knot, or the figure eight knot, and let $M'$ be the quotient of $M$ by the unknot submodule.
 \begin{enumerate} 
  \item The action of $\H_{q,t_1,t_2,1,1}$ preserves the subspace $M \subset M^\loc$. 
  \item The action of $\H_{q,t_1,t_2,t_3,t_4}$ preserves the subspace $M' \subset (M')^\loc$.
 \end{enumerate}
\end{theorem}
\begin{remark}
 For the knots listed in the theorem, the results of Section \ref{sec_qeqqskeinmodules} make it clear that there is a unique map from the skein module of the unknot to the skein module of $K$, so the quotient $M'$ in the second statement of the theorem is well-defined.
\end{remark}

\begin{proof}
We define the operators
\begin{equation}\label{operatoru}
 U_0 := \frac{1}{1-q^2X^2}(1-s\hat y),\quad U_1 := \frac{1}{1-X^2}(1-s)
\end{equation}
To prove the first statement, it suffices to show that $U_0$ preserves $M \subset M^\loc$. Once this is proved, the second statement is implied by the statement $U_1 M' \subset M'$. To avoid confusion of notation, after proving the technical Lemma \ref{lemma_upreservesM}, we will divide the proof into separate subsections (one for each knot). 
\end{proof}

\begin{remark}\label{remark_eUe}
 Before continuing with the proof of the theorem, we remark that if $\e M = K_q(S^3\setminus K)$, the conjecture that $U_0$ preserves the nonsymmetric skein module $M \subset M^\loc$ implies that $q^{-1}\e X^{-1} U_0 \e$ preserves $\e M$ inside the localization of $\e M$ at polynomials in $x = X+X^{-1}$. Under the identification $A_q^{\Z_2} = K_q(T^2\times [0,1])$, we have
 \[
  q^{-1}\e X^{-1}U_0 \e = \frac{1}{(x-(q+q^{-1}))(x+(q+q^{-1}))}\left[ (q - q^{-1})x - q^{-2}(1,-1) + q^2(1,1)\right]
 \]
  (Here $x$ is the meridian $(1,0)$, and the notation $(m,l)$ refers to the $(m,l)$ curve on $T^2$.) We note that if $K$ is the unknot, then the operator $q^{-1}\e X^{-1}U_0 \e$ annihilates the empty link. 
\end{remark}

We first give a technical lemma that provides conditions that imply $U_i M \subset M$ for $i=0,1$. We recall the notation of Lemma \ref{lemma_aqmodule}. In particular, suppose that $M = \C[X^{\pm 1}]\otimes_\C V$ and define operators $S, P \in \End_\C(M)$ using (\ref{equation_diagonalaction}). Furthermore, suppose that $A(X),B(X) \in \End_{\C[X^{\pm 1}]}(M)$ satisfy
\[
 B(X)B(X^{-1}) = 1, \quad A(X)B(X)A(q^2X^{-1}) = B(X)
\]
We then define $s,\yy: M \to M$ by $\yy = A(X)P$ and $s = B(X)S$, and Lemma \ref{lemma_aqmodule} shows that the operators $X, s, \yy:M \to M$ define a representation of $A_q\rtimes \Z_2$.
\begin{lemma}\label{lemma_upreservesM}
Assume the notation in the previous paragraph, and suppose the following condition holds:
 \[
  (1-B(X)A(X^{-1})) M \subset (1-q^2X^2)M
 \]
Then the operator $U_0$ from formula (\ref{operatoru}) preserves $M \subset M^\loc$. Also, the statement $U_1 M' \subset M'$ is implied by the following condition:
\[
 (1-B(X))M' \subset (1-X^2)M'
\]

\end{lemma}
\begin{proof}
Since $(SP)^2 = 1$, the elements $\f_\pm := \frac{1\pm SP}{2}$ are idempotents that satisfy $(SP)\f_\pm = \pm \f_\pm$. (These are not the standard idempotents that have been used previously.) We can then write
\begin{eqnarray*}
 (1-s\yy) &=& (1-B(X)SA(X)P)(\f_++\f_-)\\
 &=& (1-B(X)A(X^{-1})SP)(\f_+ +\f_-)\\
 &=& (1-B(X)A(X^{-1}))\f_+ + (1+B(X)A(X^{-1}))\f_-
\end{eqnarray*}
Since $(1-SP)\cdot (f(X)\otimes v) = (f(X) - f(q^{-2}X^{-1}))\otimes v$, we see that $\f_-\cdot M \subset (1-q^2X^2)M$, and since $B(X)$ and $A(X)$ are $\C[X^{\pm 1}]$-linear, this implies 
\[(1+B(X)A(X^{-1}))\f_-M \subset (1-q^2X^2)M\]
By assumption, $(1-s\yy)\f_+ M \subset (1-q^2X^2)M$, and this shows $U_0M \subset M$. The second statement follows by a similar argument which we omit.
\end{proof}

\begin{remark}\label{remark_signproblems}
 The first condition in Lemma \ref{lemma_upreservesM} is equivalent to the conditions $B(q^{-1})A(q) = \mathrm{Id}$ and $B(-q^{-1})A(-q) = \mathrm{Id}$, where both equalities hold in $M_n(\C)$. However, the conditions in Lemma \ref{lemma_aqmodule} hold a-priori, and when specialized to $X=q^{-1}$ they become $(B(q^{-1})A(q))^2 = \mathrm{Id}$ (and similar for $X=-q$). Since a (complex) matrix that squares to the identity matrix is diagonalizable, the only obstruction for Lemma \ref{lemma_upreservesM} to hold is ``sign problems'' (at least for modules which are free and finitely generated over $\C[X^{\pm 1}]$).
\end{remark}

\subsection{The unknot}
We recall that the module structure of the (nonsymmetric) skein module $M$ of the unknot is given by
\[
 M \cong \C[X^{\pm 1}],\quad s\cdot 1 = -1, \quad \yy \cdot 1 = -1
\]
We then compute
\[
 (1-s\yy)X^n = X^n - q^{-2n}X^{-n}
\]
The right hand side is clearly divisible by $(1-q^2X^2)$, and since $\{X^n\}$ is a basis for $M$, we see that $U_0$ preserves $M \subset M^\loc$. The second claim of Theorem \ref{thm_mainconjecture} is tautological in this case.


\subsection{$(2,2p+1)$ torus knots}\label{sec_conj1fortorusknots}
Let $M$ be the nonsymmetric skein module of the $(2,2p+1)$ knot $L_p$. (More precisely, $M$ is the lift of the $K_q(T^2)$ submodule of $K_q(S^3\setminus L_p)$ generated by the empty link.) We recall that $M = \C[X^{\pm 1}]\{u,v\}$ as a $\C[X^{\pm 1}]$-module, with the action of $s, \yy$ given by
\[
 s = \left[\begin{array}{cc}-1&0\\0&1\end{array}\right]S, \quad \yy = A(X)P := \left[\begin{array}{cc}-1& (-1)^{p}q^{2p+4}(X^{-2p-3}-q^{-4} X^{-2p+1})\\0 &q^{2(2p+1)}X^{-2(2p+1)}\end{array}\right]P
\]
Since $u$ generates the unknot submodule, we may abuse notation and identify $M' = \C[X^{\pm 1}]v$, with the action of $s,\yy$ given by
\[
 s = B'S = \left[ 1 \right] S,\quad \yy = \left[ q^{2(2p+1)}X^{-2(2p+1)}\right]P
\]

\begin{lemma}
 The operator $U_0$ preserves $M \subset M^\loc$, and the operator $U_1$ preserves $M' \subset (M')^\loc$.
\end{lemma}
\begin{proof}
To show that the condition of Lemma \ref{lemma_upreservesM} holds, we compute
\[
 1 - BA(X^{-1}) = \left[\begin{array}{cc}0& (-1)^{p+1}q^{2p+4}(X^{2p+3}-q^{-4} X^{2p-1})\\0 &1 - q^{2(2p+1)}X^{2(2p+1)}\end{array}\right]
\]
Since all the entries are divisible by $1-q^2X^2$, this shows the first claim. The second claim follows from the fact that $B'$ is the identity matrix.
\end{proof}
\begin{example}\label{example_trefoil}
The action of the generators $X, s, Y$ of $\H_{q,t_1,t_2,1,1}$ on the generators $u,v$ can be computed explicitly. For example, for the trefoil (i.e. $p=1$), $X$ acts by multiplication, and the action of $s$ and $Y$ are given by the following:
\begin{eqnarray*}
 s\cdot u &=& -u\\
 s\cdot v &=& v\\
 Y\cdot u &=& -t_1u\\
 Y\cdot v &=& [t_1(q^2X^{-1}-q^6X^{-5}) - (q^2\bar t_1X^{-2} + q\bar t_2X^{-1})(q^4X^{-3}+q^2X^{-1}) ]u\\
 &\,& + [t_1q^6X^{-6} - (q^2\bar t_1X^{-2} + q\bar t_2X^{-1})(q^4X^{-4}+q^2X^{-2}+1) ]v
\end{eqnarray*}
Here we have written $\bar t_i = t_i - t_i^{-1}$. It is clear that when $t_1=t_2=1$, these formulas specialize to operator $Y$ in (\ref{eq_trefoilexplicit}). 
\end{example}
\subsection{The figure eight}
Let $M$ be the nonsymmetric skein module of the figure eight knot. We recall that as a $\C[X^{\pm 1}]$-module we have $M = \C[X^{\pm 1}] \otimes_\C (\C p' \oplus \C u \oplus \C v)$. We recall the polynomials $a,b,c \in \C[X^{\pm 1}]$:
\[
 a = -q^{-2}X^4+q^{-2}X^2+q^2, \quad b = -q^{-2}X^2+q^2X^{-2},\quad c = q^{-2}X^3 - q^2X^{-1}
\]
Finally, $s$ and $\yy$ are defined by matrices (with respect to the ordered basis $\{p', u, v\}$):
\[
 s = \left[\begin{array}{ccc}-1&0&0\\0&1&0\\0&0&1\end{array}\right]S, \quad \yy = \left[\begin{array}{ccc}-1& c(X)&q^2c(q^2X^{-1})\\0 &a(X)&q^4b(q^2X^{-1})\\0&b(X)&a(q^2X^{-1})\end{array}\right]P
\]
Since $p'$ generates the unknot submodule, we may write the action of $s$ on the quotient $M'$ as
\[
 s = B'S = \left[\begin{array}{cc}1&0\\0&1\end{array}\right]S
\]

\begin{lemma}
 The operator $U_0$ preserves $M \subset M^\loc$, and $U_1$ preserves $M' \subset (M')^\loc$.
\end{lemma}
\begin{proof}
 To show that the first condition of Lemma \ref{lemma_upreservesM} holds, we compute
 \[
 1 - BA(X^{-1}) = \left[\begin{array}{ccc}0& -c(X^{-1})&-q^2c(q^{2}X)\\0 &1-a(X^{-1})& -q^4b(q^2X)\\0&-b(X^{-1})&1-a(q^2X)\end{array}\right]
 \]
 It is straightforward to check that all entries in this matrix are divisible by $1-q^2X^2$, which proves the first claim. The second claim follows from the fact that $B'$ is the identity matrix.
\end{proof}

\subsection{3-variable Jones polynomials}
In the previous section we gave examples of skein modules of knot complements that extend to representations of $\H_{q,t_1,t_2,1,1}$. In this section we use these modules to give example calculations of 3-variable polynomials $J_n(q, t_1,t_2) \in \C[q^{\pm 1},t_1^{\pm 1}, t_2^{\pm 1}]$ that specialize to the colored Jones polynomials when $t_1,t_2=1$. More precisely, 
if $Y_{t_1,t_2}=sT_0 \in D_q$, we define
\begin{equation}\label{equation_defof3varpolys}
 J_n(q,t_1,t_2) := \langle \varnothing, S_{n-1}(Y_{t_1,t_2}+Y^{-1}_{t_1,t_2})\cdot \varnothing\rangle
\end{equation}
where $\varnothing$ is the empty link and $\langle -,-\rangle$ is the pairing from Corollary \ref{cor_pairinglifts}. We remark that the following equality is a direct consequence of Theorem \ref{thm_coloredjonespolys} and the fact that $Y_{1,1}+Y^{-1}_{1,1}$ is the longitude:
\begin{equation}\label{equation_specialization}
 J_n(q) = J_n(q, t_1=1,t_2=1)
\end{equation}
In other words, the polynomials $J_n(q,t_1,t_2)$ specialize to the classical colored Jones polynomials of the knot $K$. We prove the following symmetry property of these 3-variable polynomials, which extends the well-known symmetry for the classical colored Jones polynomials.

\begin{proposition}\label{prop_mirror}
 Let $\bar K$ be the mirror image of $K$, and suppose Conjecture \ref{mainconjecture} holds for $K$. Then
 \[
  J_n(K; q,t_1,t_2) = J_n(\bar K; q^{-1},t_1^{-1},t_2^{-1})
 \]
\end{proposition}
\begin{proof}
The mirror map $S^3 \to S^3$ induces a $\C$-linear isomorphism 
\begin{equation}\label{equation_mirror}                                                                                                                                                           
K_q(S^3\setminus K) \stackrel \sim \to K_{q^{-1}}(S^3\setminus \bar K)                                                                                                                                                         \end{equation}
and we identify these two skein modules as vector spaces using this map. If $\epsilon(K,q): K_q(S^3\setminus K) \to K_q(S^3)$ is the $\C$-linear map induced by the inclusion $K \to S^3$, then $\epsilon(K,q) = \epsilon(\bar K, q^{-1})$ (under the identification (\ref{equation_mirror})). Furthermore, the skein module $K_{q^{-1}}(S^3\setminus \bar K)$ is the twist of the skein module $K_q(S^3\setminus K)$ by the isomorphism 
 \[
  \varphi:A_q\rtimes \Z_2 \to A_{q^{-1}}\rtimes \Z_2,\quad \varphi(X) = X,\quad \varphi(\hat y) = \hat y^{-1},\quad \varphi(s) = s
 \]
(Since this isomorphism sends $s \mapsto s$, it descends to an isomorphism of the spherical subalgebras.) This automorphism extends to $\varphi:D_q \to D_{q^{-1}}$. Let $\Theta_{q,t_1,t_2}:\H_{q,t_1,t_2,1,1} \to D_q$ be the standard embedding given by Proposition \ref{prop_dunklembedding}. To prove the claim, we show that 
\begin{equation}\label{equation_ymirror}
\varphi(\Theta_{q,t_1,t_2}(Y_{t_1+t_2}+Y^{-1}_{t_1,t_2})) = \Theta_{q^{-1},t_1^{-1},t_2^{-1}}(Y_{t_1^{-1},t_2^{-1}} + Y^{-1}_{t_1^{-1},t_2^{-1}})
\end{equation}
Translating \cite[Prop. 5.8]{NS04} into our notation and specializing $t_3=t_4=1$, we have
\begin{eqnarray}
 \Theta_{q,t_1,t_2}(Y_{t_1,t_2}+Y_{t_1,t_2}^{-1}) &=& A_{q,t_1,t_2}(X)(\hat y -1) + A_{q,t_1,t_2}(X^{-1})(\hat y^{-1} - 1)+ t_1 + t_1^{-1}\label{eq_askeywilsonop}\\
  A_{q,t_1,t_2}(X) &:=& \frac{t^{-1}_1qX^{-1} - t_2 + t_2^{-1} - t_1q^{-1}X}{qX^{-1}-q^{-1}X}\notag
\end{eqnarray}
We then have the equality
\[
 \varphi(A_{q,t_1,t_2}(X)) = A_{q^{-1},t_1^{-1},t_2^{-1}}(X^{-1})
\]
which proves (\ref{equation_ymirror}) and completes the proof of the proposition.
\end{proof}

For the unknot we compute a closed formula for $J_n(q,t_1,t_2)$ in the following lemma, but for nontrivial knots this computation seems to be more difficult, so we include several example computations in the appendix.
\begin{theorem}
If $K$ is the unknot, then
\[
 J_n(K;q,t_1,t_2) = \frac{(t_1^{-1}q^2)^n - (t_1^{-1}q^2)^{-n}}{t_1^{-1}q^2-(t_1^{-1}q^2)^{-1}}
\] 
\end{theorem}
\begin{proof}
 Using (\ref{eq_askeywilsonop}) and (\ref{formula_liftedunknot}), it is straightforward to compute 
 \[
  (Y_{t_1,t_2} + Y_{t_1,t_2}^{-1})\cdot \varnothing = -(t^{-1}_1q^2+t_1q^{-2})\varnothing
 \]
Then the claimed equality follows from the fact that $\langle \varnothing,\varnothing\rangle = 1$ and from the identities for Chebyshev polynomials in Lemma \ref{chebyshevidentities}.
\end{proof}

\begin{remark}
 If $q$ is specialized to $q=1$, then the classical Jones polynomials are independent of the knot $K$. In all examples we have checked, this is also true for the polynomials $J_n(K; q=1, t_1,t_2)$ (when they are normalized as in (\ref{equation_defof3varpolys})). 
\end{remark}

\section{5-parameter deformations}\label{sec_5paramdeformations}
In this section we discuss deformations of the nonsymmetric skein module $M$ of a knot to a family of modules over the DAHA $\H_{q,\ult}$ for all parameters $\ult \in (\C^*)^4$. Unfortunately, the deformations of $M$ that we produce in this section do not seem to be canonical, and in particular depend on a choice of $\C[X^{\pm 1}]$-splitting of the module $M$. If $M$ is the nonsymmetric skein module of the $(2,2p+1)$ torus knot, there is a natural choice of such a splitting (see Remark \ref{remark_splittingcanonical}), but for other knots this does not seem to be the case. For simplicity, in this section we will only discuss the trefoil.

We recall that the $\H_{q,\ult}$ is the subalgebra of $\End_\C(\C[X^{\pm 1}])$ generated by $X^{\pm 1}$ and the two operators
\begin{eqnarray*}
 T_0 &=& t_1 s\yy - \frac{q^2 \bar t_1 X^2  + q\bar t_2X}{1-q^2X^2}(1-s\yy)\\
 T_1 &=& t_3s + \frac{\bar t_3+\bar t_4 X}{1-X^2}(1-s)
\end{eqnarray*}
In Section \ref{sec_ccdeformations} we showed that the operator $T_0$ acts on the nonsymmetric skein module of several knot complements. However, the operator $T_1$ does not even act on the nonsymmetric skein module $V$ of the unknot because $V$ is the sign representation of $A_q\rtimes \Z_2$. More precisely, we recall that $V= \C[X^{\pm 1}]$ as a $\C[X^{\pm 1}]$-module, with the action of $\yy$ and $s$ determined by
\begin{equation}\label{equation_signrep5param}
 \yy\cdot 1 = -1,\quad s\cdot 1 = -1
\end{equation}
Then $T_1\cdot 1$ is a rational function in $X$, so in the sign representation, the action of $T_1$ on $\C(X)$ does not preserve the subspace $\C[X^{\pm 1}]$. 

To avoid this problem, we define a new operator $T^-_1 \in D_q$ that acts on $V$ and show that the operators $X$, $T_0$, and $T^-_1$ provide an action of $\H_{q,\ult}$ on the sign representation. In examples, the nonsymmetric skein module splits over $\C[X^{\pm 1}]\rtimes \Z_2$ as a sum of the standard and sign polynomial representations, and we can therefore use this splitting to define an action of $\H_{q,\ult}$ on the nonsymmetric skein module. We describe this action for the trefoil and then use this action to define 5-variable polynomials that specialize to the colored Jones polynomials of the trefoil.

\subsection{The Dunkl embedding for the sign representation}\label{sec_signdunklemb}
Let $V$ be the sign representation from (\ref{equation_signrep5param}) and let $V^\loc$ be the localization of $V$ at all nonzero polynomials in $X$. Define the following operator:
\begin{equation}\label{equation_t1minus}
 T_1^- := t_3^{-1}s + \frac{t_3-t_3^{-1} + (t_4-t_4^{-1})X}{1-X^2}(1+s)
\end{equation}

(Note the sign in the term $(1+s)$ is not a typo - it is required for $T_1^-$ to act on the sign representation $V$.) We recall from (\ref{ccdaharelations}) that $\H_{q,\ult}$ is generated by elements $V_i$ and $V_i^\vee$ (for $i=1,2$) with relations (\ref{ccdaharelations}).

\begin{lemma}\label{lemma_5paramsignrep}
 The operators $T_0$ and $T_1^-$ preserve the subspace $V \subset V^\loc$. Furthermore, the operators $X$, $T_0$, and $T_1^-$ generate a copy of $\H_{q,\ult}$ inside of $D_q$ via the map
 \begin{equation}\label{equation_5paramdunkl}
  T_0 \mapsto T_0,\quad T_1 \mapsto T_1^-,\quad T_0^\vee \mapsto qT_0^{-1}X,\quad T_1^\vee \mapsto X^{-1}(T_1^-)^{-1}
 \end{equation}

\end{lemma}
\begin{proof}
 The operator $s\yy$ acts in the same way on the trivial and sign representations of $D_q$, and this implies $T_0$ preserves the subspace $V \subset V^\loc$. Since the map $D_q \to \End_{\C}(\C(X))$ is injective for $q$ not a root of unity, this also shows that $T_0$ and $T_0^\vee$ satisfy the first two relations in (\ref{ccdaharelations}). Also, the final relation holds by definition.
 
 To show that the third and fourth relations hold, we first note that as a $\C[X+X^{-1}]$-module, we have a splitting $V = \C[X+X^{-1}]1\oplus \C[X+X^{-1}]X$. Furthermore, the operator $T_1^-$ commutes with $X+X^{-1} \in D_q$, which shows that it suffices to check that the left hand sides of the third and fourth relation annihilate the elements $1,X \in V$. These are straightforward computations which we omit.
\end{proof}

\subsection{5-parameter deformation for the trefoil}
We recall that the $\C[X^{\pm 1}]$-module structure of the nonsymmetric skein module $M$ of the trefoil is
\begin{equation}\label{equation_trefoilsplitting}
 M \cong \C[X^{\pm 1}]u \oplus \C[X^{\pm 1}]v
\end{equation}
To define the $A_q\rtimes \Z_2$-module structure of $M$, we first define two operators $S, P: M \to  M$ as in equation (\ref{equation_diagonalaction}). We can then write the action of the elements $\yy, s \in A_q\rtimes \Z_2$ in terms of matrices (with respect to the ordered basis $(u, v)$ of $M$):
\[
 s = \left[\begin{array}{cc}-1&0\\0&1\end{array}\right]S, \quad \yy = \left[\begin{array}{cc}-1& q^2X^{-1}-q^6X^{-5}\\0&q^6X^{-6}\end{array}\right]P
\]
\begin{remark}\label{remark_splittingcanonical}
 The element $u \in M$ generates a submodule $M' \subset M$, and we write $M''$ for the quotient, so that we have an exact sequence
 \[
  0 \to M' \to M \to  M'' \to 0
 \]
 If $q$ is not a root of unity, then $M'$ and $M''$ are simple $A_q\rtimes \Z_2$-modules, and since this extension is non-split, it is easy to see that $M'$ is the \emph{unique} nontrivial submodule of $M$ (see Lemma \ref{lemma_trefoilseq}). The module $M''$ has a distinguished generator $v''$ (the image of the empty link $v$), so a $\C[X^{\pm 1}]$ splitting of this exact sequence is determined by the image of $v''$ in $M$. The natural choice for this image is the empty link in $M$, which is $v$. In this sense, the choice of $\C[X^{\pm 1}]$-splitting given in (\ref{equation_trefoilsplitting}) is natural. However, in general the analogue of the quotient $M''$ has $\C[X^{\pm 1}]$-rank greater than 1, so a choice of $\C[X^{\pm 1}]$-splitting is not determined by a single element of $M$.
\end{remark}

We now define the operator $T_1^*:M \to M$ as follows:
\[
 T_1^* = \left[\begin{array}{cc} T_1^-&0\\0&T_1\end{array}\right]
\]

\begin{theorem}
 The operators $X$, $T_0$, and $T_1^*$ provide an action of $\H_{q,\ult}$ on $M$ via the map (\ref{equation_5paramdunkl}).
\end{theorem}
\begin{proof}
It was shown in Section \ref{sec_conj1fortorusknots} that the operator $T_0$ preserves the subspace $M \subset M^\loc$, and this implies $T_0$ and $qT_0^{-1}X$ satisfy the first two relations of (\ref{ccdaharelations}). The final relation holds by definition. Finally, since the splitting in (\ref{equation_trefoilsplitting}) is a splitting of $\C[X^{\pm 1}]\rtimes \Z_2$-modules, the operator $T_1^*$ preserves this splitting. Then the third and fourth relations of (\ref{ccdaharelations}) follow from Lemma \ref{lemma_5paramsignrep} and from the fact that $T_1$ satisfies these relations in the standard polynomial representation.
\end{proof}

\begin{remark}
This action can be used to produce 5-variable polynomials $J_n(q,t_1,t_2,t_3,t_4)$ that restrict to the colored Jones polynomials $J_n(q)$ using a similar formula to (\ref{equation_defof3varpolys}) - however, the polynomials this produces are quite lengthy, so for the sake of brevity we omit them.
\end{remark}

\section{Appendix: Computer calculations of multi-variable Jones polynomials}

We now include computer computations
of 3-variable polynomials normalized as in (\ref{equation_defof3varpolys}). To shorten notation, we write $t := t_1$ and $v = t_2 - t_2^{-1}$. (We remind the reader of the normaliqation conventions for the colored Jones polynomials, see Remark \ref{remark_signconvention}. In particular, the $q$ of the KnotAtlas is our $q^2$.)


\subsection{The trefoil}
For the trefoil, we have the following polynomials:

\begin{polynomial}
 J_2 = v q  - t^{-1}  q^{2}+(    - t^{-1}    + 
t )q^{4}  -t q^{6}+  v q^{7}  -v q^{9}  - t^{-1}  q
^{10}+(   t^{-1}  - 
t )q^{12}  -v q^{15}+   t^{-1}  q^{18}
\end{polynomial}

\begin{polynomial}
  J_3 = v^2 q^{2}    - t^{-1}     v q^{3}+(    -1   + 
 t^{-2}  + t^2 )q^{4}+(      -2     t^{-1}   v   +   t  
v   )q^{5}+(    -1   + 
 t^{-2}  )q^{6}+(      - t^{-1}     v   +   t  
v   )q^{7}+(   t^{-2}  + v^2 )q^{8}     - t^{-1}vq^{9}+(    -1   +  t^{-2}  )q^{10}    - t^{-1}     
vq^{11}   +t^{-2}  q^{12}    -t    v q^{13}+(  1 - 
t^2 )q^{14}    -t    v q^{15}+(  1 - 
v^2 )q^{16}+(     t^{-1}   v   -   t  v   )q^{17}+(  1 
- t^2 - v^2 )q^{18}+(    2   t^{-1}   v   -   t  
v   )q^{19}+(  2 -  t^{-2}  )q^{20}+(    2   t^{-1}   v   
-   t  v   )q^{21}+(    - t^{-2}    + t^2 - 
v^2 )q^{22}+(     t^{-1}   v   -   t  v   )q^{23}+(  2 
-  t^{-2}  - t^2 )q^{24}+(     t^{-1}   v   +   t  
v   )q^{25}+(    - t^{-2}    + t^2 - 
v^2 )q^{26}+   t^{-1}   v q^{27}  + t^{-1}   
v q^{29}+(  1 -  t^{-2}  + v^2 )q^{30}+(     t^{-1}   v   
+   t  v   )q^{31}+(    -2   + t^2 - 
v^2 )q^{32}    - t^{-1}     v q^{33}+(  1 -  t^{-2}  + 
v^2 )q^{34}+  t  v q^{35}+  t  v q^{37}+(    -1   + 
 t^{-2}  )q^{38}    - t^{-1}     
v q^{39}+(    - t^{-2}    + v^2 )q^{40}+(    -1   + 
 t^{-2}  )q^{42}    - t^{-1}     
v q^{43}    - t^{-1}     v q^{45}+   t^{-2}  q^{48}
\end{polynomial}

We now set $v = 0$ and give the first coefficients of the Taylor expansion of $J_2$ and $J_3$ around $t-1$. Since the 3-variable polynomials specialiqe to the classical colored Jones polynomials when $t=1$, the coefficient of $(t-1)^0$ is the classical colored Jones polynomial. For the $t-1$ expansion of $J_2$, we have the coefficients

\begin{polynomial}
 -q^2-q^6-q^{10}+q^{18} \\
 q^2+2 q^4-q^6+q^{10}-2 q^{12}-q^{18} \\
 -q^2-q^4-q^{10}+q^{12}+q^{18}
\end{polynomial}

For the expansion of $J_3$, we have the coefficients
\begin{polynomial} q^4+q^8+q^{12}+q^{16}+q^{20}-q^{32}-q^{36}-q^{40}+q^{48}, \\
 -2 \Big(q^6+q^8+q^{10}+q^{12}+q^{14}+q^{18}-q^{20}-2 q^{22}-2 q^{26}-q^{30}-q^{32}-q^{34}+q^{38}-q^{40}+q^{42}+q^{48}\Big), \\
 4 q^4+3 q^6+3 q^8+3 q^{10}+3 q^{12}-q^{14}-q^{18}-3 q^{20}-2 q^{22}-4 q^{24}-2 q^{26}-3 q^{30}+q^{32}-3 q^{34}+3 q^{38}-3 q^{40}+3 q^{42}+3 q^{48}
 \end{polynomial}
We also give the first two coefficients in the expansion of $J_4$:
 \begin{polynomial}
  -q^6-q^{10}-q^{14}-q^{18}-q^{22}-q^{26}-q^{30}+q^{46}+q^{50}+q^{54}+q^{58}+q^{62}-q^{74}-q^{78}-q^{82}+q^{90} \end{polynomial}\begin{polynomial}
 q^6+2 q^8-q^{10}+4 q^{12}+3 q^{14}+2 q^{16}+3 q^{18}+3 q^{22}+4 q^{24}+q^{26}+2 q^{28}-q^{30}-2 q^{34}-4 q^{36}-2 q^{38}-8 q^{40}-4 q^{42}-6 q^{44}-3
q^{46}-6 q^{48}-3 q^{50}+q^{54}+q^{58}-q^{62}+6 q^{64}+6 q^{68}+2 q^{72}+3 q^{74}+q^{78}-2 q^{80}+3 q^{82}-2 q^{84}-3 q^{90}
 \end{polynomial}

\subsection{The $(5,2)$ torus knot}
For the $(5,2)$ torus knot, we have

\begin{polynomial}
 J_2 = v q+(    - t^{-1}    + t )q^{2}+(    - t^{-1}    
+ 
t )q^{4}+  v q^{5}  -t q^{6}+  v q^{7}+(    - t^{-1}  
  + 
t )q^{8}  -v q^{9}  - t^{-1}  q^{10}+  v q^{11}+(   t^{-1}  - 
t )q^{12}  - t^{-1}  q^{14}  -v q^{15}+(   t^{-1}  - 
t )q^{18}  -v q^{21}+(   t^{-1}  - 
t )q^{24}  -v q^{27}+   t^{-1}  q^{30}
\end{polynomial}

For the $(t-1)$ expansion of $J_2$, we have the coefficients
\begin{polynomial}
-q^6-q^{10}-q^{14}+q^{30} \\
 2 q^2+2 q^4-q^6+2 q^8+q^{10}-2 q^{12}+q^{14}-2 q^{18}-2 q^{24}-q^{30} \\
 -q^2-q^4-q^8-q^{10}+q^{12}-q^{14}+q^{18}+q^{24}+q^{30}
\end{polynomial}

The first two terms of the $(t-1)$ expansion of $J_3$ are
\begin{polynomial}
q^{12}+q^{16}+q^{20}+q^{24}+q^{28}-q^{56}-q^{60}-q^{64}+q^{80} \end{polynomial}
\begin{polynomial}
 -2 \Big(q^8+q^{10}+q^{12}+2 q^{14}+q^{16}+2 q^{18}+q^{20}+q^{22}+q^{28}-q^{30}-2 q^{34}-2 q^{36}-q^{40}-q^{42}-q^{44}- \\ 
 2 q^{46}-q^{48}-q^{52}-2 q^{54}-q^{58}+q^{62}-q^{64}+q^{68}+q^{70}+q^{74}+q^{80}\Big)
\end{polynomial}

\subsection{The figure eight}

Since the figure eight knot is isotopic to its mirror image, Proposition \ref{prop_mirror} shows that $J_n(q,t,v) = J_n(q^{-1},t^{-1},-v)$ (in the current notation for parameters). Explicitly, we have

\begin{polynomial}
J_2 =   -t q^{-10}  -v q^{-7}+(   t^{-1}  - 
t )q^{-4}  -v q^{-1}+  v q+(    - t^{-1}    + 
t )q^{4}+  v q^{7}  - t^{-1}  q^{10}
\end{polynomial}
\begin{polynomial}
 J_3 =   t^2 q^{-28}+  t  v q^{-25}+  t  
v q^{-23}+(    -1   + t^2 )q^{-22}+(    -t^2   + 
v^2 )q^{-20}+  t  v q^{-19}+(    -1   + 
t^2 )q^{-18}+   - t^{-1}     v q^{-17}+(    -1   + 
t^2 )q^{-16}    - t^{-1}     v q^{-15}+(  1 - t^2 + 
v^2 )q^{-14}+  2  t  v q^{-13}+(    -2   +  t^{-2}  + t^2 
- v^2 )q^{-12}  - t^{-1}     
v q^{-11}+  v^2 q^{-10}    -t    v q^{-9}+(  1 - t^2 
+ v^2 )q^{-8}+(    -1   +  t^{-2}  -   2  
v^2   )q^{-6}+(      - t^{-1}     v   -   t  
v   )q^{-5}+(  2 -  t^{-2}  )q^{-4}+(    2   t^{-1}   v   
-   2  t  v   )q^{-3}+(  2 -   2  t^2   )q^{-2}+(    2  
 t^{-1}   v   -   t  v   )q^{-1}+(  1 -   2  
v^2   )+(     t^{-1}   v   -   2  t  
v   )q+(  2 -   2   t^{-2}    )q^{2}+(    2  
 t^{-1}   v   -   2  t  v   )q^{3}+(  2 - 
t^2 )q^{4}+(     t^{-1}   v   +   t  
v   )q^{5}+(    -1   + t^2 -   2  v^2   )q^{6}+(  1 - 
 t^{-2}  + v^2 )q^{8}+   t^{-1}   
v q^{9}+  v^2 q^{10}+  t  v q^{11}+(    -2   + 
 t^{-2}  + t^2 - v^2 )q^{12}    -2     t^{-1}   
v q^{13}+(  1 -  t^{-2}  + v^2 )q^{14}+  t  
v q^{15}+(    -1   +  t^{-2}  )q^{16}+  t  
v q^{17}+(    -1   +  t^{-2}  )q^{18}    - t^{-1}     
v q^{19}+(    - t^{-2}    + v^2 )q^{20}+(    -1   + 
 t^{-2}  )q^{22}    - t^{-1}     
v q^{23}    - t^{-1}     v q^{25}+  t^{-2}  q^{28}
\end{polynomial}

For the $t-1$ expansion of $J_2$ we have the following coefficients:
\begin{polynomial}
 -q^{10}-q^{-10} \\
 2 q^4+q^{10}-q^{-10}-2 q^{-4} \\
 -q^4-q^{10}+q^{-4}
\end{polynomial}

For the $t-1$ expansion of $J_3$ we have the following coefficients:
\begin{polynomial}
 1+q^4-q^{20}+q^{28}+q^{-28}-q^{-20}+q^{-4}, \\
 -2 \Big(2 q^2-q^4+q^6+q^8+q^{14}-q^{16}-q^{18}+q^{20}-q^{22}-q^{28}+q^{-28}+q^{-22}-q^{-20}+q^{-18}+q^{-16}-q^{-14}-q^{-8}-q^{-6}+q^{-4}-2 q^{-2}\Big) \\
 -6 q^2-q^4+q^6-3 q^8+4 q^{12}-3 q^{14}+3 q^{16}+3 q^{18}-3 q^{20}+3 q^{22}+3 q^{28}+q^{-28}+q^{-22}-q^{-20}+q^{-18}+q^{-16}-q^{-14}+4 q^{-12}-q^{-8}+3
q^{-6}-3 q^{-4}-2 q^{-2}
\end{polynomial}

The first three terms of the expansion of $J_4$ are
\begin{polynomial}
  -q^{18}-q^{22}-q^{34}+q^{42}+q^{46}-q^{54}-q^{-54}+q^{-46}+q^{-42}-q^{-34}-q^{-22}-q^{-18} 
  \end{polynomial}\begin{polynomial}
 -2 q^2+6 q^4+4 q^8+2 q^{10}+4 q^{12}+4 q^{14}+q^{18}-q^{22}+2 q^{24}-2 q^{26}-2 q^{28}-4 q^{32}+3 q^{34}-2 q^{36}-q^{42}+2 q^{44}-3 q^{46}+2 q^{48}+3
q^{54}-3 q^{-54}-2 q^{-48}+3 q^{-46}-2 q^{-44}+q^{-42}+2 q^{-36}-3 q^{-34}+4 q^{-32}+2 q^{-28}+2 q^{-26}-2 q^{-24}+q^{-22}-q^{-18}-4 q^{-14}-4 q^{-12}-2
q^{-10}-4 q^{-8}-6 q^{-4}+2 q^{-2} \end{polynomial}
\begin{polynomial}
 -8-7 q^2-7 q^4-4 q^6-2 q^8-q^{10}+2 q^{12}+2 q^{14}+4 q^{16}+3 q^{18}+4 q^{20}+4 q^{22}-q^{24}+9 q^{26}-3 q^{28}+6 q^{32}-10 q^{34}+5 q^{36}-4 q^{38}+q^{42}-5
q^{44}+6 q^{46}-5 q^{48}-6 q^{54}-3 q^{-54}-3 q^{-48}+3 q^{-46}-3 q^{-44}-4 q^{-38}+3 q^{-36}-7 q^{-34}+2 q^{-32}-5 q^{-28}+7 q^{-26}+q^{-24}+3 q^{-22}+4
q^{-20}+4 q^{-18}+4 q^{-16}+6 q^{-14}+6 q^{-12}+q^{-10}+2 q^{-8}-4 q^{-6}-q^{-4}-9 q^{-2}
\end{polynomial}

\bibliography{secondbibtexfile_3_6}{}

\newcommand{\etalchar}[1]{$^{#1}$}
\providecommand{\bysame}{\leavevmode\hbox to3em{\hrulefill}\thinspace}
\providecommand{\MR}{\relax\ifhmode\unskip\space\fi MR }
\providecommand{\MRhref}[2]{%
  \href{http://www.ams.org/mathscinet-getitem?mr=#1}{#2}
}
\providecommand{\href}[2]{#2}
\begin{thebibliography}{CCG{\etalchar{+}}94}

\bibitem[Ago]{35687}
Ian Agol, \emph{Complete knot invariant?}, MathOverflow,
  URL:http://mathoverflow.net/q/35687 (version: 2010-08-16).

\bibitem[AW85]{AW85}
Richard Askey and James Wilson, \emph{Some basic hypergeometric orthogonal
  polynomials that generalize {J}acobi polynomials}, Mem. Amer. Math. Soc.
  \textbf{54} (1985), no.~319, iv+55. \MR{783216 (87a:05023)}

\bibitem[Bar99]{Bar99}
John~W. Barrett, \emph{Skein spaces and spin structures}, Math. Proc. Cambridge
  Philos. Soc. \textbf{126} (1999), no.~2, 267--275. \MR{1670233 (99k:57006)}

\bibitem[BC11]{BC11}
Yuri Berest and Oleg Chalykh, \emph{Quasi-invariants of complex reflection
  groups}, Compos. Math. \textbf{147} (2011), no.~3, 965--1002. \MR{2801407}

\bibitem[BH95]{BH95}
G.~W. Brumfiel and H.~M. Hilden, \emph{{${\rm SL}(2)$} representations of
  finitely presented groups}, Contemporary Mathematics, vol. 187, American
  Mathematical Society, Providence, RI, 1995. \MR{1339764 (96g:20004)}

\bibitem[BLF05]{BL05}
Doug Bullock and Walter Lo~Faro, \emph{The {K}auffman bracket skein module of a
  twist knot exterior}, Algebr. Geom. Topol. \textbf{5} (2005), 107--118
  (electronic). \MR{2135547 (2006a:57012)}

\bibitem[BP00]{BP00}
Doug Bullock and J{\'o}zef~H. Przytycki, \emph{Multiplicative structure of
  {K}auffman bracket skein module quantizations}, Proc. Amer. Math. Soc.
  \textbf{128} (2000), no.~3, 923--931. \MR{1625701 (2000e:57007)}

\bibitem[BS12]{BS12}
Yuri Berest and Peter Samuelson, \emph{Dunkl operators and quasi-invariants of
  complex reflection groups}, Mathematical aspects of quantization, Contemp.
  Math., vol. 583, Amer. Math. Soc., Providence, RI, 2012, pp.~1--23.
  \MR{3013091}

\bibitem[Bul97]{Bul97}
Doug Bullock, \emph{Rings of {${\rm SL}\sb 2({\bf C})$}-characters and the
  {K}auffman bracket skein module}, Comment. Math. Helv. \textbf{72} (1997),
  no.~4, 521--542. \MR{1600138 (98k:57008)}

\bibitem[BZ03]{BZ03}
Gerhard Burde and Heiner Zieschang, \emph{Knots}, second ed., de Gruyter
  Studies in Mathematics, vol.~5, Walter de Gruyter \& Co., Berlin, 2003.
  \MR{1959408 (2003m:57005)}

\bibitem[CCG{\etalchar{+}}94]{CCG94}
D.~Cooper, M.~Culler, H.~Gillet, D.~D. Long, and P.~B. Shalen, \emph{Plane
  curves associated to character varieties of {$3$}-manifolds}, Invent. Math.
  \textbf{118} (1994), no.~1, 47--84. \MR{1288467 (95g:57029)}

\bibitem[Che95]{Che95}
Ivan Cherednik, \emph{Double affine {H}ecke algebras and {M}acdonald's
  conjectures}, Ann. of Math. (2) \textbf{141} (1995), no.~1, 191--216.
  \MR{1314036 (96m:33010)}

\bibitem[Che05]{Che05}
\bysame, \emph{Double affine {H}ecke algebras}, London Mathematical Society
  Lecture Note Series, vol. 319, Cambridge University Press, Cambridge, 2005.
  \MR{2133033 (2007e:32012)}

\bibitem[{Che}11]{Che11}
I.~{Cherednik}, \emph{Jones polynomials of torus knots via daha}, ArXiv
  1111.6195 (2011).

\bibitem[CKM12]{CKM12}
S.~{Cautis}, J.~{Kamnitzer}, and S.~{Morrison}, \emph{{Webs and quantum skew
  Howe duality}}, ArXiv e-prints (2012).

\bibitem[CLPZ]{CLPZ14}
Q.~{Chen}, K.~{Liu}, P.~{Peng}, and S.~{Zhu}, \emph{Congruent skein relations
  for colored {H}omfly-pt invariants and colored {J}ones polynomials}, ArXiv
  1402.3571 (2014).

\bibitem[CM11]{CM11}
L.~{Charles} and J.~{Marche}, \emph{{Knot state asymptotics I, AJ Conjecture
  and abelian representations}}, ArXiv 1107.1645 (2011).

\bibitem[EOR07]{EOR07}
Pavel Etingof, Alexei Oblomkov, and Eric Rains, \emph{Generalized double affine
  {H}ecke algebras of rank 1 and quantized del {P}ezzo surfaces}, Adv. Math.
  \textbf{212} (2007), no.~2, 749--796. \MR{2329319 (2008h:20006)}

\bibitem[FG00]{FG00}
Charles Frohman and R{\u{a}}zvan Gelca, \emph{Skein modules and the
  noncommutative torus}, Trans. Amer. Math. Soc. \textbf{352} (2000), no.~10,
  4877--4888. \MR{MR1675190 (2001b:57014)}

\bibitem[Fuk05]{Fuk05}
Shinji Fukuhara, \emph{Explicit formulae for two-bridge knot polynomials}, J.
  Aust. Math. Soc. \textbf{78} (2005), no.~2, 149--166. \MR{2141874
  (2006f:57013)}

\bibitem[GL05]{GL05}
Stavros Garoufalidis and Thang T.~Q. L{\^e}, \emph{The colored {J}ones function
  is {$q$}-holonomic}, Geom. Topol. \textbf{9} (2005), 1253--1293 (electronic).
  \MR{2174266 (2006j:57029)}


\bibitem[Gel02]{Gel02}
R{\u{a}}zvan Gelca, \emph{Non-commutative trigonometry and the {$A$}-polynomial
  of the trefoil knot}, Math. Proc. Cambridge Philos. Soc. \textbf{133} (2002),
  no.~2, 311--323. \MR{1912404 (2004c:57021)}

\bibitem[GS03]{GS03}
R{\u{a}}zvan Gelca and Jeremy Sain, \emph{The noncommutative {A}-ideal of a
  {$(2,2p+1)$}-torus knot determines its {J}ones polynomial}, J. Knot Theory
  Ramifications \textbf{12} (2003), no.~2, 187--201. \MR{1967240 (2004d:57015)}

\bibitem[GS04]{GS04}
\bysame, \emph{The computation of the non-commutative generalization of the
  {$A$}-polynomial of the figure-eight knot}, J. Knot Theory Ramifications
  \textbf{13} (2004), no.~6, 785--808. \MR{2088746 (2005f:57020)}


\bibitem[GL89]{GL89}
C.~McA. Gordon and J.~Luecke, \emph{Knots are determined by their complements},
  J. Amer. Math. Soc. \textbf{2} (1989), no.~2, 371--415. \MR{965210
  (90a:57006a)}


\bibitem[Hab08]{Hab08}
Kazuo Habiro, \emph{A unified {W}itten-{R}eshetikhin-{T}uraev invariant for
  integral homology spheres}, Invent. Math. \textbf{171} (2008), no.~1, 1--81.
  \MR{2358055 (2009b:57020)}

\bibitem[KM91]{KM91}
Robion Kirby and Paul Melvin, \emph{The {$3$}-manifold invariants of {W}itten
  and {R}eshetikhin-{T}uraev for {${\rm sl}(2,{\bf C})$}}, Invent. Math.
  \textbf{105} (1991), no.~3, 473--545. \MR{1117149 (92e:57011)}

\bibitem[Koo08]{Koo08}
Tom~H. Koornwinder, \emph{Zhedanov's algebra {$\rm AW(3)$} and the double
  affine {H}ecke algebra in the rank one case. {II}. {T}he spherical
  subalgebra}, SIGMA Symmetry Integrability Geom. Methods Appl. \textbf{4}
  (2008), Paper 052, 17. \MR{2425640 (2010e:33028)}

\bibitem[L{\^e}06]{Le06}
Thang T.~Q. L{\^e}, \emph{The colored {J}ones polynomial and the
  {$A$}-polynomial of knots}, Adv. Math. \textbf{207} (2006), no.~2, 782--804.
  \MR{2271986 (2007k:57021)}

\bibitem[LM85]{LM85}
Alexander Lubotzky and Andy~R. Magid, \emph{Varieties of representations of
  finitely generated groups}, Mem. Amer. Math. Soc. \textbf{58} (1985),
  no.~336, xi+117. \MR{818915 (87c:20021)}

\bibitem[Lus89]{Lus89}
George Lusztig, \emph{Affine {H}ecke algebras and their graded version}, J.
  Amer. Math. Soc. \textbf{2} (1989), no.~3, 599--635. \MR{991016 (90e:16049)}

\bibitem[Mac03]{Mac03}
I.~G. Macdonald, \emph{Affine {H}ecke algebras and orthogonal polynomials},
  Cambridge Tracts in Mathematics, vol. 157, Cambridge University Press,
  Cambridge, 2003. \MR{1976581 (2005b:33021)}

\bibitem[Min82]{Min82}
Jerome Minkus, \emph{The branched cyclic coverings of {$2$} bridge knots and
  links}, Mem. Amer. Math. Soc. \textbf{35} (1982), no.~255, iv+68. \MR{643587
  (83g:57004)}

\bibitem[NS04]{NS04}
Masatoshi Noumi and Jasper~V. Stokman, \emph{Askey-{W}ilson polynomials: an
  affine {H}ecke algebra approach}, Laredo {L}ectures on {O}rthogonal
  {P}olynomials and {S}pecial {F}unctions, Adv. Theory Spec. Funct. Orthogonal
  Polynomials, Nova Sci. Publ., Hauppauge, NY, 2004, pp.~111--144. \MR{2085854
  (2005h:42057)}

\bibitem[Obl04]{Obl04}
Alexei Oblomkov, \emph{Double affine {H}ecke algebras of rank 1 and affine
  cubic surfaces}, Int. Math. Res. Not. (2004), no.~18, 877--912. \MR{2037756
  (2005j:20005)}

\bibitem[Prz91]{Prz91}
J{\'o}zef~H. Przytycki, \emph{Skein modules of {$3$}-manifolds}, Bull. Polish
  Acad. Sci. Math. \textbf{39} (1991), no.~1-2, 91--100. \MR{1194712
  (94g:57011)}

\bibitem[PS00]{PS00}
J{\'o}zef~H. Przytycki and Adam~S. Sikora, \emph{On skein algebras and {${\rm
  Sl}\sb 2({\bf C})$}-character varieties}, Topology \textbf{39} (2000), no.~1,
  115--148. \MR{1710996 (2000g:57026)}

\bibitem[Ric79]{Ric79}
R.~W. Richardson, \emph{Commuting varieties of semisimple {L}ie algebras and
  algebraic groups}, Compositio Math. \textbf{38} (1979), no.~3, 311--327.
  \MR{535074 (80c:17009)}

\bibitem[RT90]{RT90}
N.~Yu. Reshetikhin and V.~G. Turaev, \emph{Ribbon graphs and their invariants
  derived from quantum groups}, Comm. Math. Phys. \textbf{127} (1990), no.~1,
  1--26. \MR{1036112 (91c:57016)}

\bibitem[Sah99]{Sah99}
Siddhartha Sahi, \emph{Nonsymmetric {K}oornwinder polynomials and duality},
  Ann. of Math. (2) \textbf{150} (1999), no.~1, 267--282. \MR{1715325
  (2002b:33018)}

\bibitem[Sam12]{Sam12}
Peter Samuelson, \emph{Kauffman bracket skein modules and the quantum torus},
  Ph.D. thesis, Cornell University, 2012.

\bibitem[Sam14]{Sam14}
\bysame, \emph{A topological construction of {C}herednik's $sl_2$ torus knot
  polynomials}, in preparation (2014).

\bibitem[Sco83]{Sco83}
Peter Scott, \emph{The geometries of {$3$}-manifolds}, Bull. London Math. Soc.
  \textbf{15} (1983), no.~5, 401--487. \MR{705527 (84m:57009)}

\bibitem[Sik05]{Sik05}
Adam~S. Sikora, \emph{Skein theory for {${\rm SU}(n)$}-quantum invariants},
  Algebr. Geom. Topol. \textbf{5} (2005), 865--897 (electronic). \MR{2171796
  (2006j:57033)}

\bibitem[Sto03]{Sto03}
Jasper~V. Stokman, \emph{Difference {F}ourier transforms for nonreduced root
  systems}, Selecta Math. (N.S.) \textbf{9} (2003), no.~3, 409--494.
  \MR{2006574 (2004h:33040)}

\bibitem[SW07]{SW07}
Adam~S. Sikora and Bruce~W. Westbury, \emph{Confluence theory for graphs},
  Algebr. Geom. Topol. \textbf{7} (2007), 439--478. \MR{2308953 (2008f:57004)}

\bibitem[Tha01]{Tha01}
Michael Thaddeus, \emph{Mirror symmetry, {L}anglands duality, and commuting
  elements of {L}ie groups}, Internat. Math. Res. Notices (2001), no.~22,
  1169--1193. \MR{1862614 (2002h:14019)}

\bibitem[Wal68]{Wal68}
Friedhelm Waldhausen, \emph{On irreducible {$3$}-manifolds which are
  sufficiently large}, Ann. of Math. (2) \textbf{87} (1968), 56--88.
  \MR{0224099 (36 \#7146)}

\end{thebibliography}
\bibliographystyle{amsalpha}

\end{document}